\documentclass[preprint,11pt]{article}
\usepackage{fullpage}
\usepackage{amssymb,amsfonts,amsmath,amsthm,amscd,dsfont,mathrsfs,bm}
\usepackage{graphicx,float,psfrag,epsfig,amssymb}
\usepackage[usenames,dvipsnames,svgnames,table]{xcolor}
\definecolor{darkgreen}{rgb}{0.0,0,0.9}
\usepackage[pagebackref,letterpaper=true,colorlinks=true,pdfpagemode=none,citecolor=OliveGreen,linkcolor=BrickRed,urlcolor=BrickRed,pdfstartview=FitH]{hyperref}
\usepackage{wrapfig}
\usepackage{relsize}
\usepackage{color}
\usepackage{pict2e}
\usepackage[tight]{subfigure}
\usepackage{algorithm}
\usepackage[noend]{algorithmic}
\usepackage{caption}
\usepackage{nameref}
\usepackage{makecell}
\usepackage[font={small}]{caption} 

\DeclareMathAlphabet{\mathpzc}{OT1}{pzc}{m}{it}

\newtheorem{propo}{Proposition}[section]
\newtheorem{lemma}[propo]{Lemma}
\newtheorem{definition}[propo]{Definition}
\newtheorem{coro}[propo]{Corollary}
\newtheorem{thm}[propo]{Theorem}

\theoremstyle{definition}
\newtheorem{remark}[propo]{Remark}

\def\cone{\mathcal{C}}

\def\hS{\widehat{S}}

\def\baF{\overline{F}}
\def\baX{\overline{X}}
\def\baS{\overline{S}}

\def\Cov{{\rm Cov}}

\def\tx{\widetilde{x}}
\def\bax{\overline{x}}
\def\bay{\overline{y}}

\def\cF{{\cal F}}
\def\cA{{\cal A}}

\def\cH{{\cal H}}
\def\cC{{\cal C}}
\def\cG{{\cal G}}
\def\cE{{\cal E}}

\def\event{\mathcal{E}}

\def\tSigma{\tilde{\Sigma}}
\def\PSD{{\sf PSD}}
\def\ctheta{{\check{\theta}}}
\def\csigma{{\check{\sigma}}}
\def\cgamma{{\check{\gamma}}}
\def\tdelta{\tilde{\delta}}
\def\cG{\mathcal{G}}
\def\xopta{x_{{\rm opt},1}}
\def\xoptb{x_{{\rm opt},2}}

\def\reals{{\mathbb R}}

\def\eps{{\varepsilon}}
\def\prob{{\mathbb P}}
\def\E{{\mathbb E}}
\def\Var{{\rm Var}}

\def\tC{\widetilde{C}}

\def\cB{{\cal B}}

\def\L0{{L_i}}

\def\de{{\rm d}}
\def\<{\langle}
\def\>{\rangle}
\def\diag{{\rm diag}}

\def\trho{\tilde{\rho}}

\def\Risk{{\sf R}}
\def\hRisk{\widehat{\sf R}}

\def\hth{\widehat{\theta}}
\def\oth{\overline{\theta}}
\def\otau{\overline{\tau}}
\def\dth{\widehat{\theta}^{{\rm d}}}
\def\sth{\widehat{\theta}^{{\rm split}}}
\def\hSigma{\widehat{\Sigma}}
\def\hOmega{\widehat{\Omega}}
\def\hsigma{\widehat{\sigma}}

\def\supp{{\rm supp}}

\def\F{{\sf F}}
\def\ind{{\mathbb I}}

\def \Tr{{\rm Trace}}
\def\F{{\sf F}}
\def\normal{{\sf N}}
\def\Lth{\widehat{\theta}^{\mbox{\tiny \rm Lasso}} }

\def\sT{{\sf T}}

\def\id{{\rm I}}

\def\hgamma{\hat{\gamma}}
\def\htau{\hat{\tau}}

\def\proj{{\rm P}}

\def\sign{{\rm sign}}

\def\event{\mathcal{E}}

\def\v*{v_i}
\def\T*{T_i}

\def\u*{u_i}
\def\F*{F_i}

\definecolor{olivegreen}{rgb}{0,0.6,0.4}

\def\cH{{\mathcal{H}}}
\def\cI{{\mathcal{I}}}

\def\ctB{\tilde{\mathcal{B}}}
\def\tlambda{\widetilde{\lambda}}
\def\bias{{\sf bias}}
\def\htheta{\widehat{\theta}}

\def\ty{\tilde{y}}
\def\tx{\tilde{x}}

\def\bw{{w}}

\def\Xnz{X_{\sim i}}

\def\pe{{\rm p}}
\def\cL{\mathcal{L}}

\def\th{{\theta}}
\def\hth{{\widehat{\theta}}}
\def\tth{{\theta^*}}

\def\thnz{{\theta_{\sim i}}}
\def\hthnz{\hth_{\sim i}}
\def\tthnz{{\theta^*_{\sim i}}}
\def\tthT{{\theta^*_{T}}}
\def\pthnz{\hth^{\pe}_{\sim i}}
\def\pthT{\hth^{\pe}_T}
\def\diag{{\rm diag}}

\def\event{\mathcal{E}}
\def\pproj{\proj^{\perp}}
\def\SD{{\sf SD}}
\def\SE{{\sf SE}}
\def\oS{\overline{S}}
\def\SURE{\mbox{\tiny {\rm SURE}}}

\newcommand{\ajcomment}[1]{}

\makeatletter
\newcommand{\labitem}[2]{%
\def\@itemlabel{\text{#1}}
\item
\def\@currentlabel{#1}\label{#2}}
\makeatother

\addtocontents{toc}{\protect\setcounter{tocdepth}{2}}

\title{Debiasing the Lasso:\\
 Optimal Sample Size for Gaussian Designs}

\author{Adel~Javanmard\footnote{Data Sciences and Operations Department, Marshall School of Business, University
of Southern California, Email: \url{ajavanma@marshall.usc.edu} }
             \;\; and\;\; 
Andrea~Montanari\footnote{Department of Electrical Engineering and Department of Statistics, Stanford University. Email: \url{montanar@stanford.edu}}
            }

\begin{document}
\maketitle

\begin{abstract}
Performing statistical inference in high-dimensional models is an
outstanding challenge. A major source of difficulty is the absence of
precise information on the distribution of high-dimensional
regularized estimators.

Here, we consider linear regression in the high-dimensional regime $p\gg
n$ and the Lasso estimator. In this context, we would like to perform
inference on a high-dimensional parameters vector $\theta^*\in\reals^p$.
Important progress has been achieved in computing confidence intervals 
and p-values for single coordinates $\theta^*_i$, $i\in \{1,\dots,p\}$. A
key role in these new inferential methods is played by a certain debiased (or
de-sparsified) estimator $\dth$ that is constructed from the Lasso
estimator. Earlier work establishes that, under suitable assumptions
on the design matrix, the coordinates of $\dth$ are asymptotically
Gaussian provided the true parameters vector $\theta^*$ is $s_0$-sparse
with $s_0 = o(\sqrt{n}/\log p )$. 

The condition $s_0 = o(\sqrt{n}/ \log p )$ is considerably stronger than
the one required for consistent estimation, namely  $s_0 = o(n/ \log p
)$. 
In this paper, we consider Gaussian designs with known or unknown population covariance.
When the covariance is known, we prove that the debiased estimator is asymptotically Gaussian under 
the nearly optimal condition   $s_0 = o(n/ (\log p)^2)$. Note that 
\emph{earlier work was limited to  $s_0 = o(\sqrt{n}/
\log p)$ even for perfectly known covariance.} 

 The same conclusion holds if the population covariance is unknown but can be
estimated sufficiently well, e.g. \emph{under the same sparsity conditions on the inverse covariance 
as assumed by earlier work}.
For intermediate regimes, we describe the trade-off between sparsity
in the coefficients $\theta^*$, and sparsity in the inverse covariance
 of the design. We further discuss several other applications of our results to high-dimensional inference.
 In particular, we propose a thresholded Lasso estimator that is minimax optimal up to a factor $1+o_n(1)$
for i.i.d. Gaussian designs. 
  
\end{abstract}

\newpage

\section{Introduction}

\subsection{Background}

Consider  random design model where we are given $n$ i.i.d. pairs $(y_1,x_1)$, $(y_2,x_2)$, $\cdots$,
$(y_n,x_n)$ with $y_i\in \reals$, and $x_i\in\reals^p$. The response
variable $y_i$ is a linear function of  $x_i$, contaminated by noise
$w_i$ independent of $x_i$
\begin{eqnarray}\label{eqn:regression}
y_i \,=\, \<\th^*,x_i\> + w_i\, ,\;\;\;\;\;\;\;\; w_i\sim
\normal(0,\sigma^2)\, .
\end{eqnarray}
Here $\theta^*\in\reals^p$ is a vector of parameters to be estimated and
$\<\,\cdot\,,\,\cdot\,\>$ is the standard scalar product. 

In matrix form,
letting  $y = (y_1,\dots,y_n)^\sT$ and denoting by $X$ the matrix with
rows $x_1^\sT$,$\cdots$, $x_n^\sT$ we have
\begin{eqnarray}\label{eq:NoisyModel}
y\, =\, X\,\th^*+ w\, ,\;\;\;\;\;\;\;\; w\sim
\normal(0,\sigma^2 \id_{n\times n})\, .
\end{eqnarray}

We are interested in the high-dimensional regime wherein
the number of parameters $p$ exceeds the sample size $n$.
Over the last 20 years, impressive progress has been made in
developing and understanding highly effective estimators in this
regime \cite{Dantzig,BickelEtAl,buhlmann2011statistics}. A prominent
approach is the Lasso
\cite{Tibs96,BP95} defined through the
following convex optimization problem
\begin{eqnarray}\label{eq:Lasso}
\Lth(y,X;\lambda)\equiv \arg\max_{\theta\in\reals^p}
\left\{\frac{1}{2n}\|y-X\theta\|_2^2+\lambda
\|\theta\|_1\right\}\,.
\end{eqnarray}
(We will omit the arguments of $\Lth(y,X;\lambda)$ whenever clear from
the context.)

A far less understood question is how to perform statistical inference
in the high-dimensional setting, for instance computing confidence
intervals and p-values for quantities of interest.
Progress  in this direction was achieved only over the last couple of
years.
In particular, several papers
\cite{BuhlmannSignificance,zhang2014confidence,javanmard2013hypothesis,van2014asymptotically,javanmard2014confidence} 
develop methods to compute
confidence intervals for single coordinates of the parameters vector
$\theta^*$.
More precisely, these methods compute intervals $J_i(\alpha)$ 
depending on $y,X$, of nearly minimal size,  with the coverage guarantee 
\begin{align}
\prob\big(\theta^*_{i}\in J_i(\alpha)\big) \ge 1-\alpha - o_n(1)\,.
\end{align}
The $o_n(1)$ term is explicitly characterized, and 
vanishes along sequence of instances of increasing dimensions under
suitable condition on the design matrix $X$.

The fundamental idea developed in \cite{zhang2014confidence,javanmard2013hypothesis,van2014asymptotically,javanmard2014confidence} is to construct a debiased 
(or de-sparsified) estimator that takes the form
\begin{align}\label{eq:debiased}
\dth = \Lth + \frac{1}{n} M X^\sT (y- X\Lth)\,,
\end{align}
where $M\in\reals^{p\times p}$ is a matrix that is a function of $X$, but not
of $y$. While the construction of $M$ varies across different
papers, the basic intuition is that $M$ should be a good estimate of
the precision matrix $\Omega = \Sigma^{-1}$, where $\Sigma =
\E\{x_1x_1^{\sT}\}$ is the population covariance. 

Assume $\theta^*$ is $s_0$-sparse, i.e. it has only $s_0$ non-zero
entries. The key result that allows the construction of confidence
intervals in \cite{zhang2014confidence,van2014asymptotically,javanmard2014confidence} is  the following (holding under suitable conditions on the design matrix). 
If $M$ is `sufficiently
close' to $\Omega$, and the sparsity level is
\begin{align}
s_0 \ll \frac{\sqrt{n}}{\log p}\, ,\label{eq:Strong}
\end{align}
then $\dth_i$ is approximately Gaussian with mean $\theta^*_i$ and
variance of order $\sigma^2/n$. 

The condition (\ref{eq:Strong}) comes as a surprise, and is somewhat
disappointing. Indeed, consistent estimation using --for instance--
the Lasso
can be achieved under the much weaker condition $s_0\ll n/\log p$. 
More specifically, in this regime, with high probability \cite{Dantzig,zhang2008sparsity,BickelEtAl,ye2010rate,buhlmann2011statistics}
\begin{align}
\big\|\Lth-\theta^*\|_2^2 \le \frac{Cs_0\sigma^2}{n}\, \log p\, .
\end{align}
This naturally leads to the following question:
\begin{quote}
\emph{Does the debiased estimator have a Gaussian limit under the
  weaker
condition $s_0\ll n/\log p$?}
\end{quote}

Let us emphasize that the key technical challenge here does not lie in
the fact that $M$ is not a good estimate of the precision matrix
$\Omega$. Of course, if $M$ is not close to $\Omega$, then $\dth$ will
not have a Gaussian limit. 
However \emph{earlier proofs \cite{zhang2014confidence,van2014asymptotically,javanmard2014confidence} cannot establish the Gaussian
  limit for $s_0\gtrsim \sqrt{n}/\log p$, even if $\Omega$ is known
  and we set $M=\Omega$.}
Even the idealized case where the columns of $X$ are known to be
independent and identically distributed (i.e. $\Omega=\id$) is only understood in the asymptotic
limit $s_0,n,p\to\infty$ with $s_0/p$, $n/p$ having constant limits in $(0,1)$ \cite{javanmard2013hypothesis}. 

In order to describe the challenge, let us set $M=\Omega$, and recall the common step of the
proofs in \cite{zhang2014confidence,van2014asymptotically,javanmard2014confidence}. Using the definitions (\ref{eq:NoisyModel}),
(\ref{eq:debiased}), we
get
\begin{align}
\begin{split}
\sqrt{n}(\dth-\th^*)& = \sqrt{n}(\Lth-\th^*) +\frac{1}{\sqrt{n}} \Omega X^\sT  (X\th^*+w- X\Lth)\\
& = \frac{1}{\sqrt{n}} \Omega X^\sT w + \sqrt{n}(\Omega\hSigma-\id)(\th^*-\Lth) \, ,\label{eq:noisebias}
\end{split}
\end{align}
where $\hSigma = X^{\sT}X/n\in\reals^{p\times p}$ is the empirical design covariance.
Since $w\sim\normal(0,\sigma^2\id_n)$, it is easy to see that  vector $\Omega X^\sT w/\sqrt{n}$
has Gaussian entries of variance of order one. In order for
$\dth$ to be approximately Gaussian, we need the second  term  (which
can be interpreted as a bias) to vanish.  Earlier papers \cite{zhang2014confidence,van2014asymptotically,javanmard2014confidence}
address this by a simple $\ell_1$-$\ell_{\infty}$ bound.
Namely (denoting by $|Q|_{\infty}$ the maximum absolute value of any
entry of matrix $Q$):
\begin{align}
\begin{split}
\Big\|\sqrt{n}(\Omega\hSigma-\id)(\th^*-\Lth) \Big\|_{\infty} & \le
                                                                \sqrt{n}|\Omega\hSigma-\id|_{\infty}\|\th^*-\Lth\|_1\\
&\le \sqrt{n} \times C\sqrt{\frac{\log p}{n}} \times Cs_0\sigma
  \sqrt{\frac{\log p}{n}}\\
&\le C^2\sigma \frac{s_0\log p}{\sqrt{n}}\, ,\label{eq:EarlierUB}
\end{split}
\end{align}
where the bound $|\Omega\hSigma-\id|_{\infty}\le C\sqrt{(\log p)/n}$
follows from standard concentration arguments, and the bound on
$\|\th^*-\Lth\|_1$ is order-optimal and
is proved, for instance, in \cite{BickelEtAl,buhlmann2011statistics}.

This simple argument implies that the debiased estimator is
approximately Gaussian if the upper bound in Eq.~(\ref{eq:EarlierUB})
is negligible,
i.e. if $s_0 = o(\sqrt{n}/\log p)$. We see therefore that this requirement  is not imposed as to control the error in
estimating $\Omega$. It instead follows
from the simple $\ell_1$-$\ell_{\infty}$ bound \emph{even if $\Omega$ is known.}

\subsection{Main results}

The above exposition should clarify that the
$\ell_1-\ell_{\infty}$ bound is quite conservative. Considering the
$i$-th entry in the bias vector $\bias = (\Omega\hSigma-\id)(\th^*-\Lth)$, the $\ell_1$-$\ell_{\infty}$ bound
controls it as  $|\bias_i|\le
\|(\Omega\hSigma-\id)_{i,\cdot}\|_{\infty}
\|\th^*-\Lth\|_{1}$. This bound would be accurate only if the signs
of the entries $(\th_j^*-\Lth_j)$ were aligned to the signs
$(\Omega\hSigma-\id)_{i,j}$,  $j\in \{1,\dots, p\}$. 
While intuitively this is quite unlikely,  it is difficult to formalize this intuition; Note that in a random design setting, the terms $(\Omega\hSigma-\id)_{i,\cdot}$ and $\th^*-\Lth$ are highly
dependent: $\Lth$ is a deterministic function of the random pair $(X,w)$,
while  $(\Omega\hSigma-\id)= (\Omega X X^{\sT}/n-\id)$ is a function
of $X$. 

Our main result overcomes this technical hurdle via a careful
analysis of such dependencies. We follow a leave-one-out proof technique. 
Roughly speaking, in order to understand the distribution of the $i$-th
coordinate  of the debiased estimator $\dth_i$, we consider a modified
problem in which column $i$ is removed from the design matrix $X$. We
then study the consequences of adding back this column, and bound the effect
of this perturbation. An outline of this proof strategy is provided in
Section \ref{sec:Outline}.

We state below a simplified version of our main result, referring to
Theorem \ref{thm:main} below for a full statement, including technical
conditions.
\begin{thm}[Known covariance]\label{thm:MainSimplified}
Consider the linear model~\eqref{eq:NoisyModel} where $X$ has
independent Gaussian rows, with zero mean and covariance $\Sigma
=\Omega^{-1}$.
Assume that $\Sigma$ satisfies the  technical conditions stated in Theorem
\ref{thm:main}.
Define the debiased estimator $\dth$ via Eq.~(\ref{eq:debiased}) with $M =
\Omega$
and $\Lth = \Lth(y,X;\lambda)$ with $\lambda = 8\sigma\sqrt{(\log
  p)/n}$. 

If $n,p\to\infty$ with $s_0 = o(n/(\log p)^2)$, then we have
\begin{align}
\sqrt{n}(\dth-\th^*)= Z + o_P(1)\,, \quad \quad\quad
Z|X\sim\normal(0,\sigma^2\Omega\hSigma\Omega)\, .
\end{align}
Here $o_P(1)$ is a (random) vector satisfying $\|o_P(1)\|_{\infty}\to
0$ in probability as $n,p\to \infty$, and
$Z|X\sim\normal(0,\sigma^2\Omega\hSigma\Omega)$ means that the
conditional distribution of $Z$ given $X$ is centered Gaussian, with the stated covariance.
\end{thm}
\begin{remark}
The more complete statement of this result, Theorem  \ref{thm:main}
provides explicit non-asymptotic bounds on the error term $o_P(1)$, In
particular $\|o_P(1)\|_{\infty}$ turns out to be of order
$\sqrt{s_0/n}\,(\log p)$ with probability converging to one as $n,p\to\infty$.
\end{remark}

Theorem \ref{thm:MainSimplified} raises an important question: 
\emph{Does the Gaussian limit hold even if $M$ is an imperfect
  estimate of $\Omega$?}

 If the precision matrix $\Omega$ is sufficiently
structured, then it can be reliably estimated from the design matrix
$X$. 
Both \cite{zhang2014confidence} and \cite{van2014asymptotically} 
assume that $\Omega$ is sparse, and use 
the node-wise Lasso to construct an estimate $\hOmega$
\cite{MeinshausenBuhlmann}.
They then set $M = \hOmega$. 

We followed the same procedure and hence generalized Theorem
\ref{thm:MainSimplified} to the setting of unknown, sparse precision matrix.
We state here a simplified version of this result, deferring to Theorem
\ref{thm:unknown} for a more technical statement including
non-asymptotic probability bounds.
\begin{thm}[Unknown covariance]\label{thm:UnknownSimplified}
Consider the linear model~\eqref{eq:NoisyModel} where $X$ has
independent Gaussian rows with precision matrix $\Omega$, satisfying the  technical conditions of
Theorem \ref{thm:MainSimplified} (stated in Theorem
\ref{thm:main}).
Define the debiased estimator $\dth$ via Eq.~(\ref{eq:debiased}) 
with $\Lth = \Lth(y,X;\lambda)$, $\lambda = 8\sigma\sqrt{(\log
  p)/n}$, and $M =
\hOmega$ computed through node-wise Lasso  (see Section \ref{sec:debias}). 

Let $s_{\Omega}$ the maximum number of non-zero entries in any row of $\Omega$.
If $n,p\to\infty$ with $s_0 = o(n/(\log p)^2)$ and $\min(s_\Omega,s_0) = o(\sqrt{n}/\log p)$, then we have
\begin{align}
\sqrt{n}(\dth-\th^*)= Z + o_P(1)\,, \quad \quad\quad
Z|X\sim\normal(0,\sigma^2\Omega\hSigma\Omega)\, ,
\end{align}
where $o_P(1)$ is a (random) vector satisfying $\|o_P(1)\|_{\infty}\to
0$ in probability as $n,p\to \infty$.
\end{thm}
\begin{remark} As mentioned above, this version of the debiased estimator
can be constructed entirely from data. The only unspecified steps are the
choice of the regularization parameter $\lambda$, and the estimation
of the noise level $\sigma$. These can be addressed as in
\cite{zhang2014confidence,van2014asymptotically,javanmard2014confidence}
without changes in the sparsity condition : we will further discuss these points below.
\end{remark}
\begin{remark}
The sparsity condition  $\min(s_0,s_\Omega) = o(\sqrt{n}/\log p)$ nicely
illustrates the practical improvement implied by our more refined
analysis.
If the sparsity of the precision matrix is larger than the
sparsity of $\theta^*$, we recover 
the condition $s_0 = o(\sqrt{n}/\log p)$  which is assumed in the
results of \cite{zhang2014confidence,van2014asymptotically}.
(Note that \cite{javanmard2014confidence} obtain the same condition
without sparsity assumption on $\Omega$.) 
In this regime, our improved analysis does not bring any advantage,
since the bottleneck is due to the inaccurate estimation of $\Omega$.

On the other hand, if the precision matrix is sparser, we obtain a
much weaker condition on the coefficients $\theta^*$.
In particular, if $s_\Omega = o(\sqrt{n}/\log p)$, then the condition on $s_0$ is relaxed into
a nearly optimal condition $s_0 = o(n/(\log p)^2)$.

It is instructive to compare this with the past progress in
sparse estimation and compressed sensing. In that context, earlier work based on incoherence
conditions \cite{donoho2001uncertainty,donoho2006stable} implied
accurate reconstruction from a number of random samples scaling
quadratically in the number of non-zero coefficients. Subsequent
progress was based on the restricted isometry property \cite{CandesStable,Dantzig}, and
established accurate reconstruction from a linear number of measurements.
\end{remark}

\subsection{Extensions and applications}

\noindent{\bf Sample splitting.} An alternative approach to avoid the $\ell_1$-$\ell_{\infty}$ bound in
Eq.~(\ref{eq:EarlierUB}) is to  modify the definition of debiased
estimator in Eq.~(\ref{eq:debiased}), using sample-splitting. 
Roughly speaking, we can split the same in two batches of size
$n/2$. One batch is then used to estimate $\Lth$ and the other batch for $y$
and $X$ appearing in Eq.~(\ref{eq:debiased}) (and possibly for
computing $M$).

Appendix~\ref{app:dataSplit} discusses in greater detail this
method. This approach is subject to variations due to the random
splitting, and does not make use of part of half of the response
variables. While it provides a viable alternative,  it is not the focus of the present work.

\vspace{0.4cm}

\noindent{\bf Confidence intervals.} Theorem \ref{thm:UnknownSimplified} (and its formal version, 
Theorem \ref{thm:unknown}) allows the construction of confidence intervals using the same  general procedure 
as in
\cite{zhang2014confidence,van2014asymptotically,javanmard2014confidence}.
Namely, we construct the debiasing matrix $M$ from the design
matrix $X$, and an estimate $\hsigma$ of the noise variance.
Then, for a significance level $\alpha\in(0,1)$, we form the following
confidence interval for parameter $\theta_i$:
\begin{eqnarray}
J_i(\alpha) &\equiv& [\hth^\de_i-\delta(\alpha,n),\hth^\de_i+\delta(\alpha,n)]\,\label{eq:ConfInterval1}\\
\delta(\alpha,n) &\equiv&
\Phi^{-1}(1-\alpha/2)\frac{\hsigma}{\sqrt{n}}(M\hSigma
M^\sT)_{i,i}^{1/2}\,, \label{eq:ConfInterval2}
\end{eqnarray}
where $\Phi(x) \equiv \int_{-\infty}^xe^{-t^2/2}\de t/\sqrt{2\pi}$ is
the Gaussian distribution.
Section \ref{sec:debias} presents a formal analysis of
this procedure. 
A straightforward generalization also allows to compute p-values for
the null hypothesis $H_{0,i}:$\, $\theta_{i}^* = 0$.

\vspace{0.4cm}

\noindent{\bf Noise level and regularization.} The construction of the confidence
  interval $J_i(\alpha)$ in Eqs.~(\ref{eq:ConfInterval1}),
  (\ref{eq:ConfInterval2}) requires  a suitable choice of the
  regularization parameter $\lambda$, and an estimate of the noise
  level $\hsigma$. The same difficulty was present in
 \cite{zhang2014confidence,van2014asymptotically,javanmard2014confidence}.
The approaches used there (for instance, using the scaled Lasso
\cite{SZ-scaledLASSO}) can be followed in the present case as well. Under the
assumptions of Theorem \ref{thm:MainSimplified}, the same
proofs of \cite{javanmard2014confidence}  show that the additional
error due to the choice of $\lambda$ and $\hsigma$ are negligible.

\vspace{0.4cm}

\noindent{\bf Semi-supervised learning.}
In some applications, the precision matrix $\Omega$ can be estimated more accurately thanks to additional information.
For instance, in semi-supervised learning, the statistician is
given additional samples $\bax_1,\bax_2,\dots, \bax_{N}\in\reals ^p$
with the same distribution as the $\{x_i\}_{1\le i\le n}$.
For these `unlabeled' samples, the response variable is unknown. There are indeed many applications
in which acquiring the response variable is much more challenging than 
capturing the covariates \cite{chapelle2006semi}, and therefore $N\gg n$ or even $N\gg p$. In this
setting, we can estimate $\Omega$ more accurately from
$\{\bax_i\}_{1\le i\le N}$, then
use this estimate to construct $M$.

\vspace{0.4cm}

{\bf Non-Gaussian designs.} We expect that generalization of  Theorem \ref{thm:MainSimplified} and 
Theorem \ref{thm:UnknownSimplified} should hold for a broad
class of random designs with independent sub-Gaussian rows, although new proof ideas are required.
The main technical challenge in extending the present approach is to
generalize the leave-one-out construction. As discussed in Section
\ref{sec:Outline}, when studying the effect of modifying column $i$,
we need to account for dependencies between columns. For Gaussian designs, these dependencies are fully captured by
the design covariance $\Sigma$.

Note that the Gaussian assumption holds in the context of estimating Gaussian
graphical models. This is itself a broad topic that attracted
significant interest, since the seminal work of \cite{MeinshausenBuhlmann}.
Remarkably, recent contributions have shown the utility of debiasing
methods in this context  \cite{jankova2015confidence,chen2015asymptotically,jankova2015honest}.

\subsection{Organization and contributions}

The rest of the paper presents the following contributions:
\begin{enumerate}
\item Section \ref{sec:Results}. We state formally our Gaussian limit theorems, and use them to construct valid confidence intervals, 
of nearly optimal size. In particular, our results subsume (and improve) all previously known 
results on the debiased estimator for Gaussian designs.
\item Section \ref{sec:Minimax}. We establish a minimax lower bound on the $\ell_{\infty}$ norm of  the non-Gaussian component
in $\dth$. This implies that our Gaussian limit theorems cannot be substantially improved.
\item Section \ref{sec:Other}. Apart from the construction of confidence intervals, our Gaussian limit theorems have several fundamental implications. 
We discuss a a few examples, that we consider particularly interesting.
In particular, we construct a thresholded Lasso estimator that is minimax optimal up to a factor $(1+o_n(1))$
(an alternative approach to the same problem was recently proposed in \cite{su2015slope}).
\end{enumerate}
Section \ref{sec:Related} discusses relations with earlier work in this
area.
Outlines of the proofs of the main theorems are given in Section \ref{sec:thm_main} and Section
\ref{proof:thm_unknown} with most of the technical work deferred to  appendices.

%
%
\section{Related work}
\label{sec:Related}

A parallel line of research develops methods for performing valid inference
after a low-dimensional model is selected for fitting high-dimensional
data
\cite{lockhart2014significance,fithian2014optimal,taylor2014exact,chernozhukov2014valid}.
 The resulting significance statements are typically
conditional on the selected model. In contrast, here we are interested
in classical (unconditional) significance statements: the two
approaches are broadly complementary.


The focus of the present paper is assessing statistical significance, such as confidence
intervals, for single coordinates in the parameters vector $\theta^*$ and more generally
for small groups of coordinates.  Other inference tasks are also
interesting and challenging in high-dimension,
 and were the object of recent investigations \cite{bayati2013estimating,barber2015controlling,janson2015eigenprism,janson2015familywise}.

Sample splitting provides a general methodology for inference in high
dimension \cite{wasserman2009high,MeinshausenBuhlmannStability}. 
As mentioned above, sample splitting can also be used to define a
modified debiased estimator, see Appendix \ref{app:dataSplit}.
However sample splitting techniques typically use only part of the data
for inference, and are therefore sub-optimal. Also, the result depend
on the random split of the data.

A method for inference without assumptions on the design matrix was
developed in  \cite{meinshausen2014group}. The resulting confidence
intervals are typically quite conservative.

The debiasing method was developed independently from several points
of view
\cite{BuhlmannSignificance,zhang2014confidence,javanmard2013hypothesis,van2014asymptotically,javanmard2014confidence}.
The present authors were motivated by the AMP analysis of the Lasso
\cite{DMM09,BM-MPCS-2011,BayatiMontanariLASSO,bayati2015universality}, and by the Gaussian limits that this analysis implies. In
particular  \cite{javanmard2013hypothesis} used those techniques to
analyze standard Gaussian designs (i.e. the case $\Sigma = \id$) in the asymptotic limit $n,p,s_0\to \infty$
with $s_0/p$, $n/p$ constant. In this limit, the debiased estimator
was proven to be asymptotically Gaussian provided $s_0\le C\,
n/\log(p/s_0)$ (for a universal constant $C$). This sparsity condition
is even weaker than the one of Theorem \ref{thm:MainSimplified} (or
Theorem \ref{thm:main}), but the result of
\cite{javanmard2013hypothesis}  only holds asymptotically. 
Also \cite{javanmard2013hypothesis}  proved Gaussian convergence in a
weaker sense than the one established here, implying coverage of the
constructed confidence intervals only `on average' over the
coordinates $i\in\{1,\dots,p\}$.

A non-asymptotic result under weaker sparsity conditions, and for
designs with dependent columns, was proved in
\cite{javanmard2013nearly}. However, this only establishes
gaussianity of $\dth_i$ for most of the coordinates $i\in
\{1,\dots,p\}$. Here we prove a significantly stronger result holding uniformly over $i\in
\{1,\dots,p\}$.

Most of the work on statistical inference in high-dimensional models 
has been focused so far on linear regression. 
The debiasing method admits a natural extension to generalized linear models that
was analyzed in \cite{van2014asymptotically}. Robustness to model
misspecification was studied in \cite{buhlmann2015high}.
An R-package for inference in high-dimension that uses the node-wise
Lasso is available \cite{dezeure2015high}.
An R implementation of the method \cite{javanmard2014confidence}
(which does not make sparsity assumptions on $\Omega$) is also
available\footnote{See {\sf
    http://web.stanford.edu/~montanar/sslasso/}.}. 
%
%
\section{Main results: Gaussian limit theorems}
\label{sec:Results}

\subsection{General notations}

We use $e_i$ to
refer to the $i$-th standard basis element, e.g., $e_1 = (1,0,\dotsc,0)$. For a vector $v$, $\supp(v)$ represents
the positions of nonzero entries of $v$. Further, $\sign(v)$ is the
vector with entries $\sign(v)_i = +1$ if $v_i > 0$, $\sign(v)_i = -1$ if $v_i < 0$, and $\sign(v)_i = 0$ otherwise.
For a matrix $M\in \reals^{n\times p}$ and a set of indices $J\subseteq[p]$
we use $M_{J}$ to denote the submatrix formed by columns in $J$. Likewise, for a vector $\theta$ and a subset $S$, 
$\theta_S$ is the restriction of $\theta$ to indices in $S$. For an integer $p\ge 1$, we use the notation $[p] = \{1,\cdots, p\}$ and the shorthand $\sim{i}$ for the set $[p]\backslash{i}$. We write $\|v\|_p$ for the standard $\ell_p$ norm of a vector $v$, i.e., $\|v\|_p = (\sum_i |v_i|^p)^{1/p}$
and $\|v\|_0$ for the number of nonzero entries of $v$.  For a matrix
$A\in \reals^{m\times n}$, $\|A\|_p$ denotes it $\ell_p$ operator
norm; in particular, $\|A\|_\infty = \max_{1\le i\le m}\sum_{j=1}^n
|A_{ij}|$. 
This is to be contrasted with the maximum absolute value of any entry
of $A$ that, as mentioned above, we denote by $|A|_{\infty}
\equiv\max_{i\le m, j\le n}|A_{ij}|$.
For a matrix $A$, we denote its maximum and minimum singular values by $\sigma_{\max}(A)$ and $\sigma_{\min}(A)$, respectively.
If $A$ is symmetric, $\lambda_{\max}(A)$ and $\lambda_{\min}(A)$ are its maximum and minimum eigenvalues.
Finally,  for two
functions $f(n)$ and $g(n)$, the notation $f(n) \gg g(n)$ means that $f$ `dominates' $g$ asymptotically,
namely, for every fixed positive $C$, there exists $n(C)$ such that $f(n) \ge C g(n)$ for $n > n(C)$. We also use $f(n) \lesssim g(n)$ to indicate that
$f$ is `bounded' above by $g$ asymptotically, i.e., $f(n) \le C g(n)$ for some positive constant $C$. The notations $f(n) \ll g(n)$ and
$f(n) = o(g(n))$ are defined analogously, and we use $o_P(\,\cdot\,)$ to indicate
asymptotic behavior in probability as the sample size $n$ tends to
infinity.

We will use $c, C,\dots$ to denote generic constants that can vary
from one position to the other of the paper. 

\subsection{Preliminaries}\label{sec:preliminary}
This section includes some preliminary results that are repeatedly used in our proofs.
We start by  some  well-known results about the Lasso estimator.
For the sake of simplicity, we will often use $\hth =
\hth(y,X;\lambda)$ instead of  $\Lth$ to denote the Lasso estimator.

We denote the rows of the design matrix $X$ by $x_1,\dotsc, x_n \in \reals^p$ and 
its columns by $\tx_1, \dotsc, \tx_p\in \reals^n$. The empirical covariance of the design $X$ is defined as $\hSigma \equiv (X^\sT X)/n$. 
The population covariance will be denoted by $\Sigma$, and we let $\Omega \equiv \Sigma^{-1}$ be the precision matrix. 
\begin{definition}\label{def:phi}
Given a symmetric matrix $\hSigma\in\reals^{p\times p}$ and a set
$S\subseteq [p]$, the corresponding \emph{compatibility constant}
is defined as
\begin{align}
\phi^2(\hSigma,S) \equiv
\min\Big\{\frac{|S|\,\<\theta,\hSigma\,\theta\>}{\|\theta_S\|_1^2} :\;
\theta\in\reals^p, 
\;\|\theta_{S^c}\|_1\le 3\|\theta_S\|_1\Big\}\, .
\end{align}
We say that $\hSigma\in\reals^{p\times p}$ satisfies the
\emph{compatibility condition} for the  set $S\subseteq [p]$, with
constant $\phi$ if $\phi(\hSigma,S)\ge \phi$.
We say that it holds for the design matrix $X$, if it holds for
$\hSigma = X^{\sT}X/n$.
\end{definition}

It is also useful to recall some notation for the restricted eigenvalue condition, introduced by Bickel, Ritov and
Tsybakov \cite{BickelEtAl}. 
For an integer $0<s_0<p$ and a positive number $L$, define $\cone(s_0,L)\in\reals^p$  by the 
following cone constraints:
\begin{align}
\cone(s_0,L) \equiv \{\theta\in \reals^p: \; \exists S\subseteq[p],\; |S| = s_0, \; \|\theta_{S^c}\|_1 \le L \|\theta_S\|_1\}\,.
\end{align}

In high-dimension, the empirical covariance $\hSigma$ is singular. However,
we can ask for non-singularity of $\hSigma$  for vectors in $\cone(s_0,L)$. 
Rudelson and Zhou \cite{rudelson2013reconstruction} prove a reduction principle that bounds 
the restricted eigenvalues of the empirical covariance in terms of those of the population covariance.
We will use their result specified to the case of Gaussian matrices.
\begin{lemma}[\cite{rudelson2013reconstruction}, Theorem 3.1]\label{lem:rudelson}
Suppose that $\sigma_{\min}(\Sigma) > C_{\min}>0$ and $\sigma_{\max}(\Sigma) < C_{\max} < \infty$.
Let $X\in \reals^{n\times p}$ have independent rows drawn from $\normal(0,\Sigma)$. Set $0<\delta<1$, $0<s_0<p$, and $L>0$. Define the following event
\begin{eqnarray}\label{eq:Bdelta}
\cB_\delta(n,s_0,L) \equiv \Big\{X\in \reals^{n\times p}:\, (1-\delta)\sqrt{C_{\min}}\le \frac{\|Xv\|_2}{\sqrt{n}\|v\|_2} \le (1+\delta)\sqrt{C_{\max}}\,,\, \forall v\in  \cone(s_0,L) \,\,\text{s.t.}\,\, v\neq 0 \Big\}\,.
\end{eqnarray}
Then, there exists a constant $c_1 = c_1(L)$ such that, 
for sample size $n \ge c_1 s_0 \log (p/s_0)$, we have
\begin{align}
\prob(\cB_\delta(n,s_0,L)) \ge 1- 2e^{-\delta^2 n}\,.
\end{align}
\end{lemma}
\begin{remark}\label{rem:phi-bound}
Fix $S\subseteq [p]$ with $|S| = s_0$. Under the event $\cB_\delta(n,s_0,3)$, we have
\begin{align*}
\phi^2(\hSigma, S) &\ge \min_{\theta\in \cone(s_0,3)} \frac{s_0 \<\theta,\hSigma \theta\>}{\|\theta_S\|_1^2}\ge \min_{\theta\in \cone(s_0,3)} \frac{\<\theta,\hSigma \theta\>}{\|\theta_S\|_2^2}
\ge (1-\delta)^2 {C_{\min}}\,,
\end{align*}
where the second inequality follows from Cauchy-Schwartz inequality.
\end{remark}
We next introduce the event
\begin{eqnarray}\label{eq:ctB}
\ctB(n,p) \equiv \bigg\{w\in\reals^n:\; \frac{1}{n}\|X^\sT w\|_\infty \le 2\sigma\sqrt{\frac{\log p}{n}} \bigg\}\,.
\end{eqnarray}
On $\ctB(n,p)$ we can control the randomness  due to the measurement noise.
A well-known union bound argument shows that $\ctB(n,p)$ has large
probability
(see, for instance, \cite{buhlmann2011statistics}).
\begin{lemma}[\cite{buhlmann2011statistics}, Lemma 6.2]\label{lem:ctB}
Suppose that $\hSigma_{ii} \le1$ for $i\in [p]$. Then we have
$$\prob(\ctB(n,p)) \ge 1-2p^{-1}\,.$$
\end{lemma}
The following Lemma states that the Lasso estimator is sparse. Its proof is given in Appendix~\ref{app:Lasso-supp-size}.
\begin{lemma}\label{lem:Bickel}
Consider the Lasso selector $\hth$ with $\lambda = \kappa\sigma \sqrt{\log p/n}$, for a constant $\kappa\ge 8$.
On the event $\cB\equiv \ctB(n,p) \cap \cB_\delta(n,s_0,3)$, the following holds:
\begin{align}
|\hS| < C_*s_0\,,\label{eq:bound-phi0}
\end{align}
with 
\begin{align}\label{eq:C*}
C_*\equiv \frac{16 C_{\max}}{(1-\delta)^2 C_{\min}} \,.
\end{align}
\end{lemma}

Our next Lemma states a property of Gaussian design matrices which will be used repeatedly in our analysis.
Its proof is very short and is given here for the reader's convenience. 
\begin{lemma}\label{lem:independence}
Let $v_i= X \Omega e_i$. Then $v_i$ and $X_{\sim i}$ are independent. 
\end{lemma}
\begin{proof}
Define $u = \Omega e_i$ and fix $j\neq i$. Recall that $\tx_\ell$ denotes the $\ell$-th column of $X$. 
We write $v_i = \sum_{\ell=1}^p \tx_\ell u_\ell$ and 
\begin{align*}
\E(v_i\tx_j^\sT) &= \sum_{\ell=1}^p u_\ell \E(\tx_\ell \tx_j^\sT) 
= \sum_{\ell=1}^p u_\ell \Sigma_{\ell j} \id_{n\times n} = \sum_{\ell=1}^p \Omega_{\ell i} \Sigma_{\ell j} \id_{n\times n}
=  (\Omega \Sigma)_{ij} \id _{n\times n} = 0\,,
\end{align*}
where the last step holds since $i \neq j$. Since $v_i$ and $\tx_j$ are jointly Gaussian, this implies that they are independent. 
\end{proof}

We finally introduce some parameters  that are used in stating our main theorems.
For an integer $k$ and an invertible matrix $A\in \reals^{p\times p}$, we define $\rho(A,k)$ as follows:
\begin{align}\label{eq:mu}
\rho(A,k) \equiv \max_{T\subseteq [p],|T|\le k}\; \|A_{T,T}^{-1}\|_{\infty} \,,
\end{align}
where we adopt the convention $A^{-1}_{T,T} = (A_{T,T})^{-1}$ and recall that $\|\cdot\|_\infty$ denotes the $\ell_\infty$
operator norm (maximum $\ell_1$ norm of the rows). It is clear that $\rho(A,k)$ is non-decreasing in $k$.

\begin{lemma}\label{lem:prop-norm}
For an invertible matrix $A$, we have
\begin{align}
\rho(A,p) = \|A^{-1}\|_\infty\, .
\end{align}
\end{lemma}
Lemma~\ref{lem:prop-norm} is proved in Appendix~\ref{app:prop-norm}.
As a result of Lemma~\ref{lem:prop-norm} and the non-decreasing property of $\rho(A,k)$, for any $1\le k\le p$ we have
\begin{align}\label{ineq-norm}
\rho(A,k)\le \rho(A,p) = \|A^{-1}\|_\infty\,.
\end{align}
Another bound on $\rho(A,k)$ is as follows:
\begin{align}\label{eq:B2}
\rho(A,k)&\le \max_{T\subseteq [p],|T|\le k}\;\max_{j\in [p]}\; \sqrt{k}\, \|A_{T,T}^{-1} e_j\|_2 \le
\max_{T\subseteq [p],|T|\le k}\; \sqrt{k}\; \sigma_{\max}(A_{T,T}^{-1})
\le \frac{\sqrt{k}}{\sigma_{\min}(A)}.
\end{align}

%
\subsection{Statement of main theorems}\label{sec:debias}

In our first theorem, we assume that the precision matrix $\Omega \equiv \Sigma^{-1}$ is available and we set $M=\Omega$.
We prove the corresponding debiased estimator is asymptotically unbiased provided that $n\gg s_0(\log p)^2$.

%
\subsubsection{Known covariance}
\begin{thm}[Known covariance]\label{thm:main}
Consider the linear model~\eqref{eq:NoisyModel} where $X$ has independent Gaussian rows, with zero mean and covariance $\Sigma$ and $\theta^*$ is $s_0$-sparse.  
Suppose that $\Sigma$ satisfies the following conditions:
\begin{itemize}
\labitem{$(i)$}{Condition:diag} For $i\in [p]$, we have $\Sigma_{ii}\le 1$.
\labitem{$(ii)$}{Condition:eig} We have $\sigma_{\min}(\Sigma) > C_{\min}>0$ and $\sigma_{\max}(\Sigma) < C_{\max}$ for some constants
$C_{\min}$ and $C_{\max}$.
\labitem{$(iii)$}{Condition:L1} Define $C_0 \equiv (32 C_{\max}/C_{\min}) +1$. We have $\rho(\Sigma,C_0  s_0) \le \rho$, for some constant $\rho > 0$.
\end{itemize}
Let $\hth$ be the Lasso estimator defined by~\eqref{eq:Lasso} with
$\lambda = \kappa\sigma \sqrt{(\log p)/n}$, for  $\kappa\in [8,\kappa_{\rm max}]$.
 Further, let $\dth$ be defined as per equation~\eqref{eq:debiased}, with $M=\Omega\equiv \Sigma^{-1}$. Then, there exist constants $c, C$ depending solely on $C_{\min},C_{\max}$, and $\kappa_{\rm max}$, such that, for $n\ge \max(25 \log p , cs_0\log(p/s_0))$
the following holds true: 
\begin{align}
&\sqrt{n}(\dth-\th^*)= Z + R\,, \quad \quad Z|X\sim\normal(0,\sigma^2\Omega\hSigma\Omega)\,, \label{eq:components}\\
&\prob\Big(\|R\|_\infty \ge C\rho \sigma\sqrt{\frac{s_0}{n}}\log p \Big) \le
  2pe^{-c_*n/s_0} + pe^{-n/1000}+8p^{-1}\,,\label{eq:R}
\end{align}
with $c_* \equiv  C_{\min}/16$.
\end{thm}
The proof of this theorem is presented in Section \ref{sec:thm_main}.

This theorem states that if the sample size satisfies $n = \Omega(s_0\log p)$, then the maximum size of the `bias' $R_i$
over $i\in [p]$ is bounded by 
\[\|R\|_\infty = O_P\Big(\sqrt{\frac{s_0}{n}} \log p\Big)\,.\]
On the other hand, each entry of the `noise term' $Z_i$ has variance $\sigma^2(\Omega\hSigma\Omega)_{ii}$.
Applying Lemma 7.2 in~\cite{javanmard2013nearly}, we have $|\Omega\hSigma\Omega - \Omega|_\infty = o_P(1)$ and thus $\min_{i\in [p]}(\Omega\hSigma\Omega)_{ii} \ge \min_{ii}\Omega_{ii} - o_P(1)$ is of order one because $\Omega_{ii} \ge C_{\max}^{-1}$. Hence, $|R_i|$ is much smaller than $Z_i$ for $n \gg s_0 (\log p)^2$. We summarize this observation in the remark below.
%
%
%
\begin{remark}(\emph{Discussion of the assumptions on $\Sigma$}.)
Assumption $(i)$ sets the normalization of the design matrix.
Assumptions $(ii)$ on the eigenvalues of $\Sigma$ is common in high-dimensional models.
Further, note that by Assumption $(ii)$ and invoking Eq.~\eqref{eq:B2}, we have $\rho(\Sigma,C_0 s_0) \le \sqrt{C_0 s_0}/C_{\min}$. Using this bound for $\rho$ in Eq.~\eqref{eq:R},  we recover 
the bound $\|R\|_\infty \lesssim s_0  \log p/\sqrt{n}$  which is established in previous work
\cite{zhang2014confidence,van2014asymptotically, javanmard2014confidence}.
Note that this bound on the bias does not require Assumption $(iii)$ (namely, that $\rho$ is a bounded constant). However, Theorem~\ref{thm:main} asserts that,
if $\rho$ is a constant (Assumption $(iii)$), we have a sharper bound on the bias, namely $\|R\|_\infty \lesssim \sqrt{s_0/n} \,\log p$. 

A large family of covariance matrices satisfy conditions of Theorem~\ref{thm:main}. Examples include block diagonal matrices where the size of blocks are bounded, and circulant matrices, where $\Sigma_{i,j} = r^{|i-j|}$, for some $r\in(0,1)$.
\end{remark}

\begin{coro}\label{coro:distribution}
Under the assumptions of Theorem~\ref{thm:main}, if $s_0 \ll n/(\log p)^2$, then $\hth^\de$ is normal distributed. More precisely, 
let $\hsigma = \hsigma(y,X)$ be an estimator of the noise
level satisfying, for any $\eps>0$,
\begin{align}
\lim_{n\to\infty} \sup_{\tth\in\reals^p;\, \|\theta_0\|_0 \le s_0 }\prob\Big(\Big|\frac{\hsigma}{\sigma}-1\Big|\ge \eps
\Big)=0\, .\label{eq:ConsistencySigma}
\end{align}
If $s_0 \ll n/(\log p)^2$, then, for all $x\in\reals$, we have
\begin{eqnarray}\label{eq:distribution}
\lim_{n\to\infty}\sup_{\theta_0\in\reals^p;\, \|\tth\|_0 \le s_0 }\left|\prob \left\{\frac{\sqrt{n}(\dth_i - \theta^*_i)}{\hsigma [\Omega \hSigma \Omega^\sT]_{i,i}^{1/2}} 
\le x  \right\} -\Phi(x)\right| =0\, . 
\end{eqnarray}
\end{coro}
Armed with a precise distributional characterization of $\hth^\de$, we
can construct asymptotically valid confidence intervals for each
parameter $\theta_{0,i}$ as per Eqs.~(\ref{eq:ConfInterval1}),
(\ref{eq:ConfInterval2}). Validity of the constructed confidence intervals requires a
consistent estimator of $\sigma$. There are several proposal for such estimator. A non-exhaustive list 
includes~\cite{SCAD01,sis08,SBvdG10,mcp10,SZ-scaledLASSO,
BelloniChern,reid2013study,dicker2012residual,isis13,bayati2013estimating}.
For concreteness, we use the the scaled Lasso~\cite{SZ-scaledLASSO} given by
\begin{align}
\{\htheta, \hsigma\} \equiv \underset{\theta\in\reals^p,\sigma> 0}{\arg\min}\,
\Big\{\frac{1}{2\sigma n}\|Y-X\theta\|^2_2+ \frac{\sigma}{2}
+\bar{\lambda}\|\theta\|_1\Big\}\, .
\label{eq:SLASSO}
\end{align}

The following proposition 
shows that the scaled Lasso estimate $\hsigma$ satisfies the consistency criterion
(\ref{eq:ConsistencySigma}).
\begin{lemma}\label{lemma:ConsistencySigma}
Let $\hsigma$ be the scaled Lasso estimator of the noise level, see
Equation~(\ref{eq:SLASSO}), with $\bar{\lambda} = 10\sqrt{(2\log p)/n}$. Then
 $\hsigma$ satisfies Equation~(\ref{eq:ConsistencySigma}).
\end{lemma}
We refer to our earlier work~\cite[Appendix C]{javanmard2014confidence} for the proof of Lemma~\ref{lemma:ConsistencySigma}.

Furthermore, in the context of hypothesis testing, we can test the null hypothesis $H_{0,i}:\theta_0 = 0$ versus the alternative $H_{A,i}:\, \theta_{0,i}\neq 0$. We construct the two sided $p$-values 
\begin{align}\label{eq:2sidedP}
P_i = 2\bigg(1-\Phi\Big(\frac{\sqrt{n}|\hth^\de_i|}{\hsigma(\Omega\hSigma
  \Omega^\sT)_{i,i}^{1/2}}\Big) \bigg)\,.
\end{align}
The decision rule follows immediately: we reject $H_{0,i}$ if $P_i\le
\alpha$. 

\begin{remark} It is worth noting that the sample splitting approach, discussed in Appendix~\ref{app:dataSplit}, does not require Assumption~\ref{Condition:L1} in Theorem~\ref{thm:main}.
However as pointed in the introduction, this approach suffers from variability due to the random splitting and does not make use of half of the response variables.  
\end{remark}

\vspace{0.5cm}

\subsubsection{Unknown covariance}
We next generalize our result to the case of unknown covariance, where following ~\cite{zhang2014confidence,van2014asymptotically} we construct the debiasing matrix $M$ using node-wise Lasso on matrix $X$. 
For reader's convenience, we first describe this construction.

For $i \in [p]$, we define the vector $\hat{\gamma_i} =
(\hat{\gamma}_{i,j})_{j\in [p]\setminus i}\in\reals^{p-1}$ by
performing sparse regression of the $i$-th column of $X$ against all
the other columns. Formally
\begin{align}
\hat{\gamma}_i(\tlambda)= \underset{\gamma\in\reals^p}{\arg\min}\, \Big\{\frac{1}{2n}
  \|\tilde{x}_i- X_{\sim i}\gamma\|_2^2 + \tlambda \|\gamma\|_1\Big\}\,,  
\end{align}
where $X_{\sim i}$ is the sub-matrix obtained by removing the $i$-th
column (and columns indexed by $[p]\setminus i$). 
Also define
\begin{align}
 \widehat{C} = 
\begin{bmatrix}
1 & -\hgamma_{1,2} & \cdots & -\hgamma_{1,p}\\
-\hgamma_{2,1} & 1 & \cdots & -\hgamma_{2,p}\\
\vdots & \vdots & \ddots & \vdots\\
-\hgamma_{p,1} & -\hgamma_{p,2} & \cdots &1  
\end{bmatrix}\,,
\end{align}
and let  
\begin{align}
\widehat{T}^2 = \diag(\htau_1^2,\dotsc,\htau_p^2),\quad\quad \htau_i^2
  = 
\frac{1}{n}(\tilde{x}_i - X_{\sim i} \hgamma_i)^\sT \tilde{x}_i\,. 
\end{align}
Finally, define $M=M(\tlambda)$ by
\begin{align}\label{eq:nodewise}
M = \widehat{T}^{-2} \widehat{C}\,.
\end{align}

\begin{thm}[Unknown covariance]\label{thm:unknown}
Consider the linear model~\eqref{eq:NoisyModel} where $X$ has independent Gaussian rows, with zero mean and covariance $\Sigma$.  
Suppose that Assumptions $(i),(ii),(iii)$ in Theorem~\ref{thm:main}
hold true for $\Sigma$. We further let
$s_{\Omega}$ be the maximum sparsity of the rows of $\Omega \equiv
\Sigma^{-1}$, i.e.
\begin{align}\label{eq:sOmega}
s_\Omega \equiv \max_{i\in[p]} |\{j\neq i, \Omega_{i,j} \neq 0\}|\,.
\end{align}
Let $\hth$ be the Lasso estimator defined by~\eqref{eq:Lasso} with
$\lambda = \kappa\sigma \sqrt{(\log p)/n}$, for $\kappa\in [8,\kappa_{\max}]$ and let $\hth^\de$ be debiased
estimator with $M$ given by~\eqref{eq:nodewise} and $\tlambda=
K\sqrt{\log p/n}$ (with $K$ a suitably large universal constant). 
Suppose that $s_\Omega \ll n/(\log p)$.

Then, there exist constants $c, C$ depending solely on
$C_{\min},C_{\max}, \kappa_{\max}$ such that, for $n\ge cs_0\log p$,
the following holds true:
\begin{align}\label{eq:components2}
&\sqrt{n}(\hth^\de-\th^*)= Z + R\,, \quad \quad Z|X\sim\normal(0,\sigma^2M\hSigma M^\sT)\,, \\
&\|R\|_\infty\le C\rho \sigma \sqrt{\frac{s_0}{n}} \log p+  C \sigma\min(s_0,s_\Omega) \frac{\log p}{\sqrt{n}}\,,
\end{align}
with probability at least $1- 2pe^{-c_*n/s_0} + pe^{-cn}+8p^{-1}$, for some constants $c_*,c',c'' >0$.
\end{thm}
The proof of Theorem~\ref{thm:unknown} is deferred to Section~\ref{proof:thm_unknown}.

A result similar to Corollary~\ref{coro:distribution} holds true for the case of unknown covariance.
\begin{coro}\label{coro:distributionUnknown}
Let $\hsigma = \hsigma(y,X)$ be an estimator of the noise
level satisfying Eq.~(\ref{eq:ConsistencySigma}) for any $\eps>0$.

Under the assumptions of Theorem~\ref{thm:main}, if $\min(s_0, s_\Omega) \ll \sqrt{n}/\log p$ and 
$s_0\ll n/(\rho(\log p)^2)$, then for all $x\in\reals$ we have
\begin{eqnarray}\label{eq:distributionNEW}
\lim_{n\to\infty}\sup_{\theta_0\in\reals^p;\, \|\tth\|_0 \le s_0 }\left|\prob \left\{\frac{\sqrt{n}(\dth_i - \theta^*_i)}{\hsigma [M \hSigma M^\sT]_{i,i}^{1/2}} 
\le x  \right\} -\Phi(x)\right| =0\, ,
\end{eqnarray}
where $M$ is given by equation~\eqref{eq:nodewise}.
\end{coro}
Using the above distributional characterization, we can construct confidence intervals for the individual model parameters $\theta^*_i$ as in~\eqref{eq:ConfInterval1}, \eqref{eq:ConfInterval2} with $M$ given by~\eqref{eq:nodewise} and $\hsigma$ given by the scaled Lasso as per~\eqref{eq:SLASSO}.
For hypothesis testing task, two sided $p$-values can be built similar to~\eqref{eq:2sidedP}, where we replace $\Omega \hSigma \Omega$ with $M\hSigma M^\sT$. 
%

\subsection{Numerical illustration}\label{sec:numerical}
Our goal in this section is to numerically corroborate the results of
Theorem~\ref{thm:main} and Theorem \ref{thm:unknown}. More specifically, we would like to check whether
the  debiased estimator exhibits an unbiased Gaussian distribution provided that the sample size scales linearly with the number of 
nonzero   parameters. 

We generate data from linear model~\eqref{eqn:regression} with the following configuration. We fix $p= 3000$ and consider regression parameter $\theta_0$ with support $S_0$ chosen uniformly at random from the index set $[p]$ and $\theta_{0,i} = 0.15$ for $i\in S_0$ and zero otherwise. The design matrix $X$ has i.i.d. rows drawn from $\normal(0,\Sigma)$, where $\Sigma \in \reals^{p\times p}$ is the circulant matrix with entries $\Sigma_{i,j} = 0.8^{|i-j|}$. The measurement noise $w$ has i.i.d. standard normal entries.
 
Let $s_0 = |S_0|$ and $\varepsilon = s_0/p$ be the sparsity level and $\delta = n/p$ denote the under sampling rate.
We vary $\varepsilon$ in the set $\{0.1,0.15,0.2, 0.25,0.3\}$ and for
each value of $\varepsilon$ we compute critical value of $\delta$
above which
the unbiased estimator admits a Gaussian distribution. 
We will denote this critical value as $\delta_c$ and define it as
follows. We vary $\delta$ and for each pair $(\varepsilon, \delta)$,
compute the debiased estimator (with $M=\Sigma^{-1}$) for $100$ realizations of noise $w$. We then compute the empirical kurtosis of each coordinate $T_i=\sqrt{n}(\dth_i-\th^*_{i})/(\sigma [M\hSigma M]_{i,i}^{1/2})$. For $i \in [p]$, let $\gamma^\delta_i$ denote the empirical kurtosis of $T_i$, where we make the dependence on $\delta$ explicit in the notation. Denote by $m(\gamma^\delta)$ and $\SD(\gamma^\delta)$ the mean and the standard deviation of $\gamma^\delta = (\gamma^\delta_1,\dotsc, \gamma^\delta_p)$, respectively. We further define the standard error $\SE(\gamma^\delta) = \SD(\gamma^\delta)/\sqrt{p}$. 
We use one standard error rule to decide the value of
$\delta_c$. Namely, 
\begin{align}
\delta_c = \arg\min\{\delta\in (0,1),\,\, {\rm s. t.,}\,\,
m(\gamma^\delta) \le \SE(\gamma^\delta ) \,\}\,.
\end{align}

Figure~\ref{fig:eps02} corresponds to $\varepsilon = 0.2$. The dots indicate $m(\gamma^\delta)$ and the dotted lines correspond to $m(\gamma^\delta)\pm \SE(\gamma^\delta)$. By one standard error rule, the estimated value of $\delta_c$ works out at $\delta_c = 0.57$.

Figure~\ref{fig:delta_eps} shows $\delta_c$ versus $\varepsilon$. 
The black curve corresponds to the case of known covariance, where we set $M = \Omega$ and 
the red curve corresponds to the case of unknown covariance, where $M$ is set as in Equation~\eqref{eq:nodewise}.
 The figure confirms that $\delta_c$ scales roughly linearly in
$\varepsilon$ (for small $\varepsilon$). In other words, in order for the debiased estimator to
have unbiased Gaussian distribution, the sample size $n$ has only  to
scale linearly in the support size $s_0$. (Note that for the circulant covariance chosen in this example, $s_\Omega = 2$).

\begin{figure}[]
\centering
\includegraphics*[width = 3.5in]{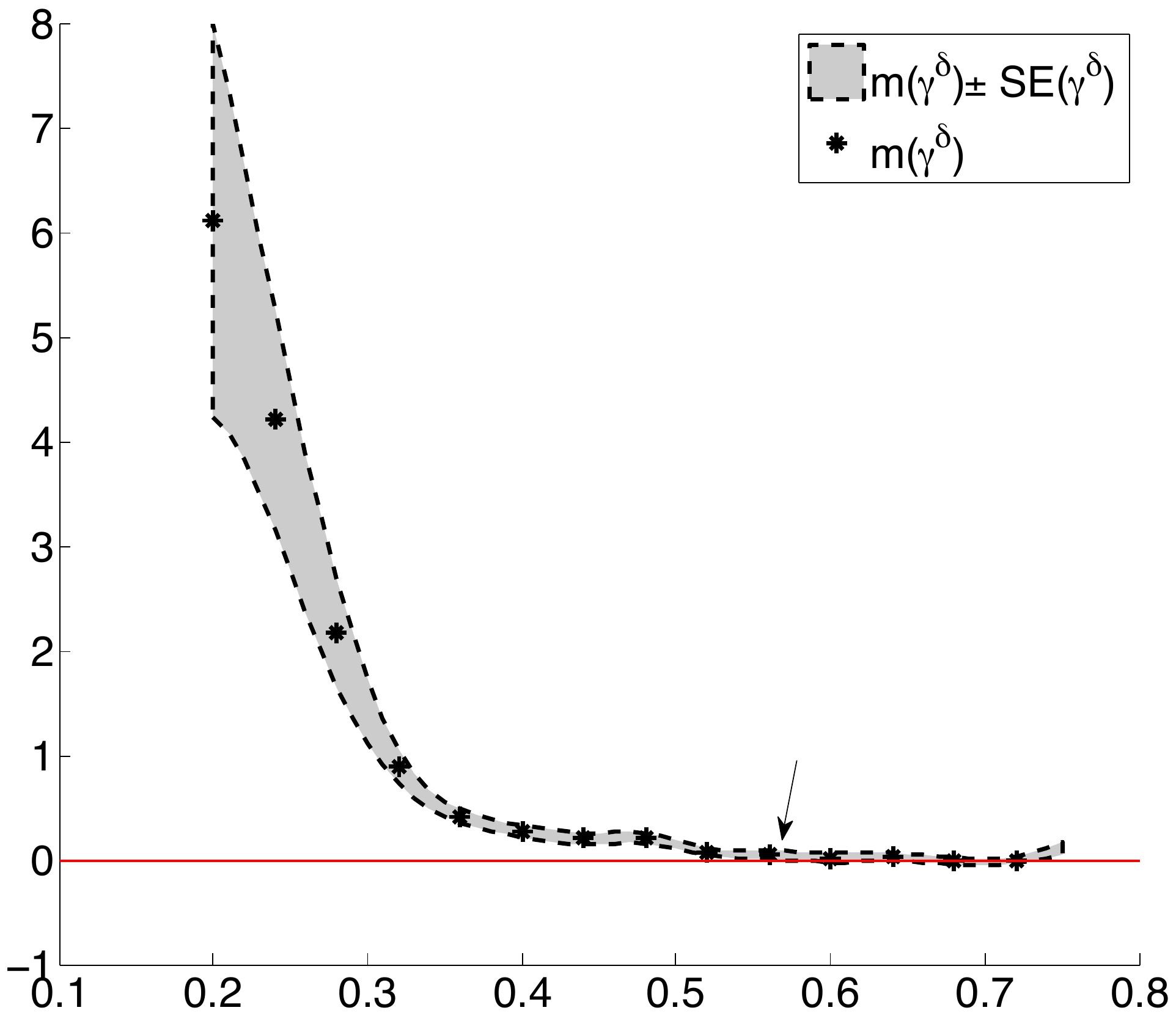}
\put(-125,-10){$\delta \equiv n/p$}
\put(-85,60){$\delta_c = 0.57$}
\caption{
Empirical kurtosis of the (rescaled) debiased Lasso estimator
$T_i=\sqrt{n}(\dth_i-\th^*_{i})/(\sigma [M\hSigma M]_{i,i}^{1/2})$.
We plot the kurtosis $m(\gamma^\delta)$ (over coordinates and $100$
independent realizations) versus $\delta$ along with the upper and
lower one standard error curves,
as a function of the number of samples per parameter $\delta$. Here,
$\varepsilon = 0.2$ and $\delta_c = 0.57$ is our empirical estimate
for the number of samples above which the debiased estimator is
approximately Gaussian.}\label{fig:eps02}
\end{figure}

\begin{figure}[]
\centering
\includegraphics*[width = 3in]{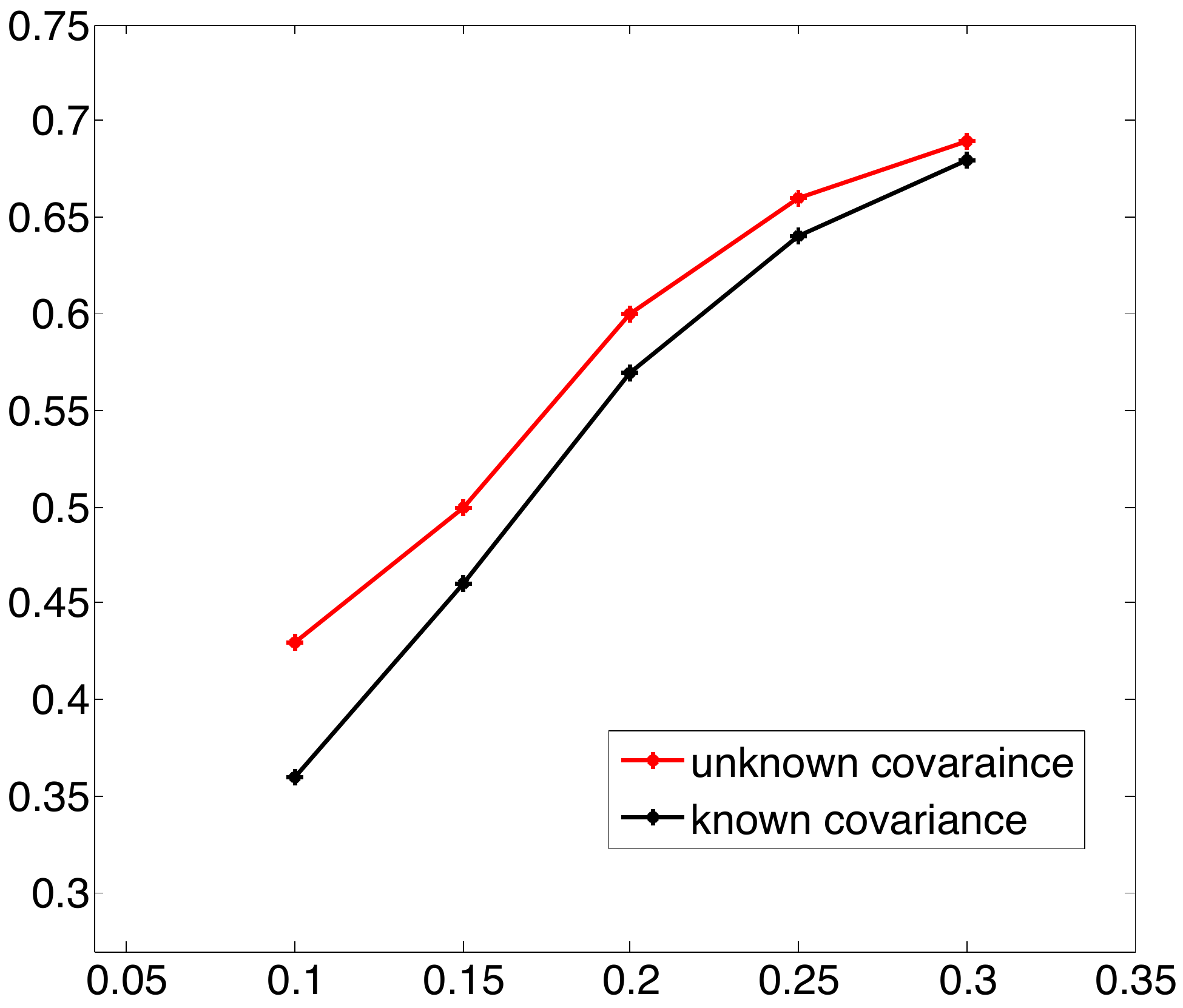}
\put(-120,-10){$\varepsilon = s_0/p$}
\put(-230,85){$\delta_c$}
\caption{Critical  number of samples per coordinate $\delta_c$, versus
  fraction of non-zero coordinates $\varepsilon$.
For $\delta>\delta_c(\eps)$ the debiased Lasso estimator is
empirically Gaussian distributed in our experiment. The approximately
linear relationship at small $\eps$ is in agreement with our theory.}\label{fig:delta_eps}
\end{figure}
\section{Minimax lower bound on the residual $R$}
\label{sec:Minimax}

In case that the design covariance matrix is unknown, Theorem~\ref{thm:unknown} establishes the following high probability bound on the residual term $R$:
\begin{align}
\|R\|_\infty\le C\rho \sigma \sqrt{\frac{s_0}{n}} \log p+  C \sigma\, \min(s_0,s_\Omega) \frac{\log p}{\sqrt{n}}\,.
\end{align}
For sparse precision matrices, such that $s_\Omega\ll \sqrt{n}/(\log p)$, the residual term $\|R\|_\infty$ vanishes asymptotically under the near optimal condition $s_0\ll n/(\log p)^2$.
The question we will study in this section is whether such condition on $s_\Omega$ is necessary.    
To answer this question, we develop a minimax lower bound on $\|R\|_\infty$. This also clarifies the  connection between our 
results and the ones of \cite{cai2015confidence}, whose general approach we build on here.

Before presenting our results we need to introduce some notations and definitions.

Consider the linear model~\eqref{eq:NoisyModel} and define parameters of the form $\gamma = (\theta,\Omega,\sigma^2)$, which consists of the signal $\theta$,
precision matrix $\Omega = \Sigma^{-1}$, and the noise standard deviation $\sigma$.

For $\alpha\in(0,1)$ and a given parameter space $\Gamma$, denote by $\cI_{\alpha}(\Gamma)$ the set of all $(1-\alpha)$-confidence intervals for 
$\theta_1$ over the entire space $\Gamma$,
\begin{align}
\cI_\alpha(\Gamma) \equiv \Big\{J_\alpha(y,X): \, \inf_{\gamma\in \Gamma} \prob_\gamma (\theta_1\in J_\alpha(y,X)) \ge 1-\alpha \Big\}\,,
\end{align}
where $\prob_\gamma$ is the induced probability distribution on $(y,X)$ for random gaussian design $X$ and noise realization $w$, given the fixed signal $\theta$. Here and below we focus on the first coordinate $\theta_1$ without loss of generality.
For a given interval $J_\alpha(y,X)\in \cI_\alpha(\Gamma)$, we let $\ell(J_\alpha(\,\cdot\,),\Gamma)$ be the maximum expected length over a parameter space $\Gamma$,
\begin{align}
\ell(J_\alpha(\,\cdot\,),\Gamma) = \sup_{\gamma\in\Gamma} \E_\gamma\{ \ell(J_\alpha(y,X))\}\,,
\end{align}
with $\E_\gamma$ expectation with respect to $\prob_\gamma$. We further define the minimax rate for the expected length of confidence intervals over $\Gamma$ as follows:
\begin{align}
\ell^*_\alpha(\Gamma) = \inf_{J_\alpha(\,\cdot\,)\in \cI_\alpha(\Gamma)} \ell(J_\alpha(\,\cdot\,),\Gamma)\,.
\end{align}
We next define parameter space $\Gamma(s_0,s_\Omega,\rho)$ as follows. 
Applying inequality~\eqref{ineq-norm}, we relax Condition~\ref{Condition:L1} as $\|\Omega\|_\infty \le \rho$
and write 
\begin{align}\label{eq:Gamma-s0}
\Gamma(s_0,s_\Omega,\rho) \equiv \bigg\{\gamma = (\theta,\Omega,\sigma^2):\, &\|\theta\|_0\le s_0, \sigma^2 \in (0,c],\nonumber\\
&(\Omega^{-1})_{ii} \le 1,\, \frac{1}{C_{\max}}< \sigma_{\min}(\Omega) \le 
\sigma_{\max}(\Omega) <\frac{1}{C_{\min}},\, \|\Omega\|_\infty \le \rho,\nonumber\\
&\max_{i\in[p]} |\{j\neq i, \Omega_{i,j} \neq 0\}| \le s_\Omega\bigg\}\,.
\end{align}
Quantities $c$, $C_{\min}$ and $C_{\max}\ge 1$ are constant which do not effect the minimax rate and therefore we have not made them explicit in our notation $\Gamma(s_0,s_\Omega,\rho)$.

\begin{propo}\label{propo:minimaxUnknown-UB} 
Consider a debiased estimator of form~\eqref{eq:debiased} with $M$ being a function of $X$ and $\hth$ the Lasso estimator at regularization parameter $\lambda$. 
Further, let $R = \sqrt{n}(M\hSigma-\id)(\hth-\tth)$ be the bias term and $Q= \diag(M\hSigma M^\sT)$ be the variance term.
Suppose that there exist a choice of $M$ and $\lambda$ such that
%
\begin{align}
&\lim_{n\to \infty} \prob\bigg(\sup \Big\{\|R\|_\infty :\,\, (\th^*,\Omega,\sigma^2)\in{\Gamma(s_0,s_\Omega, \rho)}\Big\} \le \Delta_n\bigg) = 1 \,,\label{eq:Delta1}\\
&\lim_{n\to \infty} \prob\bigg(\sup \Big\{\|Q\|_\infty :\,\, (\th^*,\Omega,\sigma^2)\in{\Gamma(s_0,s_\Omega, \rho)}\Big\} \le C \bigg) = 1 \,,\label{eq:Delta2}
\end{align}
for some known $\Delta_n$ and for some known constant $C$. Then, we have
\begin{align}
\ell^*_\alpha(\Gamma(s_0,s_\Omega,\rho)) \lesssim \frac{(1+\Delta_n)}{\sqrt{n}} \,.
\end{align}
\end{propo}
Note that since $Q$ is a function of only $X$, the arguments $\tth$ and $\sigma^2$ in Equation~\eqref{eq:Delta2} are superfluous.
To establish the above upper bound, we construct a confidence interval $J^\de_\alpha$ using a debiased estimator, such that $J^\de_\alpha\in \cI(\Gamma(s_0,s_\Omega,\rho))$. We refer to Section \ref{sec:proofminimax_U} for the proof of Proposition~\ref{propo:minimaxUnknown-UB}. 

The next proposition provides a lower bound on $\ell^*_\alpha(\Gamma(s_0,s_\Omega,\rho))$.
\begin{propo}\label{propo:minimaxUnknown-LB}
Suppose that $\alpha\in (0,1/2)$ and $s_0 \lesssim \min(p^\eta, n/\log p)$ for some constant $0\le \eta<1/2$.  Further, assume $\rho\ge 1.02$. The minimax 
expected length for $(1-\alpha)$-confidence intervals of $\theta_1$ over $\Gamma(s_0,s_\Omega,\rho)$ satisfies
\begin{align}
\ell^*_\alpha(\Gamma(s_0,s_\Omega,\rho)) \gtrsim \frac{1}{\sqrt{n}} + \min\bigg(s_0\frac{\log p}{n}, s_\Omega \frac{\log p}{n}, \rho\sqrt{\frac{\log p}{n}}\bigg)\,.
\end{align}
\end{propo}
Proposition~\ref{propo:minimaxUnknown-LB} generalizes the result of~\cite[Theorem 2]{cai2015confidence} which shows that without the sparsity constraint on $\Omega$ and the constraint $\|\Omega\|_\infty\le\rho$, the minimax rate for expected confidence interval length is lower bounded as $\ell^*_\alpha(\Gamma(s_0,p)) \ge (1/\sqrt{n} + s_0 \log p/n)$.  Proposition~\ref{propo:minimaxUnknown-LB} provides 
a more refined lower bound that takes into account the sparsity structure of the precision matrix. We refer to Section \ref{sec:proofminimax_L}
for its proof.

By comparing the upper and lower bounds on $\ell^*_\alpha(\Gamma(s_0,s_\Omega,\rho))$, we conclude that 
the condition $\min(s_0,s_{\Omega})\log p \lesssim \sqrt{n}$ is necessary for having $\|R\|_{\infty}\le \Delta_n\to 0$.
If this is not the case then $\Delta_n\gtrsim \min(s_0,s_{\Omega})\log p / \sqrt{n}$.

In particular, in order to get $\Delta_n= o(1)$ at a nearly optimal condition $s_0 \ll n/(\log p)^2$, we 
need the precision matrix to be sparse with $s_\Omega \lesssim \sqrt{n}/(\log p)$.

\section{Other applications}
\label{sec:Other}

Our main results, Theorem \ref{thm:main}  and Theorem \ref{thm:unknown} establish a Gaussian limit for the debiased  Lasso estimator. 
While our main motivation was the construction of confidence intervals for single coordinates of the parameter vector,
we want to emphasize that the Gaussian limit has other important applications.
We  illustrate this point using three examples: $(i)$ We establish a characterization of the Lasso estimator in terms of a certain 
denoising problem. $(ii)$ We develop a new thresholded Lasso estimator and provide a tight characterization of its $\ell_2$ risk.
In the case of standard Gaussian designs this approach is minimax optimal up to a factor $1+o_n(1)$. 
$(iii)$ We prove that the celebrated Stein's Unbiased Estimate of the prediction risk \cite{efron2012estimation} is consistent
in high dimension an unbiased estimator, for standard Gaussian designs.

\subsection{A probabilistic approximation result for the Lasso}

As a first consequence of our main theorem, we obtain a precise approximation result for the Lasso estimator. In order to state this
result, let $\eta_{\Sigma}:\reals^p\to\reals^p$ be defined by
\begin{align}
\eta_{\Sigma}(z)\equiv \arg\min_{\theta\in\reals^p}\Big\{\frac{1}{2}\big\|\Sigma^{1/2}(\theta-z)\big\|_2^2
+\lambda\|\theta\|_1\Big\}\, . \label{eq:EtaSigmaDef}
\end{align}
Note that the minimizer is always unique because $\Sigma$ is strictly positive definite. In the case $\Sigma=\id$, 
$\eta_{\Sigma}$ coincides with component-wise soft thresholding at level $\lambda$. More generally, $\eta_{\Sigma}(\,\cdot\,)$
can be viewed as a denoising operator associated to the problem of estimating $\theta^*$  from the noisy observation
$z = \theta^*+\tilde{w}$, where $\tilde{w}$ has covariance $\Sigma$.
Our next theorem connects the Lasso to this denoising problem.
\begin{thm}\label{thm:Approximation}
Consider the linear model~\eqref{eq:NoisyModel} where $X$ has
independent Gaussian rows, with zero mean and covariance $\Sigma$,
satisfying the assumptions of Theorem \ref{thm:main}.  Further assume the following condition:
\begin{enumerate}
\item[$(iv)$] Letting $C_*\equiv 32 C_{\max}/ C_{\min}$, we assume $\|\Sigma_{T,T^c}\|_{\infty}\le \trho$ for some constant $\trho$ and all $T\subseteq[p]$,
$|T|\le 2C_*s_0$.
\end{enumerate}
Let $\Lth = \Lth(y,X;\lambda)$ be the Lasso estimator with $\lambda =\kappa\sigma\sqrt{(\log
  p)/n}$, for  $\kappa \in [ 8,\kappa_{\max}]$.
Then, there exist constants $c, \tC$ (depending on $C_{\min}$, $C_{\max}$, $\rho$, $\trho$, $\kappa_{\max}$), such that
for $n\ge \max(25 \log p , cs_0\log(p/s_0))$, the following holds true 
with high probability.
\begin{align}
\Big\|\Lth - \eta_{\Sigma}\Big(\theta^*+\frac{1}{n}\Omega X^{\sT}w\Big)\Big\|_2^2\le \tC \sigma^2 
\Big(\frac{s_0\log p}{n}\Big)^2\, .
\end{align}
\end{thm}
Under the hypothesis of this theorem,  the Lasso $\ell_2$ error is known to be bounded as $\|\Lth-\theta^*\|^2_2\le C(s_0\log p)/n$ \cite{BickelEtAl}.
Hence, Theorem \ref{thm:Approximation} provides a characterization of the Lasso estimator that is one order of magnitude more accurate than
what available in the literature.

This characterization is particularly convenient if the population covariance has a simple structure.
For instance we obtain the following immediate corollary that characterizes the $\ell_2$ error 
for standard  designs.
\begin{coro}\label{coro:MSELasso}
Consider the linear model~\eqref{eq:NoisyModel} where $X$ has
independent Gaussian rows, with zero mean and covariance $\Sigma = \id$. 
Let $\Lth = \Lth(y,X;\lambda)$ be the Lasso estimator with $\lambda =\kappa\sigma\sqrt{(\log
  p)/n}$, for a constant $\kappa\ge 8$.
Then, for $n\ge \max(25 \log p , cs_0\log(p/s_0))$ we have
\begin{align}
\|\Lth-\theta^*\|^2_2 = \sum_{i\in \supp(\theta^*)}\E_Z\big\{\big[\eta(\theta^*_i+n^{-1/2}Z_i;\lambda)-\theta^*_i\big]^2\big\}+ O_P\left(\sigma^2\frac{\sqrt{s_0\log p}}{n}\vee\sigma^2\Big(\frac{s_0\log p}{n}\Big)^{3/2}\right)\, .
\end{align}
where expectation is taken with respect to $Z_i\sim\normal(0,1)$, and the $O_P(\,\cdot\,)$ is uniform for 
$\kappa\in[8,\kappa_{\max}]$.
\end{coro}
Let us emphasize that this is not an upper bound, but an equality up to higher order terms.
It provides a connection between the Lasso mean square error and the mean square error of soft-thresholding
denoising in the classical sequence model. A similar connection was anticipated --for instance-- in 
\cite{DMM-NSPT-11,donoho2013accurate}.
An asymptotic characterizations of the Lasso mean square error for standard Gaussian designs was first
obtained in \cite{BayatiMontanariLASSO}. However, in the present case 
we recover this as a corollary of a result for general Gaussian designs, and in a non-asymptotic form.

\subsection{Minimax optimal estimation}

The analysis in the last section suggests that it is possible to reduce the estimation 
error through a two step procedure. For the sake of simplicity, we shall assume here that $\Sigma$ is known. 
Our approach  can be extended to imperfectly known covariance by using
Theorem \ref{thm:unknown}, but we leave  this for future work.
The suggested procedure is:
\begin{enumerate}
\item[$(i)$] Compute the Lasso estimator $\Lth = \Lth(y,X;\lambda)$ with $\lambda =8\sigma\sqrt{(\log
  p)/n}$.
\item[$(ii)$] Compute the debiased estimator $\dth = \Lth +n^{-1}\Omega X^{\sT}(y-X\Lth)$.
\item[$(iii)$] Compute a new estimator $\hth^{(2)}$ by soft thresholding $\dth$ component-wise, namely
\begin{align}
\hth^{(2)}_i= \eta(\dth_i;\tau_i)\, , \;\;\;\; \tau_i = \sqrt{\frac{2\sigma^2 \Omega_{ii} \log(p/s_0)}{n}}\, .
\end{align}
Here $\eta(x;\tau) \equiv (|x|-\tau)_+\sign(x)$ is the scalar soft-thresholding function.
\end{enumerate}

Let us emphasize that in the last step we soft-threshold at a level that is smaller than 
the regularization used in the Lasso. Indeed, since $\Omega_{ii}\le C_{\min}^{-1}$, we have
$\tau_i = O(\sqrt{\log(p/s_0)/n})$, while $\lambda$ is of order $\sqrt{(\log p)/n}$.
\begin{thm}\label{thm:TwoStep}
Consider the linear model~\eqref{eq:NoisyModel} where $X$ has independent Gaussian rows, with 
zero mean and covariance $\Sigma$, satisfying the assumptions of Theorem \ref{thm:main}. Further assume $s_0\to\infty$, 
$s_0/p\to 0$ and $(s_0(\log p)^3)/n\to 0$.
Let $\hth^{(2)}$ be the two-step estimator defined above.
Then 
\begin{align}
\big\|\hth^{(2)} - \theta^*\big\|_2^2\le \frac{2s_0\sigma^2}{n}\log (p/s_0) 
\left(\frac{1}{s_0}\sum_{i\in \supp(\theta^*)}\Omega_{ii}\right)\big(1+o_P(1)\big)\, .
\label{eq:MMaxBound}
\end{align}
\end{thm}

Note that, in the case $\Sigma = \id$, the right-hand side of  (\ref{eq:MMaxBound}) is 
\emph{minimax optimal} risk, up to a factor going to one as $n,s_0,p\to\infty$ \cite{su2015slope}. Cand\'es and Su
\cite{su2015slope} recently proved that SLOPE achieves the same guarantee for Gaussian designs with $\Sigma=\id$.
On one hand, the approach of \cite{su2015slope} has the advantage of being adaptive to unknown sparsity level $s_0$.
On the other, Theorem \ref{thm:TwoStep} establishes this result as a special case of a guarantee holding for more general Gaussian
designs.

\subsection{SURE estimate of the prediction error}\label{sec:Risk}

Define the Lasso prediction error as
\begin{align}
\Risk(y,X,\theta^*) \equiv \frac{1}{n}\big\|X(\Lth-\theta^*)\big\|_2^2 + \frac{1}{n}\|w\|_2^2\, .
\end{align}
Notice that the first term is the standard prediction error, for given design matrix $X$. The second term is the 
residual error that would be present even for the perfect estimator $\hth=\theta^*$. We include this contribution for mathematical convenience, 
but it is just a constant,  independent of the estimator.

The naive empirical estimate for the prediction error is 
\begin{align}
\hRisk(y,X) \equiv \frac{1}{n}\big\|y-X\Lth\big\|_2^2\, .
\end{align}
Of course we expect the empirical risk to under-estimate the actual risk. Stein's Unbiased Risk Estimate 
(SURE)  provides a corrected estimate
\begin{align}
\hRisk_{\SURE}(y,X)\equiv \frac{1}{n}\big\|y-X\Lth\big\|_2^2 +\frac{2\sigma^2}{n}\,\|\Lth\|_0\, . \label{eq:RiskCorrect}
\end{align}
This approach has a rich history for which we can only provide a few pointers.
Donoho and Johnstone used SURE to develop an adaptive denoising procedure via wavelet thresholding. From the perspective of
linear regression, this corresponds to $X$ being proportional to an orthogonal matrix. 
Efron \cite{efron2012estimation} developed a general formula for estimating the prediction error, based on Stein's ideas,
and clarified the connection with classical model selection criteria such as Akaike's information criterion \cite{akaike1974new},
and Mallows $C_p$ \cite{mallows1973some}.
Zou, Hastie and Tibshirani \cite{zou2007degrees} showed that the number of degrees of freedom (which enters Efron's
formula) coincides with the number of non-zero parameters $\|\Lth\|_0$. They also proved that $\hRisk_{\SURE}(y,X)$
is consistent in the classical low-dimensional regime $n\to\infty$ with $p$ fixed.

To the best of our knowledge, this is the first case in which $\hRisk_{\SURE}(y,X)$ is proved to be consistent in high dimension
(although in a restricted setting, namely for Gaussian designs).
\begin{thm}\label{thm:RiskEst}
Consider the linear model~\eqref{eq:NoisyModel} where $X$ has
independent Gaussian rows, with zero mean and identity covariance
$\Sigma=\id$.
Let $\Lth = \Lth(y,X;\lambda)$ be the Lasso estimator with $\lambda \ge 9\sigma\sqrt{(\log
  p)/n}$. 
If $n,p\to\infty$ with $s_0 = o(n/(\log p)^2)$, then there exists $\eps_n\to 0$ as $n\to\infty$, such that
 the following holds with probability at least $1-e^{-ct^2}-o_n(1)$:
\begin{align}
\big|\hRisk_{\SURE}(y,X)- \Risk(y,X,\theta^*) \big|&  \le \frac{t\sigma^2}{\sqrt{n}} + \frac{s_0\sigma^2\eps_n}{n} 
\, . \label{eq:ExcessRisk}
\end{align}
\end{thm}

Let us emphasize a few important points:
\begin{itemize}
\item The error bound in Eq.~(\ref{eq:ExcessRisk}) is of smaller order with respect to the correction in
(\ref{eq:RiskCorrect}) which typically is of order $s_0\sigma^2/n$.
\item The SURE risk estimate $\hRisk_{\SURE}(y,X)$ is perfectly well defined for arbitrary design covariance $\Sigma$.
\item While our proof applies to standard designs, $\Sigma = \id$, we expect the conclusion of Theorem \ref{thm:RiskEst}
to hold more generally. This is also confirmed by the simulations discussed below.
\end{itemize}

In Figure \ref{fig:risks}, we present the results of a numerical
simulation with $p = 5000$, $n=1800$. We choose a subset $S\subseteq[p]$ of size $s_0 = |S| = 100$ uniformly at random
and set $\theta^*_{0,i}=0.1$ if $i\in S$ and $\th^*_{0,i} = 0$,
otherwise.  The design matrix $X$ has i.i.d random rows  $x_i\sim \normal(0,\Sigma)$
with $\Sigma_{ij} = r^{|i-j|}$. We set $r = 0.1$ to illustrate a case of low correlation between predictors and $r = 0.9$ for a case of high
correlation. 
In our simulations, we replace the noise level $\sigma$ appearing in Eq.~(\ref{eq:RiskCorrect}) with an estimate 
$\hsigma$, obtained as follows. 
We first run scaled Lasso and then perform least square after model selection to mitigate the estimation bias.
More precisely, we use the {\sf R}-package {\sf scalreg} with the default value for the regularization parameter in the scaled Lasso cost function.
This selects a model $\widehat{S}$. 
We then perform least square on $\widehat{S}$ to obtain an estimate $\hth^{\rm LS}$. The noise variance is computed as
$\hsigma = \|y-X\hth^{\rm LS}\|_2/\sqrt{n}$.

The agreement between $\hRisk_{\SURE}(y,X)$  and $\Risk(y,X,\theta^*)$ is excellent.

Let us mention that \cite{bayati2013estimating}  also studied estimators similar to $\hRisk_{\SURE}(y,X)$,
and related ideas were developed in \cite{obuchi2015cross} on the basis of non-rigorous but insightful statistical mechanics 
techniques.
Other approaches to the risk estimation, e.g. \cite{cai2016accuracy}, are based on sample-splitting, 
which has complementary shortcomings.

\begin{figure}[t]
    \centering
    \subfigure[$r=0.1$]{
        \includegraphics[width = 3.1in]{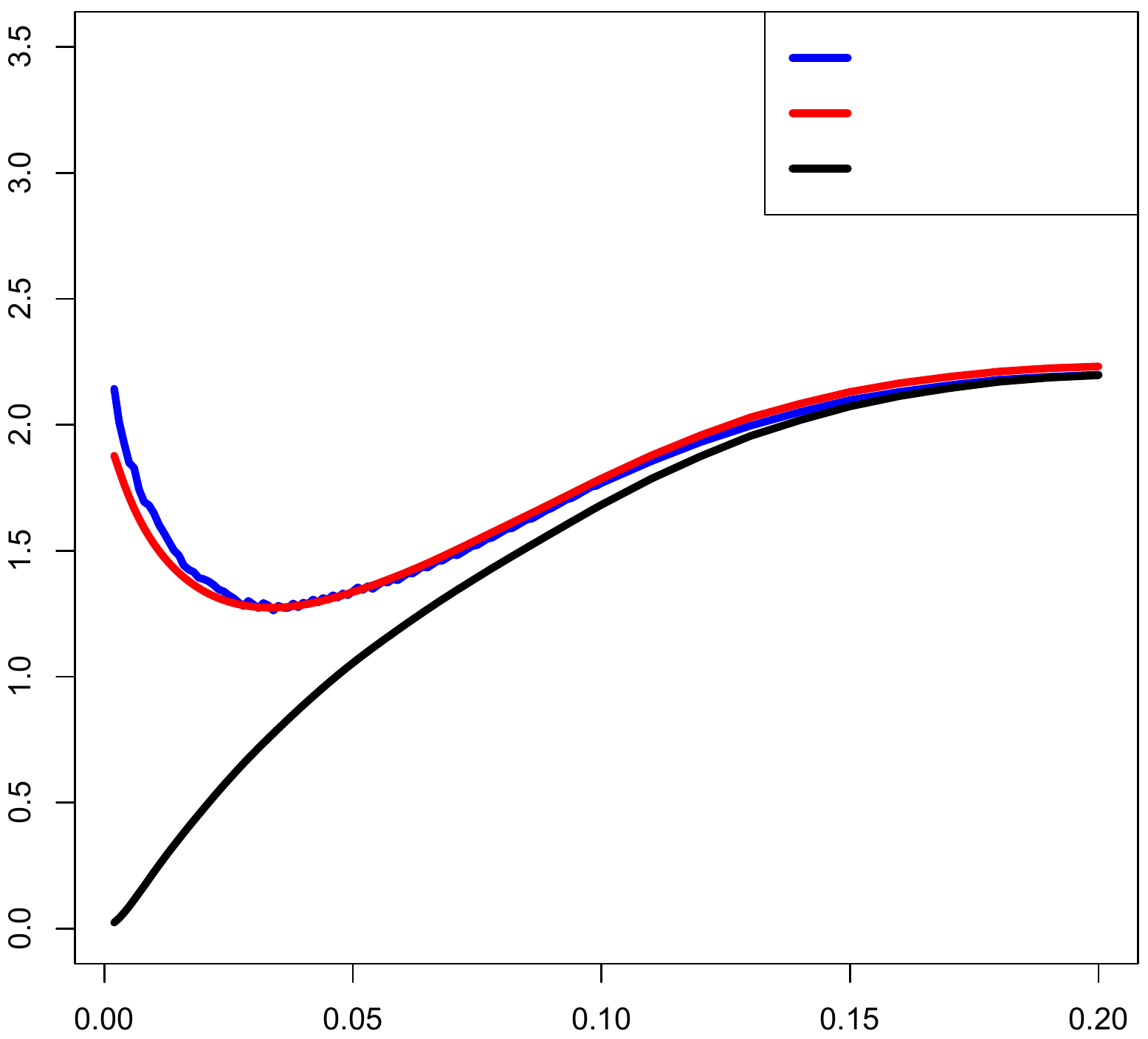}
         \put(-50,190){{\scriptsize $\hRisk_{\SURE}(y,X)$}}
        \put(-50,179){{\scriptsize $\Risk(y,X,\tth)$}}
        \put(-50,168){{\scriptsize $\hRisk(y,X)$}}
        \put(-105,-13){$\lambda$}
        \put(-105,-23){}
        }\hspace{0.2cm}
    \subfigure[$r=0.9$]{
        \includegraphics[width=3.1in]{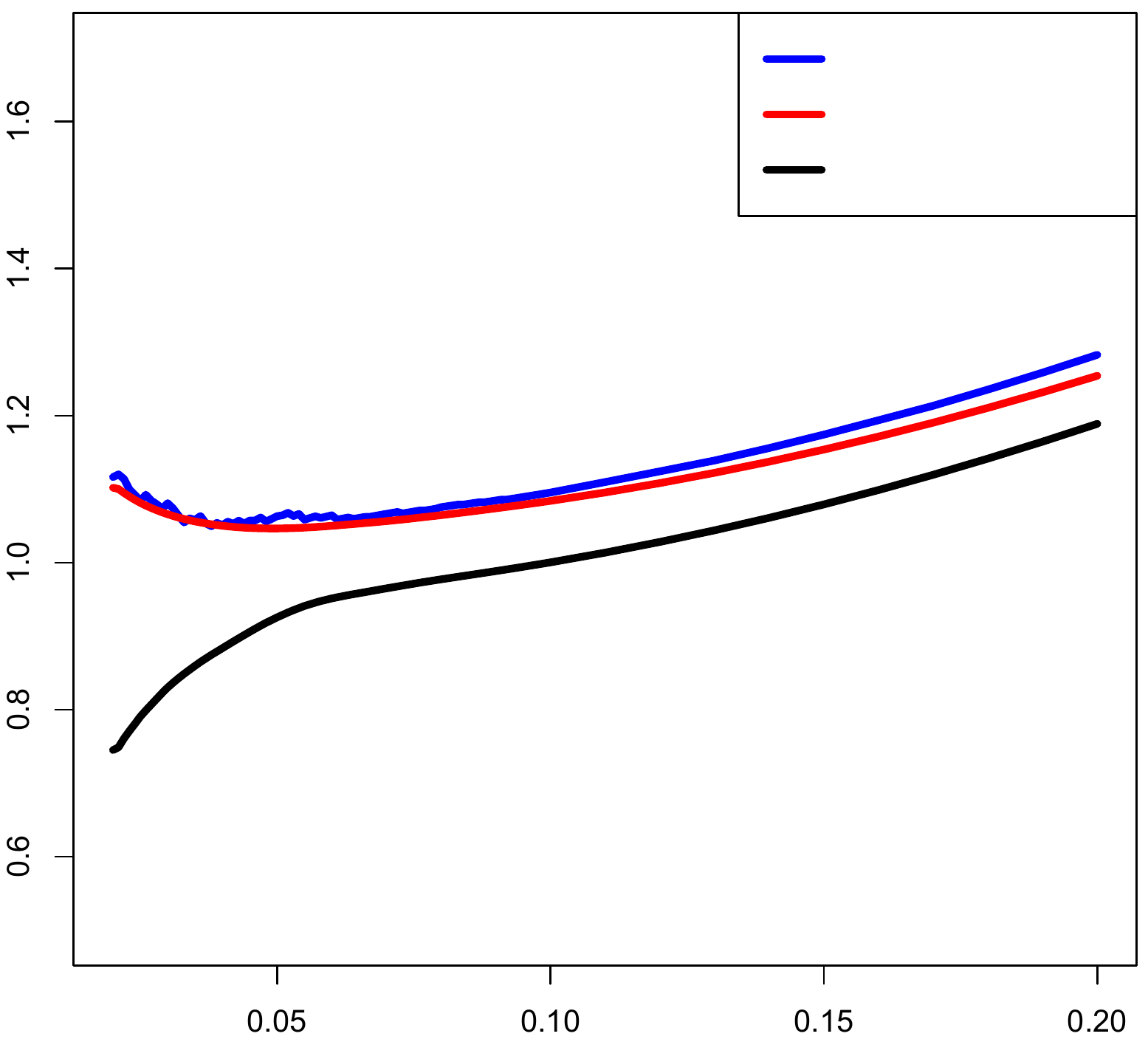}
         \put(-50,190){{\scriptsize $\hRisk_{\SURE}(y,X)$}}
        \put(-50,179){{\scriptsize $\Risk(y,X,\tth)$}}
        \put(-50,168){{\scriptsize $\hRisk(y,X)$}}
        \put(-105,-13){$\lambda$}
        \put(-105,-23){}
        }
    \caption{Lasso prediction error $\Risk(y,X,\tth)$, empirical prediction error $\hRisk(y,X)$, and SURE  estimator $\hRisk_{\SURE}$ curves versus $\lambda$ for the simulation setting described in Section~\ref{sec:Risk}.}\label{fig:risks}
\end{figure}

\section{Proof of Theorem~\ref{thm:main} (known covariance)}\label{sec:thm_main}
\subsection{Outline of the proof}
\label{sec:Outline}
Fix arbitrary integer $i \in [p]$. In our analysis, we focus on the $i$-th coordinate $\theta^*_i$, and then discuss how the argument can be adjusted to apply to all the coordinates simultaneously.
Our argument relies on a perturbation analysis. We let $\hth^\pe$ be the Lasso estimator when one forces $\hth^\pe_i = \theta^*_i$.
With a slight abuse of notation, we use the representation $\th = (\theta_i,\thnz)$.\footnote{Or without loss of generality one can assume $i=1$.}
 Adopting this convention, we have $\hth^\pe = (\theta^*_i,\pthnz)$ where
\begin{align}\label{eq:pthnz}
\pthnz &= \arg\min_{\th} \cL_{y,X}(\theta^*_i,\th)\,.
\end{align}
Throughout, we make the convention that $\cL_{y,X}(\theta^*_i,\theta) \equiv \cL_{y,X}((\theta^*_i,\theta))$.

We observe that $\pthnz$ can be written as a Lasso estimator. Specifically, by definition of Lasso cost function we have 
$$\cL_{y,X}(\theta^*_i,\theta) = \frac{1}{2n} \|y-\tx_i\theta^*_i - X_{\sim i} \theta\|_2^2 + \lambda |\theta^*_i| + \lambda \|\theta\|_1\,.$$
Letting $\ty \equiv y - \tx_i \theta^*_i = w+ \Xnz \tthnz$, we obtain
\begin{align}\label{eq:pLasso}
\pthnz = \arg\min_{\th} \cL_{\ty,X_{\sim i}}(\theta)\,.
\end{align}
%
%

Let $v_i = X\Omega e_i$ and expand $\hth^\de_i - \th^*_i$ as follows:
\begin{align}
\sqrt{n}(\hth^\de_i-\th^*_i) &\equiv \sqrt{n} \hth_i + \frac{1}{\sqrt{n}} e_i^\sT\Omega X^\sT(y - X \hth) -\sqrt{n}\theta^*_i\nonumber\\
& = \sqrt{n}\hth_i + \frac{v_i^\sT}{\sqrt{n}} \Big[w + \tx_i (\theta^*_i-\hth_i) + \Xnz (\tthnz - \hthnz)  \Big] -\sqrt{n}\theta^*_i\nonumber\\
& = \sqrt{n}\Big(1-\frac{1}{n}\<v_i,\tx_i\>\Big) (\hth_i-\theta^*_i) + \frac{v_i^\sT}{\sqrt{n}} \Big[w+\Xnz (\tthnz-\hthnz) \Big]\,.\label{eq:decompose}
\end{align}
We decompose the above expression into the following terms:
\begin{eqnarray}
\begin{split}\label{eq:T}
Z_{i}&\equiv  \frac{v_i^\sT \bw}{\sqrt{n}}\,,  \\
R^{(1)}_{i}&\equiv \sqrt{n}\Big(1-\frac{\<v_i,\tx_i\>}{n}\Big) (\hth_i-\theta^*_i)\,,\\
R^{(2)}_{i}&\equiv \frac{v_i^\sT}{\sqrt{n}} \Xnz(\tthnz - \pthnz)\,,\\
R^{(3)}_{i}&\equiv  \frac{v_i^\sT}{\sqrt{n}}  \Xnz(\pthnz - \hthnz)\,.
\end{split}
\end{eqnarray}
The bulk of the proof consists in treating each of the terms above separately. Term $Z_i$ gives the Gaussian component $Z$ in equation~\eqref{eq:components}. For bounding $R^{(2)}_i$, note that $\pthnz$ is a deterministic function of $(\ty,\Xnz)$ (and thus a deterministic function of 
$(w,\Xnz)$) by Equation~\eqref{eq:pLasso}.  Further, $v_i$ is independent of $\Xnz$, as per Lemma~\ref{lem:independence}, and independent of noise $w$.  Hence, $v_i$ is independent of $ \Xnz(\tthnz - \pthnz)$. Bounding $R^{(3)}_i$ relies on a perturbation analysis showing that the solutions of Lasso $\hth$ and its perturbed form $\hth^{\pe}$, are close to each other.

\subsection{Technical steps}
Let $Z= (Z_i)_{1\le i\le p}$. We rewrite $Z$ as
\[Z = \frac{1}{\sqrt{n}} \Omega X^\sT w\,.\]
Since $\bw\sim\normal(0,\sigma^2 \id)$ is independent of $X$, we get
\[Z|X \sim \normal(0,\sigma^2 \Omega \hSigma \Omega )\,.\]

Let $R^{(1)} = (R^{(1)}_i)_{i=1}^p, R^{(2)} = (R^{(2)}_i)_{i=1}^p, R^{(3)} = (R^{(3)}_i)_{i=1}^p\in \reals^p$.
In the following, we provide a detailed analysis to control the terms $R^{(1)}, R^{(2)}, R^{(3)}$.
\bigskip

\noindent $\bullet$ \emph{Bounding term $R^{(1)}$:} 
Recalling the definition $v_i = X\Omega e_i$, we write
\[R^{(1)}_i = \sqrt{n}\Big(1-\frac{1}{n} e_i^\sT\Omega X^\sT Xe_i \Big) (\hth_i - \theta^*_i)\,.\]
Therefore,
\begin{align*}
\|R^{(1)}\|_\infty \le  \sqrt{n}|\id - \Omega \hSigma|_\infty  \|\hth-\theta^*\|_2 \,.
\end{align*}
 For $A>0$, let $\cG_n = \cG_n(A)$ be the event that 
\begin{align}\label{eq:GnA}
\cG_n(A) \equiv \Big\{X\in \reals^{n\times p}:\,|\Omega \hSigma -\id|_\infty \le A\sqrt{\frac{\log p}{n}} \Big\}\,.
\end{align}
Using the result of~\cite[Lemma 6.2]{javanmard2014confidence} for $n\ge (A^2 C_{\min})/ (4e^2C_{\max})\log p$ we have
\begin{align*}
\prob(X \in \cG_n(a)) \ge 1- 2p^{-c}\,, \quad \quad c = \frac{A^2 C_{\min}}{24e^2 C_{\max}}-2\,. 
\end{align*}
By choosing $A\equiv 10e \sqrt{C_{\max}/C_{\min}}$ we get $c\ge 1$. Therefore, provided that $n\ge 25 \log p$, 
\begin{align}\label{eq:Gn}
\prob(X \in \cG_n(A)) \ge 1- 2p^{-1}\,. 
\end{align}
%

In addition, on the event $\cB\equiv\cB_\delta(n,s_0,3)\cap \ctB(n,p)$ we have~\cite{buhlmann2011statistics}
\[\|\hth-\th^*\|_2 \le \frac{\sqrt{20}}{(1-\delta)^2C_{\min}} \lambda\sqrt{s_0} \,. \]

Combining the above bounds, we obtain that on event $\cG_n(A) \cap \cB$,
\begin{align}\label{eq:R1}
\|R^{(1)}\|_\infty  \le \frac{5\kappa A\sigma}{(1-\delta)^2C_{\min}} \sqrt{\frac{s_0}{n}} \log p\,.
\end{align}
%
%

\bigskip

\noindent $\bullet$ \emph{Bounding term $R^{(2)}$:} 
%
To lighten the notation, we define 
\begin{align}\label{eq:zeta}
\zeta_i \equiv \frac{1}{\sqrt{n}} \Xnz(\tthnz - \pthnz)\,.
\end{align}
As discussed $\pthnz$ is a Lasso estimator with design matrix $X_{\sim i}$ and response vector $\ty = y - \tx_i \theta^*_i$, as per equation~\eqref{eq:pLasso}. We recall the following results on the prediction error of the Lasso estimator, which bounds $\|\zeta_i\|_2$.
\begin{propo}[\cite{buhlmann2011statistics}, Theorem 6.1]\label{propo:prediction_error}
Let $S\equiv \supp(\theta^*_{\sim i})$. Then on the event $\ctB(n,p)$, we have for $\lambda \ge 8\sigma\sqrt{(\log p)/n}$, 
$$\|\zeta_i\|_2^2 \le \frac{4 \lambda^2 |S|}{\phi^2(S,\hSigma_{\sim i,\sim i})}\,.$$
\end{propo}
From the definition of the compatibility constant (cf. Definition~\ref{def:phi}), it is clear that 
$\phi^2(S,\hSigma_{\sim i, \sim i}) \ge \phi^2(S,\hSigma)$.
%
Therefore, combining Proposition~\ref{propo:prediction_error} and Remark~\ref{rem:phi-bound}, we arrive at the following 
corollary:
\begin{coro}\label{coro:zeta}
On the event $\cB\equiv\cB_\delta(n,s_0,3) \cap \ctB(n,p)$, we have for $\lambda \ge 8\sigma \sqrt{(\log p)/n}$,
$$\|\zeta_i\|_2^2 \le \frac{4\lambda^2 s_0}{(1-\delta)^2C_{\min}}\,.$$
\end{coro}
Employing Corollary~\ref{coro:zeta}, we derive a tail bound on $R^{(2)}_i$. 

For $i\in [p]$ define the event 
\begin{eqnarray}\label{eq:eventE}
\event_i \equiv \bigg\{\|\zeta_i\|_2^2\le \frac{4 \lambda^2 s_0}{(1-\delta)^2C_{\min}}\bigg\}\,.
\end{eqnarray}

By Corollary~\ref{coro:zeta}, we have $\cB \subseteq \event_i$ for $i\in [p]$. Hence,
for any value $t>0$
\begin{align*}
\prob \Big(\|R^{(2)}\|_\infty \ge t; \cB\Big)&\le
\prob \Big(\max_{i\in [p]} |v_i^\sT \zeta_i| \ge t; \event_i\Big)\\
&\le p\max_{i\in [p]} \E\Big\{\ind(|v_i^\sT \zeta_i| \ge t)\cdot \ind(\event_i)\Big\}\\
&\le 2p\max_{i\in [p]}\E\Big(\exp\Big[-\frac{t^2}{2 \Omega_{ii} \|\zeta_i\|^2}\Big]\cdot \ind(\event_i) \Big)\\
&\le 2p\exp\Big(-\frac{c_*t^2}{s_0\lambda^2 \Omega_{ii}} \Big)\,,
\end{align*}
with $c_* \equiv (1-\delta)^2C_{\min}/8$. In the third inequality, we applied Fubini's theorem, and first integrate w.r.t $v_i$ and then w.r.t $\zeta_i$ using the fact that $v_i$ and $\zeta_i$ are independent. Note that $v_i \sim \normal(0,\Omega_{ii}\id_{n\times n})$ and thus $v_i^\sT \zeta_i |\zeta_i \sim \normal(0,\Omega_{ii}\|\zeta_i\|^2)$. Further, on the event $\event_i$, $\|\zeta_i\|^2$ can be bounded as in Equation~\eqref{eq:eventE}.

Setting $t\equiv \kappa \sigma \sqrt{2 s_0/(c_*C_{\min} n)} \log p$, we get
\begin{align}\label{eq:R2}
\prob \Big(\|R^{(2)}\|_\infty \ge \kappa \sigma\sqrt{\frac{2s_0}{c_*C_{\min}n}}\log p; \cB\Big) \le 2p^{-1}\,.
\end{align}
%


%
%
%
  \bigskip

\noindent $\bullet$ \emph{Bounding term $R^{(3)}$:} 
%
%
In order to bound the last term, we first need to establish the following main lemma that bounds the distance between
Lasso estimator and the solution of the perturbed problem. We refer to Section~\ref{app:lem_main} for the proof of Lemma~\ref{lem:main}.
\begin{lemma}[Perturbation bound]\label{lem:main}
Suppose that $\Sigma_{ii} \le1$, for $i\in [p]$.
Set $\lambda = 8\sigma \sqrt{(\log p)/n}$ and let 
$\cB(C_\delta) \equiv \ctB(n,p)\cap \cB_\delta(n,C_\delta s_0,3)$. The following holds true.
\begin{eqnarray}
\prob\Big(\|\hthnz - \pthnz\|_2 \ge C' \lambda;\cB(C_\delta)\Big) \le 2\exp \Big(-\frac{c_*n}{s_0}\Big) + \exp\Big(-\frac{n}{1000}\Big)\,,
\end{eqnarray}
where, 
\begin{align*}
C' \equiv &\frac{24\rho (1+\delta)\sqrt{C_{\max}}}{(1-\delta)^2 C_{\min}}\,,\quad \quad c_* \equiv \frac{1}{8}(1-\delta)^2 C_{\min}\,,\\
&\quad \quad \quad \quad  \quad C_\delta \equiv \frac{16 C_{\max}}{(1-\delta)^2 C_{\min}}+1\,.
\end{align*}
\end{lemma}
We are now ready to bound term $R^{(3)}$.

%
%
\begin{align*}
|R^{(3)}_i| &\le  \frac{1}{\sqrt{n}} \|v_i^\sT \Xnz\|_{\infty} \|\pthnz - \hthnz\|_1 \\
&\le \sqrt{\frac{C_\delta s_0}{n}} \|v_i^\sT \Xnz\|_\infty \|\pthnz - \hthnz\|_2\\
&\le \sqrt{C_\delta s_0 n}\, |\Omega \hSigma - \id|_\infty \|\pthnz - \hthnz\|_2\,,
\end{align*}
where in the first inequality we used Lemma~\ref{lem:Bickel}, which implies that $\|\pthnz - \tthnz\|_0 \le C_\delta s_0$, under $\cB$. 
Therefore, by Lemma~\ref{lem:main} and equation~\eqref{eq:Gn} and since $\cB(C_\delta) \subseteq \cB$, we have
\begin{eqnarray*}
\prob\Big(|R^{(3)}_i| \ge C''\sigma\sqrt{\frac{s_0}{n}}\log p ;\cB(C_\delta) \Big) \le 2\exp \Big(-\frac{c_*n}{s_0}\Big) + \exp\Big(-\frac{n}{1000}\Big) + 2p^{-2} \,, 
\end{eqnarray*}
with ${C''}\equiv \kappa \sqrt{(C_*+1)} A C'$. Hence, by union bound over the $p$ coordinates, we get
\begin{eqnarray}\label{eq:T4_final}
\prob\Big(\|R^{(3)}\|_\infty \ge {C''}\sigma\sqrt{\frac{s_0}{n}}\log p ;\cB(C_\delta) \Big) \le 2p\exp \Big(-\frac{c_*n}{s_0}\Big) + p\exp\Big(-\frac{n}{1000}\Big) + 2p^{-1} \,.
\end{eqnarray}

We are now in position to prove the claim of Theorem~\ref{thm:main}.

Using equations~\eqref{eq:decompose} and \eqref{eq:T}, we have $\sqrt{n}(\hth^\de - \theta^*) = Z + R$, where $Z|X \sim \normal(0,\sigma^2 \Omega \hSigma \Omega)$ and $R = R^{(1)}+R^{(2)}+R^{(3)}$. 
Combining equations~\eqref{eq:R1}, \eqref{eq:R2} and \eqref{eq:T4_final}, we get
\begin{align}\label{eq:prob1}
\prob\Big(\|R\|_\infty \ge C \sqrt{\frac{s_0}{n}}\log p; \cG_n(A) \cap \cB(C_\delta) \Big)\le 
2p\exp \Big(-\frac{c_*n}{s_0}\Big) + p\exp\Big(-\frac{n}{1000}\Big) + 4p^{-1}\,,
\end{align}
where $C$ is given by
\begin{align}\label{eq:C}
C \equiv \kappa\sigma\Big(\frac{5A}{(1-\delta)^2C_{\min}} + \sqrt{\frac{2 }{c_*C_{\min}}} + \sqrt{C_\delta}AC'\Big)\,.
\end{align}
Further, for $n\ge \max(25\log p, c_1C_\delta s_0\log (p/s_0))$, we have
\begin{align}
\prob\Big((\cG_n(A) \cap \cB(C_\delta))^c \Big) &\le \prob(\cG_n(A)^c) + \prob(\ctB(n,p)^c) + \prob(\cB_\delta(n,C_\delta s_0,3)^c)\nonumber\\
&\le 2p^{-1}+2p^{-1}+2e^{-\delta^2n} = 4p^{-1}+ 2e^{-\delta^2n}\,,\label{eq:prob2}
\end{align}
where we used bound~\eqref{eq:Gn}, Lemma~\ref{lem:rudelson} and Lemma~\ref{lem:ctB}. 

The result follows from equations~\eqref{eq:prob1} and \eqref{eq:prob2}, and setting $\delta = 1-1/\sqrt{2}$.

\subsection{Proof of Lemma~\ref{lem:main} (perturbation bound)}
\label{app:lem_main}

\begin{lemma}\label{lem1}
 For all $\th\in \reals^{p-1}$ the following holds true.
\begin{align}
\frac{1}{2n}\|\Xnz(\th - \pthnz)\|_2^2\le \cL_{y, X} (\th^*_i,\th) - \cL_{y,X}(\theta^*_i, \pthnz)\,.
\end{align}
\end{lemma}
Lemma~\ref{lem1} is proved in Appendix~\ref{app:lem1}.

\begin{lemma}\label{lem:f_k}
Let $f_k(x) = \frac{c}{2}(x-a-u_k)^2 + \lambda |x|+b_k$ for $k=1,2$. Further assume that $\min_x f_1(x)\le \min_x f_2(x)$. Then,
\begin{align}
f_1(a)-f_2(a) \le (c|u_2|+\lambda) |u_1-u_2| + \frac{c}{2}(u_1-u_2)^2\,.
\end{align}
\end{lemma}
Lemma~\ref{lem:f_k} is proved in Appendix~\ref{app:aux0}.
%
%
\begin{lemma}\label{lem:aux1}
For $\th\in \reals^{p-1}$ define
\begin{eqnarray}\label{eq:u}
u(\th) \equiv \frac{\tx_i^\sT(\bw+\Xnz(\tthnz - \th))}{\|\tx_i\|^2}
\end{eqnarray}
Also let $c_i \equiv \|\tx_i\|^2/n$.  Then, the following relation holds true.
\begin{align}
\cL(\theta_i,\th)  = \lambda |\theta_i| +\frac{c_i}{2}(\th_i - \th^*_i -u(\th))^2-\frac{c_i}{2}u(\th)^2+ \cL(\th^*_i,\th)-\lambda |\th^*_i|\,.\label{eq:expansion-nn}
\end{align}
\end{lemma}
Lemma~\ref{lem:aux1} is proved in Appendix~\ref{app:aux1}.

We let $f_1(x) =\cL (x,\hthnz)$ and $f_2(x) = \cL(x,\pthnz)$. 
Note that $(\hth_i,\hthnz)$ is the minimizer of $\cL_{y,X}(\theta_i,\theta_{\sim i})$.
Therefore, $\min f_1(x) = \cL(\hth_i,\hthnz) \le \min f_2(x)$.
Using decomposition~\eqref{eq:expansion-nn} and applying 
Lemma~\ref{lem:f_k} with
\begin{align}
&c = c_i, \quad a = \theta^*_i ,\quad u_1 = u(\hthnz), \quad u_2 = u(\pthnz),\\
&b_1 = -\frac{c_i}{2}u(\hthnz)^2+ \cL(\th^*_i,\hthnz)-\lambda |\th^*_i|,\\
&b_2 = -\frac{c_i}{2}u(\pthnz)^2+ \cL(\th^*_i,\pthnz)-\lambda |\th^*_i|\,,
\end{align}
we obtain 
\begin{align}\label{eq:L-L}
\cL(\theta^*_i,\tthnz) - \cL(\theta^*_i,\pthnz) \le (c_i |u(\pthnz)|+\lambda) |u(\hthnz)-u(\pthnz)| +
\frac{c_i}{2} (u(\hthnz)-u(\pthnz))^2\,.
\end{align}
We next write 
\begin{align}
\frac{c_i}{2} (u(\hthnz)-u(\pthnz))^2 = \frac{1}{2n} (\hthnz -\pthnz)^\sT \Xnz^\sT \tx_i \tx_i^\sT \Xnz (\hthnz -\pthnz)
 = \frac{1}{2n} \|\proj_{\tx_i}  \Xnz (\hthnz -\pthnz)\|_2^2\,,
\end{align}
where $\proj_{\tx_i} \equiv \tx_i\tx_i^\sT/\|\tx_i\|^2$ denotes the projection on the direction of $\tx_i$. 

We lower bond the left-hand side of Equation~\eqref{eq:L-L} using Lemma~\ref{lem1} and employing the above identity
to get
\begin{align}\label{eq:proj-B}
\frac{1}{2n}\|\proj^\perp_{\tx_i}\Xnz(\hthnz - \pthnz)\|_2^2 \le  (c_i |u(\pthnz)|+\lambda) |u(\hthnz)-u(\pthnz)| \,.
\end{align}
Next preposition bounds $c_i |u(\pthnz)|$.
We defer the proof of Proposition~\ref{propo:F'main} to Appendix~\ref{app:F'main}.
\begin{propo}\label{propo:F'main}
Let $\cB \equiv \ctB(n,p) \cap \cB_\delta(n,s_0,3)$, where the events $\cB_\delta(n,s_0,3)$ and $\ctB(n,p)$ are given as per equations \eqref{eq:Bdelta} and \eqref{eq:ctB}. The following holds true.
\[ \prob\Big(|c_i u(\pthnz)| \ge 1.25\rho \lambda; \cB\Big) \le 2\exp \Big(-\frac{c_*n}{s_0}\Big)\,. \]
%
where
$c_* \equiv (1-\delta)^2 C_{\min}/8$.
\end{propo}
We further have
\begin{align}
|u(\hthnz) - u(\pthnz)| = \frac{|\tx_i\Xnz(\pthnz-\hthnz)|}{\|\tx_i\|^2}
\le \frac{\|\Xnz(\pthnz-\hthnz)\|}{\|\tx_i\|}\,.
\end{align}
We next upper bound the term $\|\Xnz(\hthnz - \pthnz)\|$.
%

The corollary below follows from Proposition~\ref{lem:Bickel} and its proof is given in Appendix~\ref{proof:coroBickel}.

\begin{coro}\label{coro:Bickel}
Set $\lambda = \kappa\sigma\sqrt{(\log p)/n}$, for a constant $\kappa\ge 8$. On the event $\cB(C_*) \equiv \ctB(n,p)\cap \cB_\delta(n,(C_*+1)s_0,3)$, the following holds.
\begin{align}\label{eq:RHS_B}
\frac{1}{n}\|\Xnz(\hthnz - \pthnz)\|^2 \le (1+\delta)^2 C_{\max} \|\hthnz - \pthnz\|^2\,. 
\end{align}
\end{coro}

Corollary~\ref{coro:Bickel} is proved in Appendix~\ref{proof:coroBickel}.

%
%
We next lower bound $\|\tx_i\|_2$. Observe that the entries $\tx_{i\ell}^2-1$, $\ell\in [n]$, are zero-mean sub-exponential random variables.
We obtain the following tail-bound inequality by applying Bernstein-type inequality for sub-exponential random variables. (See e.g.~\cite[Equation (190)]{javanmard2013hypothesis}.) 
\begin{align}\label{eq:tx0_B}
\prob\Big(\|\tx_i\| > \frac{\sqrt{n}}{5}\Big) \le e^{-n/1000}\,.
\end{align}

Combining the results of Proposition~\eqref{propo:F'main} and equations~\eqref{eq:RHS_B} and \eqref{eq:tx0_B}, we obtain that on event $\cB$, with probability 
at least $1-e^{-n/1000} - 2e^{-c_*n/s_0}$, the following holds:
\begin{align}\label{eq:z2}
\frac{1}{2n}\|\proj^\perp_{\tx_i}\Xnz(\hthnz - \pthnz)\|_2^2 \le 12\rho(1+\delta)\sqrt{C_{\max}} \lambda \|\hthnz - \pthnz\|
\end{align}
%
The last step is to lower bound the left-hand side of Equation~\eqref{eq:z2}. Write
\begin{align*}
\pproj_{\tx_i} \Xnz(\hthnz - \pthnz) &= \Xnz(\hthnz - \pthnz) - \proj_{\tx_i} \Xnz(\hthnz - \pthnz)\\
&=\Xnz(\hthnz - \pthnz) - \tx_i \Big\<\frac{\tx_i}{\|\tx_i\|^2}, \Xnz(\hthnz - \pthnz)\Big\>\,.
\end{align*}
Define vector $\mu\in \reals^p$ with 
$$\mu_i \equiv -\Big\<\frac{\tx_i}{\|\tx_i\|^2}, \Xnz(\hthnz - \pthnz)\Big\>\,, \quad \quad \mu_{\sim i} = \hthnz - \pthnz\,.$$
Then $\mu \in \cone(C_\delta s_0,3)$, by Proposition~\ref{lem:Bickel}, with $C_\delta = C_*+1$ . Hence, on the event $\cB(n,C_\delta s_0,3)$, we have
\begin{align}\label{eq:LHS_B}
\frac{1}{2n}\|\pproj_{\tx_i} \Xnz(\hthnz - \pthnz)\|^2 &= \frac{1}{2n} \|X\mu\|^2 \nonumber\\
&\ge \frac{1}{2}(1-\delta)^2C_{\min} \|\mu\|^2 \nonumber\\
&\ge \frac{1}{2}(1-\delta)^2 C_{\min} \|\hthnz - \pthnz\|^2\,.
\end{align}
Finally, note that $\cone(s_0,3) \subseteq \cone(C_\delta s_0,3) $, since $C_\delta \ge1$. Therefore, $\cB_\delta(n,C_\delta s_0,3)\subseteq \cB_\delta(n,s_0,3)$, by definition. 
Letting $\cB(C_\delta) \equiv \ctB(n,p)\cap \cB(n,C_\delta s_0,3)$, we have $\cB(C_\delta) \subseteq \cB$.
Combining equations~\eqref{eq:z2} and \eqref{eq:LHS_B}, we obtain
\begin{eqnarray}
\prob\Big(\|\hthnz - \pthnz\| \ge \frac{24\rho (1+\delta)\sqrt{C_{\max}}}{(1-\delta)^2 C_{\min}} \lambda\,; \cB(C_\delta)\Big)
\le 2\exp \Big(-\frac{c_*n}{s_0}\Big) + \exp\Big(-\frac{n}{1000}\Big)\,.
\end{eqnarray}
This completes the proof.

\section{Proof of Theorem~\ref{thm:unknown} (unknown covariance)}\label{proof:thm_unknown}
We decompose $\sqrt{n}(\dth-\th^*)$ into three terms:
\begin{eqnarray*}
\sqrt{n}(\dth-\th^*) &=& \sqrt{n} (\hth-\th^*)+\frac{1}{\sqrt{n}} MX^\sT(y-X\hth)\\
&=& \sqrt{n}(\id-M\hSigma)(\hth-\th^*)+\frac{1}{\sqrt{n}} MX^\sT w\\
&=&\underbrace{\sqrt{n}(\id-\Omega\hSigma)(\hth-\th^*)}_{ I_1} + 
\underbrace{{\sqrt{n}} (\Omega-M)
  \hSigma(\hth-\th^*)}_{I_2}+\underbrace{\frac{1}{\sqrt{n}} MX^\sT
  w}_{I_3}\, .
\end{eqnarray*}
Note that the term $I_1$ is exactly the bias vector $R$ of the debiased estimator in case of known covariance (with $M=\Omega$). Therefore, by invoking the result of Theorem~\ref{thm:main}, we have
\begin{eqnarray}
\prob\Big(\|I_1\|_\infty \ge C\sqrt{\frac{s_0}{n}}\log p \Big) \le 2pe^{-c_*n/s_0} + pe^{-n/1000}+8p^{-1}+2e^{-\delta^2 n}\,.\label{eq:I1}
\end{eqnarray}
We next provide two bounds on $\|I_2\|_\infty$.

In our first bound, we use duality of $\ell_\infty$ norm (on $\hSigma(\hth-\tth)$) and $\ell_1$ norm on rows of $\Omega - M$ as follows:
\begin{eqnarray}\label{eq:I2-1}
\|I_2\|_\infty \le \sqrt{n} \|\Omega-M\|_\infty \|\hSigma(\hth-\tth)\|_\infty\,.
\end{eqnarray}
By the KKT condition for $\hth$, there exists a vector $\xi$ in the subgradient of the $\ell_1$ norm at $\hth$, such that 
$
\hSigma(\hth-\tth) = X^\sT w/n - \lambda \xi\,.
$
Therefore,
\begin{align}\label{eq:hSigma}
\|\hSigma(\hth-\tth)\|_\infty \le \frac{1}{n} \|X^\sT w\|_\infty + \lambda \|\xi\|_\infty\,.
\end{align}
We have $\|\xi\|_\infty\le 1$ and on event $\ctB(n,p)$, 
\begin{align}
\frac{1}{n}\|X^\sT w\|_\infty \le 2\sigma \sqrt{\frac{\log p}{n}} \le \frac{\lambda}{4}\,.
\end{align}
Using these bounds in Equation~\eqref{eq:hSigma}, we obtain
$
\|\hSigma(\hth-\tth)\|_\infty \le 5\lambda/4\,.
$
As proved in~\cite[Theorem 2.4]{van2014asymptotically}, we have
$
\|M-\Omega\|_\infty \lesssim \, s_\Omega \sqrt{{(\log p)}/{n}}.
$
Combining these bounds in Equation~\eqref{eq:I2-1} gives our first bound on $I_2$.
\begin{align}\label{eq:I2-1-2}
\|I_2\|_\infty \lesssim \sqrt{n} s_\Omega \sqrt{\frac{\log p}{n}} \sqrt{\frac{\log p}{n}} \lesssim \frac{s_\Omega \log p}{\sqrt{n}}\,.
\end{align}
To obtain a second bound on $I_2$, we proceed by writing $I_2$ as 
\begin{align}
I_2 &= \sqrt{n} \Big[(\Omega \hSigma - \id) - (M\hSigma-\id) \Big] (\hth-\tth)\,.
\end{align}
Therefore,
\begin{align}
\|I_2\|_\infty \le \sqrt{n} \Big( |\Omega\hSigma - \id|_\infty + |M\hSigma-\id|_\infty \Big) \|\hth-\tth\|_1\,.
\end{align}
On event $\cG_n(A)$ (see Equation~\eqref{eq:GnA}), we have $|\Omega\hSigma - \id|_\infty \lesssim \sqrt{(\log p)/n}$. Further, for $s_\Omega \ll n/\log p$,
we have $|M\hSigma - \id|_\infty\lesssim \sqrt{(\log p)/n}$. For the proof of this inequality  we refer the reader to~\cite{van2014asymptotically}, Equation (10) and Lemma 5.3 therein. In addition, on the event $\cB\equiv\cB_\delta(n,s_0,3)\cap \ctB(n,p)$ we have $\|\hth-\tth\|_1\lesssim s_0 \lambda\approx s_0 \sqrt{(\log p)/n}$. (See e.g.,~\cite{buhlmann2011statistics}.)

Combining these bounds, we arrive at
\begin{align}\label{eq:I2-2}
\|I_2\|_\infty \lesssim \frac{s_0 \log p}{\sqrt{n}}\,.
\end{align}
We summarize bounds given by~\eqref{eq:I2-1-2} and~\eqref{eq:I2-2} as 
\begin{align}
\|I_2\|_\infty \lesssim \min(s_0,s_\Omega) \frac{\log p}{\sqrt{n}}\,.
\end{align}
Finally, note that 
$$I_3|X \sim \normal(0,\sigma^2 M\hSigma M^\sT)\,.$$
The result follows by letting $Z\equiv I_3$ and $R \equiv I_1+I_2$.

\section*{Acknowledgements}

The authors would like to thank Jason D. Lee  and Cun-Hui Zhang  for 
stimulating discussions, and Zhao Ren for valuable comments to improve the presentation. A.M. was partially supported by 
NSF grants CCF-1319979 and DMS-1106627, and the
AFOSR grant FA9550-13-1-0036.

\bibliographystyle{amsalpha}

\newcommand{\etalchar}[1]{$^{#1}$}

\newpage
\appendix

\section{Proof of Lemma~\ref{lem:Bickel}}
\label{app:Lasso-supp-size}
This proposition is an improved version of Theorem 7.2 in~\cite{BickelEtAl}.

We first recall the definition of \emph{restricted eigenvalues} as given by:
\begin{align*}
\phi_{\max}(k) &\equiv \underset{1\le \|v\|_0 \le k}{\max}\, \frac{\<v,\hSigma v\>}{\|v\|_2^2}\,.
\end{align*}
Clearly, $\phi_{\max}(k)$ is an increasing function of $k$.

Employing~\cite[Remark 5.4]{Vershynin-CS}, for any $1\le k\le n$ and a fixed subset $J\subset [p]$ with $|J| = k$, we have
\[
\prob\Big(\sigma_{\max}(\hSigma_{J,J}) \ge C_{\max} + C\sqrt{\frac{k}{n}}+ \frac{t}{\sqrt{n}}\Big) \le 2e^{-ct^2}\,,
\]
for $t \ge 0$, where $C$ and $c$ depend only on $C_{\max}$. Therefore, by union bound over all possible subsets $J \subseteq [p]$ we obtain
\begin{align}
\prob\Big(\phi_{\max}(k) \ge  C_{\max} + C\sqrt{\frac{k}{n}}+ \frac{t}{\sqrt{n}} \Big) \le 2 {p\choose k} e^{-c t^2} \le  2e^{-ct^2+ k\log p+k}\,,\label{eq:bound-phi}
\end{align}
for $t \ge 0$.

Let $\hS \equiv \supp(\hth)$. Recall that the stationarity condition for the Lasso cost function reads
$X^{\sT}(y-X\hth) = n\lambda\, v(\hth)$, where
$v(\hth)\in\partial\|\hth\|_1$. Equivalently,
\begin{align*}
\frac{1}{n}X^{\sT}X(\th^*-\hth) = \lambda\,
v(\hth)-\frac{1}{n}X^{\sT} w\, .
\end{align*}
On the event $\ctB(n,p)$, we have
$\|X^\sT w\|_\infty \le n\lambda/4$. Thus for all $i\in \hS$
\begin{align*}
\left|\frac{1}{n}[X^{\sT}X(\th^*-\hth)]_i\right| \ge
\frac{\lambda}{2}\, .
\end{align*}
%
Squaring and summing the last identity
over $i\in\hS$, we obtain that, for $h \equiv
n^{-1/2}X(\th^*-\hth)$,
\begin{align}\label{eq:square}
\frac{\lambda^2}{4}|\hS|&\le \frac{1}{n} \sum_{i\in \hS} (e_i^\sT X^\sT h)^2 = \<h,\frac{1}{n}X_{\hS}X_{\hS}^{\sT}\,
h\>\le \|\hSigma_{\hS,\hS}\|_2^2 \|h\|^2\
\le  \phi_{\max}(|\hS|)\|h\|_2^2\,.
%
\end{align}
By a similar argument as in Corollary~\ref{coro:zeta}, on the event $\cB \equiv \ctB(n,p) \cap \cB(n,s_0,3)$ we have
$$\|h\|_2^2\le \frac{4\lambda^2s_0}{(1-\delta)^2 C_{\min}}\,.$$ 
Thus,
\begin{align}\label{eq:bound-phi0-A}
|\hS|&\le \frac{16 \phi_{\max}(\hS)}{(1-\delta)^2 C_{\min}}s_0\,.
\end{align}

Note that $|\hS| \le n$ by the fact
that the columns of $X$ are in generic positions. Using monotonicity property of $\phi_{\max}(\cdot)$, we have
$\phi_{\max}(|\hS|) \le \phi_{\max}(n)$. 
Invoking equation~\eqref{eq:bound-phi} with $k=n$, we have $\phi_{\max}(n) < c_1\sqrt{\log p}$ with high probability for some constant $c_1$.

Hence, by equation~\eqref{eq:bound-phi0-A}
\begin{align}
|\hS| < \widetilde{C} s_0 \sqrt{\log p}\,,\quad \quad \tilde{C} \equiv \frac{16 c_1}{(1-\delta)^2 C_{\min}}\,.
\end{align}
%
%
Now, we use this bound on $|\hS|$ along with equation~\eqref{eq:bound-phi0-A} to get a better bound on $|\hS|$.
Again by using the fact that $\phi_{\max}(k)$ is a non-decreasing function of $k$, we have
\begin{align}
\phi_{\max}(|\hS|) < \phi_{\max}(\widetilde{C} s_0\sqrt{\log p}) \le C_{\max}\,,\label{eq:phi_maxB}
\end{align}
with high probability where we used the assumption $n\gg s_0 (\log p)^2$.
Using this bound in equation~\eqref{eq:bound-phi0-A}, we get
\[|\hS|< \frac{16 C_{\max}}{(1-\delta)^2 C_{\min}}s_0\,.\]
The result follows.

\section{Proof of Lemma~\ref{lem:prop-norm}}\label{app:prop-norm}
By definition of $\ell_\infty$ operator norm, for a symmetric invertible matrix $A$ we have
\begin{align}
\|A^{-1}\|_\infty  \equiv \max_{v\neq 0} \frac{\|A^{-1} v\|_\infty}{\|v\|_\infty} = 
\max_{u\neq 0} \frac{\|u\|_\infty}{\|Au\|_\infty} = 
\frac{1}{\min_{u\neq 0} \frac{\|Au\|_\infty}{\|u\|_\infty}}\,.
\end{align}
Note that for any set $T\subseteq [p]$ we have 
$$ \min_{u\neq 0} \frac{\|Au\|_\infty}{\|u\|_\infty} \le \min_{\tilde{u}\neq 0} \frac{\|A_{T,T}\tilde{u}\|_\infty}{\|\tilde{u}\|_\infty}\,,$$
whence we obtain
\begin{align}
\|A^{-1}\|_\infty \ge \frac{1}{\min_{u\neq 0} \frac{\|A_{T,T}\tilde{u}\|_\infty}{\|\tilde{u}\|_\infty}} = \|A_{T,T}^{-1}\|_\infty\,.
\end{align}
Since the above inequality holds for any $T\subseteq [p]$, we obtain the desired result.

\section{Proof of Lemma~\ref{lem1}}\label{app:lem1}
For $\th$ we have
\begin{align*}
\cL_{y, X} (\theta^*_i,\th) = \frac{1}{2n}\|y-\tx_i \theta^*_i -\Xnz\th\|^2+\lambda\|\th\|_1+\lambda|\theta^*_i|
\end{align*}
Let $\ty \equiv y-\tx_i \theta^*_i$. We then have
\begin{align}
\cL_{y, X} (\theta^*_i,\th) &=  \frac{1}{2n}\|\ty -\Xnz\pthnz -\Xnz(\th-\pthnz)\|^2+\lambda\|\th\|_1+\lambda|\theta^*_i| \nonumber\\
&= \cL_{y,X}(\theta^*_i, \pthnz) + \frac{1}{2n}\|\Xnz(\th-\pthnz)\|^2 - \frac{1}{n}\<\ty-\Xnz \pthnz,\Xnz(\th-\pthnz)\> \nonumber\\
&\quad+ \lambda \|\th\|_1-\lambda\|\pthnz\|_1\label{eq:Diff1}
\end{align}
Since $\pthnz$ is the minimizer of $\cL_{y,X}(\theta^*_i,\th)$ by KKT condition we have
\begin{align}\label{eq:subgrad}
\frac{1}{n}\Xnz^\sT(\ty - \Xnz\pthnz) = \lambda \xi,\quad \quad \xi\in \partial\|\pthnz\|_1\,.
\end{align}
Applying equation~\eqref{eq:subgrad} in equation~\eqref{eq:Diff1} we get
\begin{align*}
\cL_{y, X} (\theta^*_i,\th) - \cL_{y,X}(\theta^*_i, \pthnz) &=
 \frac{1}{2n}\|\Xnz(\th-\pthnz)\|^2 + \lambda\Big(\|\th\|_1- \|\pthnz\|_1 - \<\xi,\th-\pthnz\>\Big)\\
 &\ge  \frac{1}{2n}\|\Xnz(\th-\pthnz)\|^2\,,
\end{align*}
where the last step follows from the definition of a subgradient.
\section{Proof of Lemma~\ref{lem:f_k}}\label{app:aux0}
Define $\xopta = \arg\min_x f_1(x)$. It is simple to see that $\xopta = \eta(a+u_1;\lambda/c_i)$, where $\eta(x;\alpha)$ is the soft-thresholding function given by 
\begin{align*}
\eta(x;\alpha) = \begin{cases}
x-\alpha & x\ge \alpha\,,\\
0 & |x|\le \alpha\,,\\
x+\alpha & x\le -\alpha\,.
\end{cases}
\end{align*}
By substituting for $\xopta$ in equation~\eqref{eq:expansion-nn} and after some algebraic manipulations, we obtain
\begin{eqnarray*}
f_1(\xopta) = c_i\cH(a+u_1;\lambda/c_i) +b_1\,,
\end{eqnarray*}
where $\cH(x;\alpha)$ is the Huber function:
\begin{eqnarray*}
\cH(x;\alpha) = \begin{cases}
\alpha |x| -\frac{\alpha^2}{2} &\text{ if } |x|>\alpha\,,\\
\\
\frac{x^2}{2} &\text{ if } |x|\le\alpha\,.
\end{cases}
\end{eqnarray*}
Similarly, setting $\xoptb = \arg\min_x f_2(x)$ we have 
\begin{eqnarray*}
f_1(\xoptb) = c_i\cH(a+u_2;\lambda/c_i) +b_2\,.
\end{eqnarray*}

Define $\Delta_1\equiv f_1(a) - f(\xopta)$ and $\Delta_2\equiv f_2(a) - f(\xoptb)$. Substituting for $f_1(a)$ and $f_2(a)$, we get
\begin{align}
\Delta_1 =c_i \frac{u_1^2}{2} + \lambda |a| - c_i\cH(a+u_1;\lambda/c_i)\,,\\
\Delta_2 =c_i \frac{u_2^2}{2} + \lambda |a| - c_i\cH(a+u_2;\lambda/c_i)\,.
\end{align}
We then write
\begin{align}
f_1(a)-f_2(a) = \Delta_1 - \Delta_2 + f_1(\xopta) - f_2(\xoptb) \le \Delta_1 - \Delta_2\,,
\end{align}
where we use the assumption $\min_x f_1(x) \le \min_x f_2(x)$.

Finally we bound $\Delta_1-\Delta_2$ as follows:
\begin{align*}
\Delta_1-\Delta_2 &= c_i \frac{u_1^2-u_2^2}{2} + c_i \Big\{\cH(a+u_2;\lambda/c_i)- \cH(a+u_1;\lambda/c_i) \Big\}\\
&\le c_i u_2 (u_1-u_2) + \frac{(u_1-u_2)^2}{2} + \lambda |u_1-u_2|\,
\end{align*}
where the last inequality holds since $\cH'(x;\alpha) = x -\eta(x;\alpha)$ and hence $|\cH'(x;\alpha)| \le \alpha$ and due to the
mean-value theorem. 
\section{Proof of Lemma~\ref{lem:aux1}}\label{app:aux1}
To lighten the notation, we drop the subscripts $y,X$ in $\cL_{y,X}(\cdot)$.
Recall that $\Delta(\th)\equiv \cL_{y,X}(\theta^*_i,\th) - \cL^+(\th)$. We start by expanding $\cL(\theta_i,\th)$.
\begin{align*}
\cL(\theta_i,\th) =\frac{1}{2n} \|y-\tx_i \theta_i- \Xnz \th\|_2^2+ \lambda |\theta^*_i| + \lambda \|\th\|_1 \,.
\end{align*}
Plugging in $y= \tx_i \theta^*_i + \Xnz \tthnz +w$ and rearranging the terms, we obtain
\begin{align}
\cL(\theta_i,\th) =&\frac{1}{2n}\|w+ \Xnz(\tthnz- \th)\|_2^2+ \frac{1}{n} \<\th^*_i-\th_i, \tx_i^\sT(w+\Xnz(\tthnz-\th))\>\nonumber\\
&+\frac{1}{2n} \|\tx_i\|^2 (\th^*_i - \th_i)^2+\lambda |\theta_i| + \lambda \|\th\|_1\,.\label{eq:expansion}
\end{align}
Therefore,
\begin{align}
\cL(\th^*_i,\th) = \frac{1}{2n}\|w+ \Xnz(\tthnz- \th)\|_2^2+\lambda |\theta^*_i| + \lambda \|\th\|_1\,.\label{eq:expansion-tt}
\end{align}
Combining equations~\eqref{eq:expansion} and~\eqref{eq:expansion-tt}, we rewrite $\cL(\theta_i,\th)$ as
\begin{align}
\cL(\theta_i,\th) =&\cL(\th^*_i,\th)+ \frac{1}{n} \<\th^*_i-\th_i, \tx_i^\sT(w+\Xnz(\tthnz-\th))\>\nonumber\\
&+\frac{1}{2n} \|\tx_i\|^2 (\th^*_i - \th_i)^2+\lambda |\theta_i|  - \lambda |\th^*_i|\nonumber\\
=& \lambda |\theta_i| + \frac{1}{2n} \|\tx_i\|^2 \Big(\th_i-\th^*_i - \frac{\tx_i^\sT}{\|\tx_i\|^2} (w+\Xnz(\tthnz-\th)) \Big)^2\nonumber\\
&-\frac{1}{2n\|\tx_i\|^2} \Big(\tx_i^\sT(w+\Xnz(\tthnz-\th)) \Big)^2 + \cL(\th^*_i,\th)-\lambda |\th^*_i|\,.\label{eq:expansion-n}
\end{align}
Writing expression~\eqref{eq:expansion-n} in terms of $c_i \equiv \|\tx_i\|^2/n$ and $u(\th)$, given by~\eqref{eq:u}, we get
\begin{align}
\cL(\theta_i,\th)  = \lambda |\theta_i| +\frac{c_i}{2}(\th_i - \th^*_i -u(\th))^2-\frac{c_i}{2}u(\th)^2+ \cL(\th^*_i,\th)-\lambda |\th^*_i|\,.\label{eq:expansion-nn-NEW}
\end{align}
%
The result follows.
%
\section{Proof of Preposition~\ref{propo:F'main}}\label{app:F'main}

Let $T = \supp(\pthnz) \cup \supp(\theta_*)$. By Lemma~\ref{lem:Bickel}, $|T| < C_\delta s_0$, where 
$$C_\delta \equiv C_*+1 = \frac{16}{(1-\delta)^2} \frac{C_{\max}}{C_{\min}}+1\,.$$
$J = \tilde{J}\backslash \{i\}$.
For $i\in[p]$ define 
$$\Sigma_{i|T} \equiv \Sigma_{i,i} - \Sigma_{i,T} (\Sigma_{T,T})^{-1} \Sigma_{T,i}\,.$$
Since $\tx_i$ and $X_T$ are jointly Gaussian, we have
\begin{eqnarray}\label{eq:schur}
\tx_i = X_T (\Sigma_{T,T})^{-1} \Sigma_{T,i} + \Sigma_{i|T}^{1/2} z\,,
\end{eqnarray}
where $z\in \reals^{n}$ is independent of $X_T$ with i.i.d standard normal coordinates. 

Recalling the definition of $c_i \equiv \|\tx_i\|^2/n$ and $u(\theta)$, given by equation~\eqref{eq:u}, we write $c_i|u(\pthnz)|$ as 
\begin{align}
c_i|u(\pthnz)| &= \frac{1}{n} \bigg|\tx_i^\sT(w+\Xnz(\tthnz-\pthnz)) \bigg|\nonumber\\
&=\frac{1}{n} \bigg|\tx_i^\sT(w+X_T(\tthT-\pthT)) \bigg|\nonumber\\
&\le \frac{1}{n} |\tx_i^\sT w| + \frac{1}{n} \Sigma_{i|T}^{1/2} \bigg|z^\sT X_T(\tthT-\pthT)\bigg|
+\frac{1}{n} \bigg|\Sigma_{i,T} (\Sigma_{T,T})^{-1} \Xnz^\sT\Xnz(\tthT-\pthT)\bigg|\nonumber\\
&\le \frac{1}{n} |\tx_i^\sT w| + \frac{1}{n} \Sigma_{i|J}^{1/2} \bigg|z^\sT X_T(\tthT-\pthT)\bigg|
+\frac{1}{n} \|\Sigma_{i,T} (\Sigma_{T,T})^{-1}\|_1 \|X_T^\sT X_T(\tthT-\pthT)\|_\infty\,.
\label{eq:u2}
\end{align}
The first inequality here follows from equation~\eqref{eq:schur}.

In the following we bound each term on the RHS of equation~\eqref{eq:u2} individually.

On the event $\ctB(n,p)$, defined by equation~\eqref{eq:ctB}, we have
\begin{align}\label{eq:term1}
\frac{1}{n} \|\tx_i^\sT w\| \le \frac{1}{n}\|X^\sT w\|_\infty \le 2\sigma\sqrt{\frac{\log p}{n}} \le \frac{\lambda}{4}\,.
\end{align}

We use Corollary~\ref{coro:zeta} to bound the second term of expression~\eqref{eq:u2}.
We recall the event $\cB_\delta(n,s_0,3)$, given by equation~\eqref{eq:Bdelta} and let $\cB \equiv  \cB_\delta(n,s_0,3) \cap \ctB(n,p)$.
Further, recall the notation $\zeta_i \equiv \Xnz(\tthnz-\pthnz)/\sqrt{n} = X_T(\tthT-\pthT)/\sqrt{n}$ and the event $\event_i$ defined by equation~\eqref{eq:eventE}. We write
\begin{align}
\prob\Big(\frac{1}{\sqrt{n}}\Sigma_{i|T}^{1/2}\; |z^\sT \zeta_i|\ge \lambda; \cB \Big) 
&\le \prob\Big(\frac{1}{\sqrt{n}}\Sigma_{i|T}^{1/2}\; |z^\sT \zeta_i| \ge \lambda; \event_i \Big) \nonumber\\
&= \E \Big\{\ind\Big(\frac{1}{\sqrt{n}}\Sigma_{i|T}^{1/2}\; |z^\sT \zeta_i|\ge \lambda \Big) \cdot \ind(\event_i) \Big\} \nonumber\\
&\le 2\E \Big(\exp\Big[-\frac{n\lambda^2}{2\|\zeta_i\|^2}\Big]\cdot  \ind(\event_i)\Big)\nonumber\\
&\le 2\exp (-c_*\frac{n}{s_0}) \,,\label{eq:term2}
\end{align}
with $c_* \equiv (1-\delta)^2C_{\min}/8$. Here, the penultimate inequality follows from Fubini's theorem where we first integrate w.r.t $z$ and then w.r.t $\zeta_i$. Note that $z$ and $\zeta_i$ are independent. Therefore, $z^\sT \zeta_i | \zeta_i \sim \normal(0,\|\zeta_i\|^2)$. In the last step, we applied Corollary~\ref{coro:zeta}.

We next bound the third term on the RHS of equation~\eqref{eq:u2}. 
Note that the KKT conditions for optimization~\eqref{eq:pLasso} reads
\begin{eqnarray}\label{eq:nu}
\frac{1}{n}\Xnz^\sT (w+\Xnz (\tthnz-\pthnz)) = \lambda \xi\,,
\end{eqnarray}
for $\xi \in \partial\|\pthnz\|_1$. Since $\tthnz-\pthnz$ is supported on $T$, we have $\Xnz(\tthnz-\pthnz) = X_T(\tthT-\pthT)$. 
To lighten the notation, let 
$$\nu \equiv \frac{1}{n}X_T^\sT X_T(\tthT-\pthT)\,.$$
We know by equation~\eqref{eq:nu},
$$\|\nu\|_\infty \le \frac{1}{n}\|X_T^\sT w\|_\infty + \lambda \|\xi_T\|_\infty\,.$$
%
On the event $\ctB(n,p)$ we have
\[\frac{1}{n}\|X_{T}^\sT w\|_\infty \le 2\sigma \sqrt{\frac{\log p}{n}} \le \frac{\lambda}{4}\,. \]
Combining the above two inequalities we obtain
\begin{align}\label{eq:nu-bound}
\|\nu\|_\infty \le 5\lambda/4\,.
\end{align}

We next employ Condition~\ref{Condition:L1} to bound $\|\Sigma_{i,T}(\Sigma_{T,T})^{-1}\|_1$.
Define $\tilde{T} = T\cup \{i\}$ and write the inverse of $\Sigma_{\tilde{T},\tilde{T}}$
using Schur complement:
\[
\Sigma_{\tilde{T},\tilde{T}}^{-1} = 
\begin{pmatrix}
\Sigma_{i|T}^{-1} & -\Sigma_{i|T}^{-1}  \Sigma_{i,T} \Sigma_{T,T}^{-1}\\
-\Sigma_{T,T}^{-1}  \Sigma_{T,i} \Sigma_{i|T}^{-1} & \Sigma_{T,T}^{-1} + \Sigma_{T,T}^{-1}  \Sigma_{T,i} \Sigma_{i|T}^{-1} \Sigma_{i,T} \Sigma_{T,T}^{-1}
\end{pmatrix}\,.
\]
By Condition~\ref{Condition:L1} and as $|\tilde{T}|\le C_\delta s_0$, $\|\Sigma_{\tilde{T},\tilde{T}}^{-1}e_i\|_1\le \rho$. Further, by Condition~\ref{Condition:diag}, $\Sigma_{i|T}\le \Sigma_{i,i}\le 1$. Hence, we get
\begin{eqnarray}\label{eq:rho-1}
\rho \ge \|\Sigma_{\tilde{T},\tilde{T}}^{-1}e_i\|_1\ge 
1+ \|\Sigma_{i,T} (\Sigma_{T,T})^{-1}\|_1\,.
\end{eqnarray}

Using equations~\eqref{eq:term1} to~\eqref{eq:rho-1}, we bound the RHS of equation~\eqref{eq:u2} as follows. Under the event $\cB$,
\[
c_i|u(\pthnz)| \le \frac{5\lambda}{4}\rho\,.
\]
%

%
%
\section{Proof of Corollary~\ref{coro:Bickel}}\label{proof:coroBickel}
Note that $\pthnz$ is the Lasso estimators corresponding to $(\ty,X_{\sim i})$, according to equation~\eqref{eq:pLasso}.
As a corollary of Proposition~\ref{lem:Bickel}, on event $\cB$, $\|\pthnz\|_0 \le C_* s_0$, with $C_* \equiv (16C_{\max}/C_{\min})(1-\delta)^{-2}$.
Also, $\|\hth_{\sim i}\|_0 \le s_0$. Therefore,$(0,\hth_{\sim i}-\pthnz) \in \cone((C_*+1)s_0,3)$ and, by definition, on event $\cB_\delta(n,(C_*+1)s_0,3)$, the claim holds true.

\section{Sample splitting techniques}\label{app:dataSplit}

In this appendix, we discuss how sample splitting can be used to 
modify the debiased estimator as to go around the sparsity barrier at 
$s_0 = o(\sqrt{n}/\log p)$. This provides an alternative to the more
careful analysis carried out in the main body of the paper, that we
discuss for the sake of simplicity. 
As mentioned in the introduction, sample splitting has its own
drawbacks, most notably the dependence of the results on the random
data split, and the  sub-optimal use of all the samples.

For the sake of notational simplicity we assume here that the number
of samples is $2n$ and is randomly split in two batches of size $n$:
$(x_1,y_1)$, \dots, $(x_n,y_n)$, and $(\bax_1,\bay_1)$,\dots,
$(\bax_n,\bay_n)$. Note that the change of notation only amounts to a
constant multiplicative factor in the sample size, which is of no
concern to us.  In vector notation, these batches are denoted as
$(y,X)$ and $(\bay,\baX)$. We then proceed as follows:
\begin{enumerate}
\item We use the second batch to compute  the Lasso estimator, namely
\begin{eqnarray}
\hth(\bay,\baX;\lambda)\equiv \arg\max_{\theta\in\reals^p}
\left\{\frac{1}{2n}\|\bay-\baX\theta\|_2^2+\lambda
\|\theta\|_1\right\}\,.
\end{eqnarray}
\item We use the first batch to compute the debiasing matrix $M$,
  e.g. using the node-wise Lasso as in Section \ref{sec:debias}.
\item We use the first batch to implement the debiasing, namely
\begin{align}\label{eq:split-debiased}
\sth = \hth(\bay,\baX) + \frac{1}{n} M X^\sT \big(y- X\hth(\bay,\baX)\big)\,.
\end{align}
\end{enumerate}
The main remark is that, thanks to the splitting, $X$ is statistically
independent from $\hth$, which greatly simplifies the analysis. Notice
that we did not use the responses in $y$.

For the sake of simplicity, we shall analyze this procedure in the
case in which the precision matrix $\Omega$ is known, and we hence set
$M=\Omega$. The generalization to $M$ constructed via the node-wise
Lasso is straightforward as in the proof of Theorem \ref{thm:unknown}.

The next statement implies that, for sparsity level $s_0 = o(n/(\log
p)^2)$,the sample splitting debiased estimator is asymptotically Gaussian.
\begin{propo}
Consider the linear model~\eqref{eq:NoisyModel} where $X$ has
independent Gaussian rows, with zero mean and covariance $\Sigma$.  
Suppose that $\Sigma$ satisfies the technical conditions of Theorem
\ref{thm:main}

Let $\hth$ be the Lasso estimator defined by~\eqref{eq:Lasso} with $\lambda = 8\sigma \sqrt{(\log p)/n}$. Further, let $\sth$ be 
the modified (sample-splitting) debiased estimator defined in
Eq.~(\ref{eq:split-debiased}) with $M=\Omega\equiv \Sigma^{-1}$. 
Then, there exist constants $c, C$ depending solely on
$C_{\min},C_{\max},\delta$ and $\rho$, such that, for $n\ge c\,\max(\log p, s_0\log(p/s_0))$
the following holds true:
\begin{align}
&\sqrt{n}(\dth-\th^*)= Z + R\,, \quad \quad Z|X\sim\normal(0,\sigma^2\Omega\hSigma\Omega)\,, \\
&\lim_{n\to \infty}\prob\Big(\|R\|_\infty \ge C\sqrt{\frac{s_0}{n}}\log p \Big) = 0\, .
\end{align}
\end{propo}
\begin{proof}
Proceeding as in the proof of Theorem \ref{thm:main}, it is sufficient
to bound the bias term of $\sqrt{n}(\sth -\th^*)$, which is given by (cf. \eqref{eq:noisebias})
\begin{align}
R\equiv \sqrt{n} (\Omega \hSigma -\id)(\th^* - \hth)\,.
\end{align}
To lighten the notation, let $u = \th^*-\hth$. Expanding $R$ we get
\begin{align}
R = \sqrt{n}(\Omega \hSigma - \id)u = \frac{1}{\sqrt{n}} \sum_{i=1}^n (\Omega x_ix_i^\sT - \id)u\,.
\end{align}
To control $\|R\|_\infty$, we bound each component $R_j$ individually. Let $e_j$ be the $j$-th element of the standard basis with one at the $j$-th position and zero everywhere else. We write 
\begin{align*}
R_j &= \frac{1}{\sqrt{n}} \sum_{i=1}^n (e_j^\sT \Omega x_i)(x_i^\sT u) - u_j\,.
\end{align*}
Let $Z_i \equiv (e_j^\sT \Omega x_i)(x_i^\sT u) - u_j$. 
Note that conditional on $(\bay,\baX)$, $\htheta$ and therefore $u$ are deterministic. Furthermore, since the first batch $(y,X)$ is independent of $(\bay,\baX)$, the rows $x_i$ are independent conditional on $(\bay,\baX)$. Therefore, $Z_i|(\bay,\baX)$ are independent with $\E(Z_i|\bay,\baX) = e_j^\sT \Omega \Sigma u - u_j = 0$. We let $\|\cdot\|_{\psi_1}$ and $\|\cdot\|_{\psi_2}$ respectively denote the sub-exponential and sub-gaussian norms and condition on $(\bay,\baX)$ in the sequel. As shown in \cite[Remark 5.18]{Vershynin-CS},
$$\|Z_i\|_{\psi_1} \le 2 \|(e_j^\sT \Omega x_i) (x_i^\sT u)\|_{\psi_1}\,.$$

In addition, for any two random variables $v$ and $w$, we have $\|v w\|_{\psi_1} \le 2 \|v\|_{\psi_2} \|w\|_{\psi_2}$. Hence,
\begin{align*}
\|(e_j^\sT \Omega x_i) (x_i^\sT u)\|_{\psi_1} &\le 2\|e_j^\sT \Omega x_i\|_{\psi_2} \|x_i^\sT u\|_{\psi_2}\\
&= 2\|e_j^\sT \Omega^{1/2} \|_2 \|\Omega^{1/2} x_i\|_{\psi_2}^2 \|\Omega^{-1/2} u \|_2\\
&\le 2\sqrt{C_{\max}/C_{\min}} \, \|\Omega^{1/2} x_i\|_{\psi_2}^2 \|u\|_2\,.
\end{align*}
Given that $\Omega^{1/2} x_i \sim \normal(0,\id)$, we get $\|\Omega^{1/2} x_i\|_{\psi_2} = 1$. Hence, $\max_{i} \|Z_i\|_{\psi_1} \le C \|u\|_2$ with $C \equiv 4 \sqrt{C_{\max}/C_{\min}}$.
Applying Bernstein-type inequality \cite[Proposition 5.16]{Vershynin-CS}, for every $t\ge 0$, we have
\begin{eqnarray}\label{eq:Bernstein}
\prob\Big\{\bigg|\sum_{i=1}^n \frac{1}{\sqrt{n}} Z_{i}\Big|\ge t \;\bigg| (\bay,\baX) \bigg\}\le
2\exp \Big[ -c \min\Big(\frac{t^2}{C^2\|u\|_2^2}, \frac{t\sqrt{n}}{C\|u\|_2}\Big)\Big] \,,
\end{eqnarray}
where $c >0$ is an absolute constant. Observe that on the event $\cB\equiv \cB_\delta(n,s_0,3)\cap \tilde{\cB}(n,p)$\footnote{See Section~\ref{sec:preliminary} for definition of $\cB_\delta(n,s_0,3)$ and $\tilde{\cB}(n,p)$}, we have
\begin{align*}
\|u\|_2^2 = \|\th^*-\hth\|_2^2 \lesssim s_0\lambda^2\,.
\end{align*}
Therefore, by using tail bound~\eqref{eq:Bernstein} and applying union
bound over the $p$ entries of $R$, we get 
(for $n\ge c\log p$ with $c$ a suitable constant) 
\[
\|R\|_\infty \lesssim \sqrt{\frac{s_0}{n}} \log p\,,
\]
with high probability. 
\end{proof}

\section{Proof of Propositions~\ref{propo:minimaxUnknown-UB} and \ref{propo:minimaxUnknown-LB}}
\label{sec:proofminimax_UL}
\subsection{Proof of Proposition~\ref{propo:minimaxUnknown-UB}}\label{sec:proofminimax_U}
Fix $M$ and $\lambda$ for which Equations~\eqref{eq:Delta1}-\eqref{eq:Delta2} hold true and let 
\[\th^\de = \hth+\frac{1}{n}MX^\sT (y-X\hth)\,.\]
We construct confidence interval $J^\de_\alpha$ centered at $\hth^\de$ as follows:
\begin{eqnarray}
J^\de_\alpha &\equiv& [\hth^\de_1-\delta(\alpha,n),\hth^\de_1+\delta(\alpha,n)]\,\label{eq:ConfInterval-app1}\\
\delta(\alpha,n) &\equiv&
{\Phi^{-1}(1-\alpha/2)}\frac{1}{(1-\eps)\sqrt{n}} \min\{\hsigma \sqrt{Q_{1}},\sqrt{(1+\eps)c C}\} +\frac{\Delta_n}{\sqrt{n}}\,, \label{eq:ConfInterval-app2}
\end{eqnarray}
where $\eps\in(0,1/2)$ is arbitrary fixed value and $\Phi(x) \equiv \int_{-\infty}^xe^{-t^2/}\de t/\sqrt{2\pi}$ is
the Gaussian distribution. Further, recall that $c$ is the bound on $\sigma$ in the definition of $\Gamma(s_0,s_\Omega, \rho)$.

We have
\begin{align}
\ell(J^\de_\alpha) \le {2\Phi^{-1}(1-\alpha/2)}\frac{1}{(1-\eps)\sqrt{n}} \sqrt{(1+\eps)c C} +\frac{\Delta_n}{\sqrt{n}}\,,
\end{align}
and therefore, $\E_\gamma\{\ell(J^\de_\alpha)\} \lesssim (1+\Delta_n)/\sqrt{n}$.


We next show that $J^\de_\alpha \in \cI_\alpha(\Gamma)$. 
Define the following events:
\begin{eqnarray}
\cE_1 &\equiv& \{(1-\eps)\sigma\le\hsigma \le (1+\eps)\sigma \},\\
\cE_2 &\equiv& \{\|Q\|_\infty \le C\} \,,\\
\cE_3&\equiv & \{\|R\|_\infty \le \Delta_n\}\,.
\end{eqnarray}
We further let $\cE = \cE_1\cap \cE_2\cap \cE_3$ and
$Z = e_1^\sT M X^\sT w/\sqrt{n}$. Since $Z/\sqrt{n} |X \sim \normal(0,\sigma^2 Q_{1}/n)$, we have
\begin{align}\label{eq:gauss}
\prob\Big(\frac{1}{\sqrt{n}}|Z| \le \sqrt{\frac{Q_{1}}{n}}\sigma \Phi^{-1}(1-\alpha/2) \Big| X\Big) = 1-\alpha\,.
\end{align}
By integrating w.r.t $X$ we get the same coverage probability unconditionally.
Note that on event $\cE$, $\hsigma \sqrt{Q_1}\le \sqrt{(1+\eps) \sigma^2 C}$ and on $\Gamma(s_0,s_\Omega, \rho)$, we have $\sigma\le \sqrt{c}$.
Further, $\sigma \le \hsigma/(1-\eps)$. Hence, on event $\cE$
\begin{align}
\delta(\alpha,n) &=
{\Phi^{-1}(1-\alpha/2)}\frac{\hsigma}{(1-\eps)}  \sqrt{\frac{Q_{1}}{n}} +\frac{\Delta_n}{\sqrt{n}}\nonumber\\
&\ge{\Phi^{-1}(1-\alpha/2)}\sigma  \sqrt{\frac{Q_{1}}{n}} +\frac{\Delta_n}{\sqrt{n}}\equiv \delta_0(\alpha,n)\,.
 \label{eq:ConfInterval-app-E}
\end{align}
We have the following bound on the coverage probability
\begin{align}
\prob(\th^*_1\in J^\de_\alpha) &= \prob(|\hth^\de_1-\tth_1| \le \delta(\alpha,n)) \\
&\ge \prob(\{|\hth^\de_1-\th^*_1| \le \delta_0(\alpha,n)\} \cap \cE)\\
&\stackrel{(a)}{\ge} \prob\Big(\Big\{\frac{1}{\sqrt{n}}|Z| \le \sqrt{\frac{Q_{1}}{n}}\sigma \Phi^{-1}(1-\alpha/2)\Big\} \cap \cE\Big) \\
&\ge \prob\Big(\frac{1}{\sqrt{n}}|Z| \le \sqrt{\frac{Q_{1}}{n}}\sigma \Phi^{-1}(1-\alpha/2)\Big) - \prob(\cE^c)\\
&\stackrel{(b)}{=} 1-\alpha - \prob(\cE^c) = \prob(\cE)-\alpha\,,
\end{align}
where $(a)$ follows from the decomposition $\hth^\de_1 = \tth_1 + Z/\sqrt{n} + R/\sqrt{n}$ and the fact that $\|R\|_\infty \le \Delta_n$ on $\cE$;
$(b)$ follows from Equation~\eqref{eq:gauss}.
Since $\prob(\cE) \to 0$, we obtain 
\begin{align}
\underset{n\to \infty} {\lim\inf}
\inf_{\gamma\in \Gamma(s_0,s_\Omega,\rho)} \prob_\gamma(\th^*_1\in J^\de_\alpha) \ge 1-\alpha\,.
\end{align}
Therefore, as claimed,
\begin{align}
\ell_\alpha^*(\Gamma(s_0,s_\Omega,\rho)) \le \E_\gamma\{\ell(J^\de_\alpha)\} \lesssim (1+\Delta_n)/\sqrt{n}\,.
\end{align}
%


\subsection{Proof of Proposition~\ref{propo:minimaxUnknown-LB}}\label{sec:proofminimax_L}

The proof follows the same lines as~\cite{cai2015confidence}[Theorem 3].  
Under the gaussian design model, the data pairs $(y_i,x_i)$ has a joint gaussian distribution with mean zero and covariance
$\tSigma$, where $\tSigma$ admits the following block decomposition:
\begin{align}
\tSigma = \begin{pmatrix}
\tSigma_{yy} & \tSigma_{yx}\\
\tSigma_{xy} & \tSigma_{xx}
\end{pmatrix} = 
\begin{pmatrix}
\theta^\sT \Sigma \theta+\sigma^2 & \theta^\sT \Sigma\\
\Sigma \theta & \Sigma
\end{pmatrix}\,,
\end{align}
where we posit the model $y=X\theta+w$ with $w\sim \normal(0,\sigma^2\id)$. (Throughout this section, we simplify 
our notations by writing $\theta$ instead of $\theta^*$ for the true model parameters.)
We also define $\PSD(p) \equiv \{M\in \reals^{p\times p}:\, M\succeq 0\}$, the set of positive semidefinite matrices of size $p$.

Notice that there is a one-to-one map between the parameter space $\Gamma \equiv\{\gamma = (\theta,\Omega,\sigma^2):\, \theta\in \reals^{p}, \Omega\in \PSD(p), \sigma^2 \in \reals_+\}$ and $\PSD(p+1)$. Specifically,  define the function $h: \PSD(p+1)\mapsto \Gamma$ as $h(\tSigma) = ((\tSigma_{xx})^{-1}\tSigma_{xy}, (\tSigma_{xx})^{-1}, \tSigma_{yy} - (\tSigma_{xy})^\sT (\tSigma_{xx})^{-1}\tSigma_{xy})$. The inverse map $h^{-1}$ is given by
\begin{align}
h^{-1}((\theta,\Omega,\sigma^2)) = \begin{pmatrix}
\theta^\sT \Omega^{-1} \theta + \sigma^2& \theta^\sT \Omega^{-1}\\
\Omega^{-1} \theta &\Omega^{-1}
\end{pmatrix}\,.
\end{align}
We next define a null hypothesis $H_0$ and an alternative hypothesis $H_1$ as follows.  Let $s_* = \min(s_0-1,s_\Omega)$ and $s_1 = s_0-s_*\ge 1$. The null space is a singleton 
$H_0 = \{\cgamma =(\ctheta,\id, \csigma^2)\}$ with $\ctheta_1=0$, 
$\|\ctheta\|_0 = s_1-1$ and $\csigma^2\in(0,c]$. We further let $S = \supp(\ctheta)$ 
and denote by $\pi_{H_0}$ the point mass prior on $H_0$.

Next we construct the alternative parameter space $H_1$.  First, we define the following set
\begin{align}\label{eq:Anu}
\cA(\nu,k) \equiv \bigg\{\delta:\, \delta\in \reals^{p_1}, \,\|\delta\|_0 = k, \, \delta_i\in\{0,\nu\} \text{ for }1\le i\le p_1 \Big\}\,,
\end{align}
where $p_1 = p-s_1$. We set  $k\equiv \min(s_*,{(\rho-1.01)}/{\nu})$ where $\rho$ comes from the constraint $\|\Omega\|_\infty \le \rho$ in the definition of $\Gamma(s_0,s_\Omega)$ . Later in the proof we enforce some constraints on the value of $\nu$ and in the hindsight, set 
a suitable value for $\nu$ that complies with those constraints.

For a given $\delta\in \reals^{p_1}$, define $\tSigma^\delta$ as follows (here the block decomposition corresponds
to decomposition $[p]=\{1\}\cup S\cup (S^c\setminus 1)$):
\begin{align}
\tSigma^\delta = \begin{pmatrix}
\|\ctheta\|^2+\csigma^2 & 0 & \ctheta_S^\sT &\csigma \delta^\sT \\
0 & 1 & 0_{1\times s_1} &\delta^\sT\\
\ctheta_S & 0_{s_1\times 1} & \id_{s_1\times s_1} & 0_{s_1\times p_1}\\
\csigma \delta & \delta & 0_{p_1\times s_1} & \id_{p_1\times p_1}
\end{pmatrix}\,.
\end{align}
We let $\cF \equiv \{\tSigma^\delta:\, \delta\in \cA(\nu,k)\}$ and construct the alternative space
\begin{align}
H_1 = \Big\{ (\theta,\Omega,\sigma^2):\, \gamma = h(\tSigma^\delta) \text{ for some }\tSigma^{\delta}\in \cF \Big\}\,.
\end{align}
We need to show that if $\tSigma^\delta\in \cF$ then $h(\tSigma^\delta) \in \Gamma(s_0,s_\Omega,\rho)$. 
Let $(\theta,\Omega, \sigma^2) = h(\tSigma^\delta)$. Then,
\begin{align}
\theta_1 = \frac{-\csigma \|\delta\|^2}{1-\|\delta\|^2},\, \theta_S  = \ctheta_S,\, \theta_{S^c\backslash\{1\}} = (\csigma-\theta_1) \delta.\label{eq:theta1}
\end{align}
Therefore, $\|\ctheta\|_0 = 1+ |S| + \|\delta\|_0 = 1+ (s_1-1)+s_* = s_0$. Further, if $\nu \le 1/\sqrt{s_*}$, then
$\|\delta\|_2\le \sqrt{s_*} \nu<1$ and we have
\begin{align}
\sigma^2 = \|\ctheta\|^2+\csigma^2 - \|\ctheta_S\|^2 - \csigma(\csigma-\theta_1)\|\delta\|^2 = \csigma^2 - \frac{\csigma^2\|\delta\|^2}{1-\|\delta\|^2} \le\csigma^2 <c\,.
\end{align}
Finally we note that
\begin{align}
\Omega = \frac{1}{1-\|\delta\|^2}\begin{pmatrix}
1 & 0 & -\delta^\sT\\
0& (1-\|\delta\|^2)\id_{s_1\times s_1} &0_{s_1\times p_1}\\
-\delta& 0_{p_1\times s_1} & (1-\|\delta\|^2) \id_{p_1\times p_1} + \delta \delta^\sT
\end{pmatrix}
\end{align}
Hence, $\max_{i\in [p]} |\{j\neq i, \Omega_{i,j}\neq 0\}| = \|\delta\|_0 \le s_*\le s_\Omega$. Further, ($\Omega^{-1})_{ii} = 1$ for all $i\in [p]$. Also by Weyl's inequality, if $\|\delta\|_2\le \sqrt{s_*} \nu\le \min(C_{\max}-1,1-C_{\min})$, then $C_{\min} \le \sigma_{\min}(\Sigma) \le \sigma_{\max}(\Sigma)\le C_{\max}$. 
The last condition is on $\|\Omega\|_\infty$. We have
\begin{align}
\|\Omega\|_\infty \le \frac{1+\|\delta\|_1}{1-\|\delta\|^2} \le \frac{\rho-0.01}{1-(\rho-1.01)\nu} \le \rho\,,
\end{align}
where the second inequality is due to the fact that $\delta\in \cA(\nu,k)$ and $k \le (\rho-1.01)/\nu$.  
The last inequality holds if we choose $\nu < \frac{0.01}{\rho(\rho-1.01)}$. 

Summarizing, $(\theta,\Omega,\sigma^2) \in \Gamma(s_0,s_\Omega,\rho)$ if we choose
\begin{align}\label{eq:nuCond}
\nu \le \min\Big\{\frac{1}{\sqrt{s_*}}(C_{\max}-1),\frac{1}{\sqrt{s_*}}(1-C_{\min}), \frac{0.01}{\rho(\rho-1.01)}\Big\}\,.
\end{align}
Let $\pi$ be the uniform prior on $\delta$ over $\cA(\nu,k)$ for a fixed $\nu$ (whose value is to be determined later) and denote by $\pi_{H_1}$ the induced prior  over $H_1$. We define $f_{1}$ and $f_0$ as the density function of marginal distribution of data $(y,x)$ with priors $\pi_{H_0}$ and $\pi_{H_1}$ respectively. Precisely, for $\gamma= (\theta,\Omega,\sigma^2)$ and $i\in\{0,1\}$, we have $f_i(y,x) = \int f_\gamma(y,x) \pi_i(\de\gamma) $, where $f_\gamma$ is the induced density on $(y,x)$  for random $x_i\sim\normal(0,\Omega^{-1})$ and noise $w\sim\normal(0,\sigma^2)$, with $y= \<x,\theta\>+w$ when we fix the signal $\theta$.

Applying~\cite[Lemma 1]{cai2015confidence}, we have (noting that $\theta_1$, $\ctheta_1$ are deterministic)
\begin{align}\label{eq:ell*B}
\E_{\cgamma} \{\ell(J_\alpha(y,X))\} \ge |\theta_1-\ctheta_1| \Big(1- 2\alpha - {\rm TV} (f_{1},f_{0}) \Big)_+\,,
\end{align}
where for two density functions ${\rm TV}(f_1,f_0)\equiv \int |f_1(z)-f_0(z)| \de z$ denotes their total variation distance.
Also recall the $\chi^2$ distance between $f_1$ and $f_0$: 
$$\chi^2(f_1,f_0) \equiv \int \frac{f_1^2(z)}{f_0(z)} \de z -1\,.$$ 
It is well known that ${\rm TV}(f_1,f_0) \le \sqrt{\chi^2(f_1,f_0)}$.
Using~\cite[Lemma 2]{cai2015confidence} we have
\begin{align}
\chi^2(f_{1},f_{0}) +1 = \E_{\delta,\tdelta} (1-2\delta^\sT \tdelta)^{-n} \le \E_{\delta,\tdelta} \exp(4n\delta^\sT\tdelta)\,,
\end{align}
for $\delta$ and $\tdelta$ two independent random draws from prior $\pi$ over $\cA(\nu,k)$.

By~\cite[Lemma 3]{cai2015confidence} we obtain
\begin{align}
\E_{\delta,\tdelta} \exp(4n\delta^\sT\tdelta) \le  e^{\frac{k^2}{p_1-k}} \Big(1-\frac{k}{p_1} + \frac{k}{p_1} e^{4n\nu^2}\Big)^k\,.
\end{align}
We set $\nu = c\sqrt{(\log p)/n}$. Since $k\le s_0\lesssim p^\eta$ for some constant $\eta \in [0,1/2)$, by choosing
$c$ small enough, we can ensure that ${\rm TV}(f_{\pi_{H_1}},f_{\pi_{H_0}}) \le 1/2-\alpha$.  Further, given that $s_*\le s_0 \lesssim n/\log p$ and $\rho$ is a constant,  condition~\eqref{eq:nuCond} holds true for small enough $c$. 

Finally, by invoking inequality~\eqref{eq:ell*B} and substituting for $\theta_1$ from Equation~\eqref{eq:theta1} and $\ctheta_1=0$, we obtain
\begin{align}\label{eq:final}
\E_\cgamma\{\ell(J_\alpha(y,X))\}\ge \frac{\csigma \|\delta\|^2}{1-\|\delta\|^2} \Big(\frac{1}{2}-\alpha\Big)
\asymp k\nu^2 = \min(\rho \nu, s_* \nu^2)\,.
\end{align}
Note that the inequality~\eqref{eq:final} implies that $\ell^*_\alpha(\Gamma(s_0,s_\Omega)) \gtrsim \min(\rho \nu, s_* \nu^2)$. 
Using $\nu \asymp \sqrt{(\log p)/n}$ and $s_* = \min(s_0-1,s_{\Omega})$, we get $\E_\cgamma\{\ell(J_\alpha(y,X))\}\gtrsim 
\min(\rho\sqrt{(\log p)/n}, s_* (\log p)/n)$. Proof of the lower bound rate $1/\sqrt{n}$ follows along the same lines as the proof in~\cite[Theorem 3]{cai2015confidence}.

It is worth noting that Equation~\eqref{eq:final} is much stronger than the implied minimax lower bound. Indeed it shows that the expected length of confidence intervals at any given point in a large subset of $\Gamma(s_0,s_{\Omega})$, namely $\{(\ctheta,\id,\csigma): \|\ctheta\|_0 = s_1-1, \csigma^2\in (0,c]\}$, is at least of the provided lower bound rate.

%
%

%
%

\section{Proof of Theorem \ref{thm:Approximation} and Corollary \ref{coro:MSELasso}}

\subsection{Proof of Theorem \ref{thm:Approximation}}
\label{sec:ProofApproximation}

Throughout the proof, we will use $\hth=\Lth(y,X)$ to denote the
Lasso estimator. Using the KKT conditions, it is immediate to see that
this satisfies
\begin{align}
\hth &= \eta_{\Sigma}(\dth)\\
&  = \eta_{\Sigma}\Big(\theta^*+\frac{1}{n}\Omega X^{\sT}w+\frac{1}{\sqrt{n}} R\Big)\, ,\label{eq:hthNEW}
\end{align}
with $R = \sqrt{n}(\Omega\hSigma-\id)(\theta^*-\hth)$ defined as in Theorem \ref{thm:main}. We also define $\hth^0$ by
\begin{align}
\hth^0 &\equiv   \eta_{\Sigma}\Big(\theta^*+\frac{1}{n}\Omega X^{\sT}w\Big)\, . \label{eq:hth0}
\end{align}
Recall that $\hS=\supp(\hth)$ is the support of the 
Lasso estimator.  By Proposition \ref{lem:Bickel}, we have, with high probability $|\hS|\le C_* s_0$ for a constant 
$C_*$. Define $\hS^0=\supp(\hth^0)$.  Proceeding as in Proposition \ref{lem:Bickel}, we obtain, 
with high probability $|\hS^0|\le C_* s_0$ as well. Letting $\baS\equiv \hS\cup \hS^0$, we have $|\baS|\le 2C_*s_0$.

Write $z^0\equiv \theta^*+n^{-1}\Omega X^{\sT}w$, $r\equiv R/\sqrt{n}$. By Eq.~(\ref{eq:hthNEW}), and the definition 
of $\eta_{\Sigma}(\,\cdot\,)$, cf. Eq.~(\ref{eq:EtaSigmaDef}), we have
\begin{align}
\frac{1}{2}\big\|\Sigma^{1/2}(\hth-z^0-r)\big\|_2^2 +\lambda\|\hth\|_1\le
\frac{1}{2}\big\|\Sigma^{1/2}(\hth^0-z^0-r)\big\|_2^2 +\lambda\|\hth^0\|_1\, .
\end{align}
Expanding the squares on both sides, this can be rewritten as 
\begin{align}
\frac{1}{2}\big\|\Sigma^{1/2}(\hth-\hth^0)\big\|_2^2-\<r,\Sigma(\hth-\hth^0)\>\le
-\<(\hth-\hth^0),\Sigma(\hth^0-z^0)\>+\lambda\|\hth^0\|_1-\lambda\|\hth\|_1\, .
\end{align}
By the KKT conditions for $\hth^0$ (which follow from the definition (\ref{eq:hth0}), and the definition of $\eta_{\Sigma}$),
 there exists a vector $v(\hth^0)$ in the subgradient of the $\ell_1$ norm at $\hth^0$,
such that $\Sigma(\hth^0-z^0)+\lambda\, v(\hth^0)=0$.  Hence, by definition of subgradient
\begin{align}
\frac{1}{2}\big\|\Sigma^{1/2}(\hth-\hth^0)\big\|_2^2-\<r,\Sigma(\hth-\hth^0)\>\le
-\lambda\big[\|\hth\|_1-\|\hth^0\|_1-\<v(\hth^0),(\hth-\hth^0)\>\big]\le 0\, .
\end{align}
Using the assumption $\sigma_{\min}(\Sigma)\ge C_{\min}$, we have
\begin{align}
C_{\min} \|\hth-\hth^0\|_2^2&\le 2\<\Sigma r, (\hth-\hth^0)\>\\
&\le 2\|(\Sigma r)_{\baS}\|_2 \|\hth-\hth^0\|_2\, .
\end{align}
Hence
\begin{align}
 \|\hth-\hth^0\|_2^2&\le \frac{8}{C_{\min}^2}\big\{  \|\Sigma_{\baS,\baS} r_{\baS}\|^2_2+ \|\Sigma_{\baS,\baS^c} r_{\baS^c}\|^2_2\big\}\\ 
& \le \frac{8}{C_{\min}^2}\big\{  C_{\max}^2|\baS| \|r\|^2_{\infty}+ \trho^2|\baS|  \|r_{\baS^c}\|^2_{\infty}\big\}\\
& \le \frac{32(C_{\max}^2+\trho^2)}{C_{\min}^2}\, C_*\, \frac{s_0}{n}\,\|R\|_{\infty}^2 \equiv \tC^2  \frac{s_0}{n}\,\|R\|_{\infty}^2\, . \label{eq:BoundTT0}
\end{align}
The proof is completed by using Theorem \ref{thm:main}.

\subsection{Proof of Corollary \ref{coro:MSELasso}}

As in the previous section, we use $\hth=\Lth(y,X)$ to denote the Lasso estimator and define $\hth^0$ by
\begin{align}
\hth^0 &\equiv   \eta\Big(\theta^*+\frac{1}{n}X^{\sT}w;\lambda\Big)\,. 
\end{align}
Note that, by Lemma \ref{lem:ctB}, we have $\|X^{\sT}w/n\|_{\infty}<\lambda$ with high probability, whence 
$\hS^0 \equiv\supp(\hth^0)\subseteq S\equiv \supp(\theta^*)$.
By triangular inequality and Theorem  \ref{thm:Approximation}, we get
\begin{align}
\|\hth-\theta^*\|_2 &= \|\hth^0-\theta^*\|_2+ O_P\Big(\frac{\sigma s_0\log p}{n}\Big)
 = \|(\hth^0-\theta^*)_S\|_2+ O_P\Big(\frac{\sigma s_0\log p}{n}\Big) \, .\label{eq:ThThS}
\end{align}
We next show that $\|(\hth^0-\tth)_S\|$ concentrates around its expectation. Fixing $X\in \reals^{n\times p}$, define
$$F(w;X) = \|\hth^0_S-\th^*_S\|_2 = \Big\|\eta(\tth+\frac{1}{n}\Omega X^\sT w+\frac{1}{\sqrt{n}}R )_S - \th^*_S \Big \|_2\,.$$
Noting that the soft-thresholding function $\eta(\cdot;\lambda)$ is  1-Lipschitz continuous, we have
\begin{align}
F(w;X) - F(w';X) &= \Big\|\eta(\tth+\frac{1}{n}\Omega X^\sT w+\frac{1}{\sqrt{n}}R )_S - \th^*_S \Big \|_2 - \Big\|\eta(\tth+\frac{1}{n}\Omega X^\sT w'+\frac{1}{\sqrt{n}}R )_S - \th^*_S \Big \|_2\nonumber\\
&\le \Big\|\eta\Big(\theta^*+\frac{1}{n}X^{\sT}w;\lambda\Big)_S-\eta\Big(\theta^*+\frac{1}{n}X^{\sT}w';\lambda \Big)_S\Big\|_2\nonumber\\
&\le \frac{1}{n}\big\|(X^{\sT}w-X^{\sT}w')_S\big\|_2\nonumber\\
&\le \frac{1}{n} \|X_S\|_2 \,\|w-w'\|_2\,.
\end{align}
Next by the Bai-Yin law \cite{Guionnet}), we have $\|X_S\|_2\le 2(\sqrt{s_0}+\sqrt{n})$, with high probability. Therefore, using $s_0\le n$, we obtain
$F(w;X)-F(w';X) \le 4\|w-w'\|_2/\sqrt{n}$. 

Denote by $\prob_w$ and $\E_w$ probability and expectation with respect to $w$. By Gaussian isoperimetry \cite{Ledoux}, we have
$\prob(F(w;X)-\E_w\{F(w;X)\}\ge t)\le 2\, e^{-cn t^2/\sigma^2}$, for some universal constant $c>0$ .
This implies $\E_w \|(\hth^0-\theta^*)_S\|_2 = \E_w \{\|(\hth^0-\theta^*)_S\|_2^2\}^{1/2}+O(\sigma/\sqrt{n})$,
and therefore
\begin{align}
\|(\hth^0-\theta^*)_S\|_2 \le \E_w \{\|(\hth^0-\theta^*)_S\|_2^2\}^{1/2} + \frac{t \sigma }{\sqrt{n}}\, ,
\end{align}
with probability at least $1-2e^{-ct^2}$.
Using this together with Eq.~(\ref{eq:ThThS}), we get 
\begin{align}
\|\hth-\theta^*\|_2 &= \|\hth^0-\theta^*\|_2+ O_P\Big(\frac{\sigma s_0\log p}{n}\Big)\\
& = \sqrt{\E_w\big\{\|(\hth^0-\theta^*)_S\|_2^2}\}+ O_P\Big(\frac{\sigma}{\sqrt{n}}\vee \frac{\sigma s_0\log p}{n}\Big)\\
& = \sqrt{\sum_{i\in\supp(\theta^*)}\E\big\{\big[\eta(\theta^*_i+n^{-1/2}\tilde{Z}_i;\lambda)-\theta^*_i\big]^2\big\}}+ O_P\Big(\frac{\sigma}{\sqrt{n}}\vee \frac{\sigma s_0\log p}{n}\Big)\, ,
\end{align}
where in the last equality expectation is with respect to $\tilde{Z_i}\sim \normal(0,\|\tx_i\|_2^2/n)$.
The proof is completed by using the fact that, with high probability, 
$\max_{i\in [p]}\big|\|\tx_i\|_2^2/n-1\big|\le C\sqrt{(\log p)/n}$, and bounding the resulting error.

%
%
\section{Proof of Theorem \ref{thm:TwoStep}}

Throughout this proof, we denote by $\prob_w$ and $\E_w$, the probability and the expectation with respect to the noise vector 
$w$ (conditional on $X$).
Let 
\begin{align}
\otau_i \equiv \sqrt{\frac{2\sigma^2 (\Omega\hSigma\Omega)_{ii} \log(p/s_0)}{n}} \, ,
\end{align}
and define the estimators $\hth^{(1)}$, $\oth^{(1)}$  by
\begin{align}
\oth^{(1)}_i & \equiv \eta\Big(\theta^*_i +\frac{1}{n}(\Omega X^{\sT} w)_i;\otau_i\Big)\, ,\\
\hth^{(1)}_i & \equiv \eta\Big(\theta^*_i +\frac{1}{n}(\Omega X^{\sT} w)_i;\tau_i\Big)\, .\label{eq:Hth1_def}
\end{align}
Throughout this section, $L_n$  denotes  a deterministic sequence with $L_n\to\infty$ arbitrarily slow as $n\to\infty$.
First we claim that,, $\|\oth^{(1)}\|_0, \|\hth^{(1)}\|_0\le s_0L_n$ with high probability for any such sequence $L_n$.
In other to prove this, recall that $S\equiv \supp(\theta^*)$,   and consider $i\not\in S$. Conditional on $X$, 
we have $(\Omega X^{\sT} w)_i/n\sim\normal(0,\sigma^2(\Omega\hSigma\Omega)_{ii}/n)$. Hence,
for $Z\sim\normal(0,1)$ independent of $X$, and by $W_n$ a chi-squared random variable with $n$ degrees of freedom, we get
\begin{align}
\prob(\hth^{(1)}_i\neq 0) &= \prob\big(|(\Omega X^{\sT} w)_i|\ge n\tau_i\big)\\
& = \prob\Big(|(\Omega\hSigma\Omega)^{1/2}_{ii}Z|\ge \sqrt{2\Omega_{ii}\log(p/s_0)}\Big)\\
& \le \prob\Big(|Z|\ge \sqrt{2(1+\delta)^{-1}\log(p/s_0)}\Big)+\prob\big((\Omega\hSigma\Omega)_{ii}\ge (1+\delta)\Omega_{ii}\big)\\
& \le \left(\frac{s_0}{p}\right)^{1-\delta} + \prob(W_n\ge n(1+\delta))\\
& \le \left(\frac{s_0}{p}\right)^{1-\delta} +e^{-cn\delta^2} \le \frac{Cs_0}{p}\, ,
\end{align}
where the last inequality follows with high probability by taking $\delta = C_0\sqrt{\log(p/s_0)/n}$, and using the assumption that $\log(p/s_0)^3/n\to 0$.
Hence, by Markov inequality $\|\hth^{(1)}\|_0\le s_0L_n$ with high probability. The claim follows by the same argument
for $\oth^{(1)}$.

We next claim that $\|\hth^{(2)}\|_0\le s_0L_n$ with high probability as well. 
Indeed, by definition
\begin{align}
\hth^{(2)}_i \equiv \eta\Big(\theta^*_i +\frac{1}{n}(\Omega X^{\sT} w)_i+\frac{1}{\sqrt{n}}R_i;\tau_i\Big)\, ,
\end{align}
Using the fact that $\|R\|_{\infty}\le C\sigma\sqrt{s_0(\log p)^2/n}$,
with high probability (cf. Theorem \ref{thm:main}) and proceeding along the same lines as above, we obtain
\begin{align}
\prob(\hth^{(2)}_i\neq 0) &\le \prob\left(|(\Omega\hSigma\Omega)^{1/2}_{ii}Z|\ge \sqrt{2\Omega_{ii}\log(p/s_0)}
-C\sqrt{\frac{s_0(\log p)^2}{n}}\right)+\prob\left(\|R\|_{\infty}> C\sigma\sqrt{\frac{s_0(\log p)^2}{n}}\right)\nonumber\\
&\le \left(\frac{s_0}{p}\right)^{1-\delta}\exp\left\{C\sqrt{\frac{s_0(\log p)^2\log (p/s_0)}{n}}\right\} +e^{-cn\delta^2} +o(1)\\
& \le \frac{Cs_0}{p}\, ,
\end{align}
where in the final step we used the assumption $s_0(\log p)^3/n \to 0$.
Hence, by Markov inequality, we have $\|\hth^{(2)}\|_0\le s_0L_n$, with high probability as claimed.

Therefore , with high probability,
\begin{align}
\big\|\hth^{(2)}-\hth^{(1)}\|_2&\le \frac{1}{\sqrt{n}}\|R\|_{\infty}\sqrt{\|\hth^{(2)}\|_0+\|\hth^{(1)}\|_0}\\
&\le \frac{C\sigma}{\sqrt{n}}\, \sqrt{\frac{s_0(\log p)^{2}}{n}}\, \sqrt{2s_0L_n} = \sqrt{2L_n}C \sigma\,\frac{s_0\log p}{n}\, . \label{eq:BoundHth1}
\end{align}
Analogously, we have 
\begin{align}
\big\|\oth^{(1)}-\hth^{(1)}\|_2&\le \sqrt{\|\oth^{(1)}\|_0+\|\hth^{(1)}\|_0}\,\cdot  \max_{i\in [p]}\big|\htau_i-\tau_i\big|\\
& \le  \sqrt{s_0L_n C_{\max}} \;\; \sqrt{\frac{2\sigma^2 \log(p/s_0)}{n}} \;\;\cdot
\max_{i\in [p]} \big| (\Omega\hSigma\Omega)_{ii} -\Omega_{ii}\big|\, ,
\end{align}
where we used the fact that $C_{\max}^{-1}\le \Omega_{ii}\le C_{\min}^{-1}$ is bounded uniformly and 
$|\sqrt{x}-\sqrt{y}|\le |x-y|/\sqrt{4c}$ for $x,y\ge c$.
Since $(\Omega\hSigma\Omega)_{ii}/\Omega_{ii}$ is distributed as $W_n/n$, for $W_n$ a chi-squared random variable with
$n$ degrees of freedom, and $\Omega_{ii}\le C^{-1}_{\min}$, we have 
$\max_{i\in [p]} \big| (\Omega\hSigma\Omega)_{ii} -\Omega_{ii}\big| \le C\sqrt{(\log p)/n}$. Substituting above,
we get
\begin{align}
\big\|\oth^{(1)}-\hth^{(1)}\|_2&\le \sqrt{2L_n C_{\max}}C \sigma \, \frac{\sqrt{s_0}\log p}{n}\, . \label{eq:BoundHth2}
\end{align}
Hence, using triangular inequality together with Equations.~(\ref{eq:BoundHth1}) and (\ref{eq:BoundHth2}), we 
obtain 
\begin{align}
\big\|\hth^{(2)}-\theta^*\|_2&\le                               \big\|\oth^{(1)}-\theta^*\|_2+C \sigma \sqrt{L_n}\, \frac{ s_0\log
                               p}{n}\\
&\le  \big\|\oth^{(1)}_S-\theta^*_S\|_2+\big\|\oth^{(1)}_{S^c}\|_2+C\sigma \sqrt{L_n}\,\frac{ s_0\log p}{n}\,, \label{eq:Triangular1}
\end{align}
for some constant $C>0$.

We are left with the task of bounding $\|\oth^{(1)}_S-\theta^*_S\|_2$ and  $\|\oth^{(1)}_{S^c}\|_2$.

$\bullet$ \emph{Bounding $\|\oth^{(1)}_S-\theta^*_S\|_2$.} Fixing $X\in \reals^{n\times p}$, we let  
$F(w;X) \equiv \|\oth^{(1)}_S-\theta^*_S\|_2$.  Letting
$\sigma_i^2 \equiv \sigma^2(\Omega\hSigma\Omega)_{ii}/n$, and denoting by $Z\sim\normal(0,1)$
a standard Gaussian random variable, we have
\begin{align}
\E_w \{F(w;X)^2\} &\stackrel{(a)}{=} 
\sum_{i\in S} \E_Z\Big\{\big[\eta\big(\theta^*_i+\sigma_i\, Z; \sigma_i\sqrt{2\log(p/s_0)}\big)-\theta^*_i\big]^2\Big\} \\
&\stackrel{(b)}{\le} 2\log (p/s_0) \sum_{i\in S}\sigma_i^2\\
&\stackrel{(c)}{\le}   \frac{2s_0\sigma^2}{n}\log (p/s_0) 
\left(\frac{1}{s_0}\sum_{i\in S}\Omega_{ii}\right)
\;\Big\{1+C\sqrt{\frac{\log p}{n}}\Big\}\equiv \baF^2\ .  \label{eq:baF}
\end{align}
Here, $(a)$ follows because $(\Omega X^{\sT}w)/n\sim\normal(0,\sigma_i^2)$, $(b)$ because  the soft-thresholding
risk is maximized for $\theta^*_i\to\infty$\cite{DJ94a,DoJo95,DMM09}, and $(c)$ because,
as remarked above, with high probability we have
$\max_{i\in [p]} \big| (\Omega\hSigma\Omega)_{ii} -\Omega_{ii}\big| \le C\sqrt{(\log p)/n}$.

Recall that $\baF^2$ denotes the upper bound on the right-hand side  of Eq.~(\ref{eq:baF}).
Let $\cG_0$ denote the set of matrices $X$ for which the bound 
$\E_w \{F(w;X)^2\}\le \baF^2$ holds. By above argument $\prob(\cG_0)\to 1$ as $n,p\to\infty$.
Now note that, since $\eta(\,\cdot\, ;\tau)$ is Lipschitz continuous (with Lipschitz constant equal to one),
and denoting by $\hth^{(1)}(w)$ the vector defined in Eq.~(\ref{eq:Hth1_def}) with noise vector $w$, we have
\begin{align}
\big|F(w;X) -F(w';X)\big|& \le \|\hth^{(1)}(w)_S-\hth^{(1)}(w')_S\|_2\\
& \le \frac{1}{n} C_{\min}^{-1}\|X_S\|_{2}\|w-w'\|_2\, .
\end{align}
By the Bai-Yin law, we have $\|X_S\|_2\le 2(\sqrt{n}+\sqrt{s_0})\le 4\sqrt{n}$ with high probability (since $s_0\le n$). 
Define, $\cG = \cG_0\cap \{X\in\reals^{n\times p}:\; \|X_S\|_2\le 4\sqrt{n}\}$.
By Gaussian isoperimetry \cite{Ledoux}, we have, on $\cG$,    
$\prob_w\Big(F(w;X)\ge \E_w\{F(w;X)\}+t\Big)\le   \, e^{-c\, nt^2/\sigma^2}$. 
This implies $\E\{F(w;X)\} = \baF+O(\sigma/\sqrt{n})$.
Hence, with high probability,
\begin{align}
\big\|\oth^{(1)}_S-\theta^*_S\big\|_2\le  \baF + \frac{L_n\sigma }{\sqrt{n}}\, .\label{eq:Triangular2}
\end{align}

$\bullet$ \emph{Bounding $\|\oth^{(1)}_{S^c}\|_2$.} As above, we let
$\sigma_i^2 \equiv \sigma^2(\Omega\hSigma\Omega)_{ii}/n$. Denoting by $Z\sim\normal(0,1)$
a standard Gaussian random variable, we write
\begin{align}
\E_w \{\|\oth^{(1)}_{S^c}\|_2^2\} 
&=
\sum_{i\in S^c} \E_Z\Big\{\eta\big(\sigma_i\, Z; \sigma_i\sqrt{2\log(p/s_0)}\big)^2\Big\} \\
&\stackrel{(a)}{\le} \sum_{i\in S^c} \sigma_i^2 \E_Z\Big\{
\eta\big( Z; \sqrt{2\log(p/s_0)}\big)^2\Big\}\\
&\stackrel{(b)}{\le}  C\frac{s_0}{p}\sum_{i\in S^c} \sigma_i^2\\
& = C\frac{s_0\sigma^2}{np}\Tr(\Omega\hSigma\Omega)\, .
\end{align}
Here, $(a)$ follows because $\eta(c\, x;c\lambda) =c\, \eta( x;\lambda)$  and $(b)$ by a Gaussian integral calculation. 
As mentioned above, $(\Omega\hSigma\Omega)_{ii}/\Omega_{ii}$ is distributed as $W_n/n$ for $W_n$ a chi-squared
random variable with $n$ degrees of freedom. Tail bounds on chi-squared random variables, together
with the fact that $\Omega_{ii}\le C_{\min}^{-1}$ is bounded uniformly, imply that 
$\Tr(\Omega\hSigma\Omega)\le C p$, with high probability. Hence, with high probability with respect to the choice of $X$,
$\E_w \{\|\oth^{(1)}_{S^c}\|_2^2\}  \le   C s_0\sigma^2/n$ for some constant $C>0$. Hence, with high probability 
\begin{align}
\|\oth^{(1)}_{S^c}\|_2\le \sigma\sqrt{\frac{s_0 L_n}{n}}\, .\label{eq:Triangular3}
\end{align}
The proof is completed by putting together Equations~(\ref{eq:Triangular1}), (\ref{eq:Triangular2}), (\ref{eq:Triangular3})
and setting $L_n = \log p$.
%
%
\section{Proof of Theorem \ref{thm:RiskEst}}

Throughout this section, we let $\hth=\Lth$ denote the Lasso estimator.
Define $\hth^0$ by
\begin{align}
\hth^0 &\equiv   \eta\Big(\theta^*+\frac{1}{n} X^{\sT}w;\lambda\Big)\, , \label{eq:hth0bis}
\end{align}
where $\eta(\,\cdot\,;\lambda)$ is componentwise soft thresholding,
defined for scalars via $\eta(x;\lambda) \equiv (|x|-\lambda)_+\sign(x)$.
Further we denote  by $\prob_w$ and $\E_w$ probability and expectation with respect to $w$ (conditional on $X$).
Finally, let $\hS\equiv \supp(\hth)$, $\hS^0\equiv \supp(\hth^0)$ and
$\baS\equiv \hS\cup\hS^0$. 

Expanding the square in the definition of $\hRisk(y,X)$, we obtain
\begin{align}
\Risk(y,X,\theta^*) -\hRisk(y,X) & = \frac{2}{n}\, \<w,X(\hth-\theta^*)\>\\
& =  \frac{2}{n}\, \<w,X(\hth^0-\theta^*)\>+\frac{2}{n}\, \<w,X(\hth-\hth^0)\>\\
& \equiv \Delta_1(w,X,\theta^*)+\Delta_2(w,X,\theta^*)\, .\label{eq:Delta1Delta2}
\end{align}
We will separately study the error terms $\Delta_1$ and $\Delta_2$.

We start by considering a preliminary remark.
\begin{lemma}\label{lemma:GoodLast}
Let $X\in \reals^{n\times p}$ have iid entries $X_{ij}\sim\normal(0,1)$, and define
\begin{align}
\cG_1(M)\equiv\left\{X\in\reals^{n\times p}: \; \max_{i\in [p]} \Big|\frac{\|\tx_i\|_2^2}{n}-1\Big|\le M\sqrt{\frac{\log p}{n}},\;\;
\max_{i\neq j\in [p]}\Big|\frac{\<\tx_i,\tx_j\>}{\|\tx_i\|_2\|\tx_j\|_2}\Big|\le  M\sqrt{\frac{\log p}{n}}\right\}\, .
\end{align}
Then, for $M$ a large enough constant, we have $\prob(X\in\cG_1(M))\ge 1-p^{-10}$.  

Further, under the assumptions of Theorem \ref{thm:RiskEst}, we have 
$\prob(\baS\subseteq S)\ge 1-p^{-3}$.
\end{lemma}
\begin{proof}
The lower bound on $\prob(X\in\cG_1(M))$ is standard, and follows from union bound along with tail bounds on chi-squared 
random variables. 

As for the lower bound on $\prob(\baS\subseteq S)$, using  the definition (\ref{eq:hth0bis}) we get that
\begin{align}
\prob(\hS^0\not\subseteq S)& \le\prob(\hS^0\not\subseteq S;\; X\in\cG_1(M)) + \prob(X\not\in\cG_1(M))\\  
&\le \sum_{i\in S^c}\prob\Big(\big|\frac{1}{n}(X^{\sT}w)_i\big|\ge \lambda;\; X\in\cG_1(M)\Big) +p^{-10}\\
&\le p\, \prob\Big(\frac{1.1\sigma}{\sqrt{n}} |Z|\ge \lambda\Big) +p^{-10}
\, .
\end{align} 
where, in the last expression $Z\sim\normal(0,1)$, and we used $\max_{i\in [p]}\|\tx_i\|_2\le 1.1$. The claim then 
follows by a direct calculation.

In order to bound $\prob(\hS\not\subseteq S)$ note that, by definition,
\begin{align}
\hth & =\eta\Big(\theta^*+\frac{1}{n}X^{\sT}w +\frac{1}{\sqrt{n}} R;\lambda\Big) \, .
\end{align}
The proof follows the same lines as above noting that, by Theorem \ref{thm:main}, $\|R\|_{\infty}/\sqrt{n}\le \lambda/100$
with high probability.
\end{proof}

\begin{lemma}\label{lemma:Delta2}
Under the assumptions of Theorem \ref{thm:RiskEst}, there exists a
constant $C$ such that, with high probability
\begin{align}
|\Delta_2|&\le \frac{Cs_0\sigma^2}{n}\, \sqrt{\frac{s_0(\log p)^3}{n}} \, .
\end{align}
\end{lemma}
\begin{proof}
We have
\begin{align}
|\Delta_2|&\le\frac{2}{n}\|(X^{\sT}w)_{\baS}\|_2\|\hth-\hth^0\|_2\\
& \le \frac{2}{n} \sqrt{ |\baS| } \|X^{\sT}w\|_{\infty} \|\hth-\hth^0\|_2\\
&\le \frac{2}{n}\sqrt{s_0}\cdot  2\sigma\sqrt{n\log p}\,\cdot \tC  \sqrt{\frac{s_0}{n}}\,\|R\|_{\infty}\, ,
\end{align}
where the last inequality follows from Lemma \ref{lemma:GoodLast} along with the bound (\ref{eq:BoundTT0}), for $\Sigma = \id$.
Using Theorem \ref{thm:main} we obtain the claim.
\end{proof}

Next consider term $\Delta_1$ in the decomposition (\ref{eq:Delta1Delta2}). 
We first compute its expectation with respect to the noise vector $w$.
\begin{lemma}\label{lemma:ExpDelta1}
Assume $X$ to have i.i.d. rows $x_i\sim \normal(0,\Sigma)$. Then we have, with high probability with respect to the 
choice of $X$,
\begin{align}
\left|\E_w\{\Delta_1\} - \frac{2\sigma^2}{n}\E_w\{\|\hth^0\|_0\}\right|\le \Big(C\sigma^2 \sqrt{\frac{\log p}{n^3}}\Big)\; \E_w\{\|\htheta^0\|_0 \}\, .
\end{align}
\end{lemma}
\begin{proof}
Using Stein's lemma, we get
\begin{align}
\E_w\{\Delta_1\} &= \frac{2}{n}\sum_{i=1}^n\sum_{j=1}^pX_{ij}\E_w\{w_i(\hth^0-\theta^*)_j\}\\
& = \frac{2\sigma^2}{n}\sum_{i=1}^n\sum_{j=1}^pX_{ij}\E_w\Big\{\frac{\partial\hth_j^0}{\partial w_i}\,\Big\}\, .\label{eq:SteinD1}
\end{align}
By differentiating the KKT conditions that follow from the definition of $\eta_{\Sigma}$, cf. Eq.~(\ref{eq:EtaSigmaDef}),
  we get that for $y=\eta_\Sigma(z)$, the following holds true
\begin{align}
\frac{\partial y_j}{\partial z_k} = \ind(y_j\neq 0)\big[(\Sigma_{TT})^{-1}\Sigma_{T,\cdot}\big]_{jk}\, ,
\end{align}
where $T =\supp(y)$. Recall that $\hth^0 = \eta_\Sigma(z)$ with $z = \tth+n^{-1}\Omega X^\sT w$ and $\hS^0=\supp(\hth^0)$. Therefore,
\begin{align}
\frac{\partial \hth^0_j}{\partial w_i} &= \sum_{k'=1}^p \frac{\partial \hth^0_j}{\partial z_{k'}} \frac{\partial z_{k'}}{\partial w_i}
=\sum_{k'=1}^p\frac{1}{n} \ind(\hth^0_j \neq 0) \big[(\Sigma_{\hS^0\hS^0})^{-1}\Sigma_{\hS^0,\cdot}\big]_{jk'}(\Omega X^\sT)_{k'i}\\
&=\frac{1}{n} \ind(\hth^0_j \neq 0) \sum_{k=1}^p \Big(\sum_{k'=1}^p \big[(\Sigma_{\hS^0\hS^0})^{-1}\Sigma_{\hS^0,\cdot}\big]_{jk'} \Omega_{k'k} \Big) X_{ik}\\
&=\frac{1}{n}\, \ind(\hth^0_j\neq 0)\, \sum_{k\in \hS^0} (\Sigma_{\hS^0\hS^0})^{-1}_{jk}X_{ik}\, .
\end{align}
Substituting in Eq.~(\ref{eq:SteinD1}), after some manipulations we get
\begin{align}
\E_w\{\Delta_1\} &= \frac{2\sigma^2}{n}\,\E_w\big\{\Tr\big((\Sigma_{\hS^0\hS^0})^{-1}\hSigma_{\hS^0,\hS^0}\big)\big\}\, .
\end{align}
Using~\cite[Lemma 6.2]{javanmard2014confidence}, we have $|\Sigma^{-1}\hSigma - \id|_\infty \le C\sqrt{(\log p)/n}$, with high probability. 
Hence, 
\begin{align}
\E_w\{\Tr\big((\Sigma_{\hS^0\hS^0})^{-1}\hSigma_{\hS^0,\hS^0}\big) - |\hS^0|\}\le C\sqrt{\frac{\log p}{n}} \E_w\{|\hS^0|\}\,.
\end{align} 
The claim follows.
\end{proof}

\begin{lemma}\label{lemma:ConcentrationDelta1}
Under the assumptions of Theorem \ref{thm:RiskEst}, the following holds 
\begin{align}
\prob\left(\Big|\Delta_1-\E_w\Delta_1\Big|\ge  \frac{t\sigma^2}{\sqrt{n}}\right)\le 2\, e^{-ct^2} +o_n(1) \, .
\end{align}
\end{lemma}
\begin{proof}
Define the event 
\begin{align}
\cG_2(M) \equiv \cG_1(M)\cap \Big\{ X\in\reals^{n\times p}:\; \lambda_{\max}(\hSigma_{S,S}) \le 2\,\Big\}\, .
\end{align}
Using Lemma \ref{lemma:GoodLast}, together with standard tail bounds on the singular values of
Wishart matrices \cite{Guionnet}, we get  $\prob(X\in\cG_2(M))\ge 1-p^{-5}-e^{-cn}$.

Define the set 
\begin{align}
\cC\equiv\Big\{w\in\reals^n:\; \frac{1}{n}\big\|X^{\sT} w\big\|_{\infty}\le \lambda; \;\; \|w\|_2^2\le 2n\sigma^2\Big\}\, .
\end{align}
By a union bound argument, it is immediate to see that, for any $X\in\cG_2(M)$,
$\prob(w\not\in\cC)\le p^{-6}+ e^{-cn}$. Further note the following: 
\begin{enumerate}
\item $\cC$ is convex.
\item  For  $w\in \cC$, we have $\hS^0\subseteq S$.
\item  As a consequence, for $w\in\cC$, 
\begin{align}
\|\hth-\theta^*\|_2^2&\le s_0 \|\hth_S-\theta_S^*\|_{\infty}^2\nonumber\\
& \le s_0 \Big(\lambda + \frac{1}{n}\|X^{\sT}w\|_{\infty}\Big)^2\\
&\le 4s_0\lambda^2\, .\nonumber
\end{align}
\end{enumerate}
In order to prove the lemma, we will use Gaussian concentration \cite{Ledoux}, by proving that 
$w\mapsto \Delta_1(w,X,\theta^*)$ is Lipschitz continuous on $\cC$. 
We have
\begin{align}
\frac{\partial \Delta_1}{\partial w_i} &=  \frac{2}{n}\<x_i,(\hth^0-\theta^*)\> + \frac{2}{n^2}
\sum_{j\in \hS^0}(X^Tw)_j X_{ij}\\
& = \frac{2}{n}\big(X(\hth^0-\theta^*)\big)_i+ \frac{2}{n^2}(X\proj_{\hS^0}X^{\sT}w)_i \, ,
\end{align}
where $\proj_{\hS^0}\in\reals^{p\times p}$ is the projector onto the indices in $\hS^0$. 
Namely,  $(\proj_{\hS^0})_{ij} = 0$ if $i\neq j$, and $(\proj_{\hS^0})_{ii} = \ind(i\in \hS^0)$.
Hence 
\begin{align}
\big\|\nabla \Delta_1\big\|_2^2 &\le \frac{8}{n^2}\<(\hth^0-\theta^*),X^{\sT}X (\hth^0-\theta^*)\>+
\frac{8}{n^4}\big\|X\proj_{\hS^0}X^{\sT}w\big\|_2^2\\
&\le \frac{8}{n}\<(\hth^0-\theta^*)_S,\hSigma_{SS} (\hth^0-\theta^*)_S\>+ \frac{8}{n^4}\big\|X\proj_{\hS^0}X^{\sT}w\big\|_2^2\\
&\le \frac{8}{n}\lambda_{\max}(\hSigma_{SS})\|\hth^0-\theta^*\|_2^2 + 
\frac{8}{n^4}\big\|X\proj_{\hS^0}X^{\sT}\big\|_2^2\|w\|_2^2\, .
\end{align}
Next note that 
\begin{align}
\frac{1}{n}\big\|X\proj_{\hS^0}X^{\sT}\big\|_2&= \frac{1}{n}\big\|X\proj_{\hS^0}\big\|^2_2\\
&=  \frac{1}{n}\big\|\proj_{\hS_0}X^{\sT}X\proj_{\hS^0}\big\|_2\\
&= \lambda_{\rm max}(\hSigma_{\hS^0,\hS^0}) \le \lambda_{\rm max}(\hSigma_{S,S})\, . 
\end{align}
Substituting above, and using $X\in\cG_2(M)$, we get
\begin{align}
\big\|\nabla_w \Delta_1\big\|_2^2 &\le \frac{8}{n}\lambda_{\max}(\hSigma_{SS})
\Big\{\|\hth^0-\theta^*\|_2^2 + \frac{1}{n}\lambda_{\max}(\hSigma_{SS})\,\|w\|_2^2\Big\}\\
& \le \frac{16}{n}\big(4s_0\lambda^2+2\sigma^2\big)\\
& \le \frac{16}{n}\Big(\frac{4Cs_0\sigma^2\log p}{n}+2\sigma^2\Big)\\
&\le \frac{C\sigma^2}{n}\, .
\end{align}
Hence, using Gaussian concentration \cite{Ledoux} (applied to the Lipschitz extension of $\Delta_1$ from
$w\in\cC$ to $w\not\in\cC$), we get
\begin{align}
\prob_w\Big(\big|\Delta_1-{\rm Med}_w(\Delta_1)\big|\ge t\Big) &\le  
\prob_w\Big(\big|\Delta_1-{\rm Med}_w(\Delta_1)\big|\ge t;\; w\in \cC\Big) +\prob_w\big(w\not\in \cC\big)\\
& \le  2 e^{-nt^2/C\sigma^4} +\prob_w\big(w\not\in \cC\big)\, ,
\end{align}
where ${\rm Med}_w(\,\cdot\,)$ denotes the median w.r.t the measure $\prob_w$.
The claim follows by bounding 
$|{\rm Med}_w(\Delta_1)-\E_w\{\Delta_1\}|$ in the standard way, and using the fact that
$\prob(X\not\in\cG_2(M))$, $\prob(w\not\in \cC)\to 0$.%
\end{proof}

\begin{lemma}\label{lemma:Cheby}
Fix $X\in\cG_1(M)$, and let $L_n$ be any sequence with $L_n\to\infty$ as $n\to\infty$. Then,
we have
\begin{align}
\E_w\{\|\hth^0\|_0\} & \le s_0+1\, ,\label{eq:Esupp0}\\
\prob_w\Big(\big|\|\hth^0\|_0-\E_w\{\|\hth^0\|_0\} &\big|\ge \frac{L_ns_0(\log p)^{1/4}}{n^{1/4}}\Big)\le \frac{M}{L_n^{2}}\, . \label{eq:ChebSupp0}
\end{align}
\end{lemma}
\begin{proof}
By Lemma \ref{lemma:GoodLast}, $\prob(\hth^0_{S^c}=0)\ge 1-p^{-3}$.
We thus get $\E_w\|\hth^0\|_0\le \E_w\|\hth^0_{S^c}\|_0 +s_0 \le p\cdot p^{-3}+s_0 \le 1+s_0$. 
Since $\prob\{\hth^0_{S^c}=0\}\ge 1-p^{-3}$, in order to prove Eq.~(\ref{eq:ChebSupp0}), it is sufficient to develop a tail bound on 
$|\|\hth^0_S\|_0-\E_w\{\|\hth^0\|_0\}$, which we do via Chebyshev inequality. Letting
$T_i \equiv \ind(|\theta^*_i+n^{-1}(X^{\sT}w)_i|>\lambda)$, we have $\|\hth^0_S\|_0=\sum_{i\in S}T_i$, whence
the variance of $\|\hth^0_S\|_0$ is given by
\begin{align}
\Var_w(\|\hth_S\|_0) &= \sum_{i,j\in S} \Cov_w(T_i;T_j)\\
& \stackrel{(a)}{\le} \sum_{i,j\in S} \frac{\Cov_w((X^{\sT}w)_i;(X^{\sT}w)_j)}{\sqrt{\Var_w((X^{\sT}w)_i)\Var_w((X^{\sT}w)_j)}}\, \cdot\sqrt{\Var(T_i)\Var(T_j)}\\
& =  \sum_{i,j\in S}  \frac{\<\tx_i,\tx_j\>}{\|\tx_i\|_2\|\tx_j\|_2}\, \sqrt{\Var(T_i)\Var(T_j)}\\
& \le M\sqrt{\frac{\log p}{n}}\left(\sum_{i\in S}\sqrt{\Var(T_i)}\right)^2\\
& \le M s_0^2\sqrt{\frac{\log p}{n}}\, .\label{eq:BoundVariance}
\end{align}
Here $(a)$ follows because, for jointly Gaussian random variables $Z_1$, $Z_2$, the correlation coefficient
between $f(Z_1)$, $g(Z_2)$ is maximized by linear functions $f$, $g$. 

The claim (\ref{eq:ChebSupp0}) follows from Chebyshev inequality, using  Eq.~(\ref{eq:BoundVariance}).
\end{proof}

\begin{lemma}\label{lemma:HthHth0}
Let $L_n$ be any sequence with $L_n\to\infty$. Then, 
under the assumptions of Theorem \ref{thm:RiskEst}, we have, with high probability,
\begin{align}
\Big|\|\hth\|_0-\|\hth^0\|_0\Big|\le L_ns_0\sqrt{\frac{s_0(\log p)^2}{n}}\, .
\end{align}
\end{lemma}
\begin{proof}
Recall that, by definition
\begin{align}
\hth & =\eta\Big(\theta^*+\frac{1}{n}X^{\sT}w +\frac{1}{\sqrt{n}} R;\lambda\Big) \, ,\\
\hth^0 & =\eta\Big(\theta^*+\frac{1}{n}X^{\sT}w;\lambda\Big) \, .
\end{align}
Let $\eps_n = C\sqrt{s_0(\log p)^2/n}$ for $C$ a sufficiently large constant, and define the event
\begin{align}
\cG_0\equiv \Big\{\|R\|_{\infty}\le \sigma\eps_n; \;\; \oS\subseteq S\Big\}\, .
\end{align}
By Theorem \ref{thm:main} and Lemma \ref{lemma:GoodLast}, $\prob(\cG_0)\to 1$ as $n,p\to\infty$.
On this event, we have
\begin{align}
\Big|\|\hth\|_0-\|\hth^0\|_0\Big|&\le \sum_{i\in S}
\ind\left(\Big|\theta_i^*+\frac{1}{n}(X^{\sT}w)_i\Big|\in 
\Big[\lambda-\frac{1}{\sqrt{n}}\|R\|_{\infty},\lambda+\frac{1}{\sqrt{n}}\|R\|_{\infty}\Big] \right)\\
&\le  \sum_{i\in S}
\ind\left(\Big|\theta_i^*+\frac{1}{n}(X^{\sT}w)_i\Big|\in 
\Big[\lambda-\frac{\sigma\eps_n}{\sqrt{n}},\lambda+\frac{\sigma\eps_n}{\sqrt{n}}\Big] \right) \\
&\equiv \sum_{i\in S} W_i\, .
\end{align}
We then have, for any sequence $L_n\to\infty$, 
\begin{align}
\prob\Big(\Big|\|\hth\|_0-\|\hth^0\|_0\Big| \ge L_ns_0\eps_n\Big)\le 
\prob\Big( \sum_{i\in S} W_i\ge L_ns_0\eps_n; \cG_1(M) \Big) + \prob(\cG_0^c) +\prob(\cG_1(M)^c)\, .
\end{align}
Using Lemma \ref{lemma:GoodLast}, it is sufficient to show that the first term vanishes. 
This can be done by Markov inequality, bounding the expectation as follows
\begin{align}
\E\Big\{ \sum_{i\in S} W_i;\cG_1\Big\} &= \sum_{i\in S}\prob\left(\Big|\theta_i^*+\frac{1}{n}(X^{\sT}w)_i\Big|\in 
\Big[\lambda-\frac{\sigma\eps_n}{\sqrt{n}},\lambda+\frac{\sigma\eps_n}{\sqrt{n}}\Big] ; \cG_1(M)\right) \\
& \stackrel{(a)}{\le} 2\sum_{i\in S}\sup_{z\in\reals}\prob\left(\frac{\sigma\|\tx_i\|_2}{{n}} Z\in 
\Big[z-\frac{\sigma\eps_n}{\sqrt{n}},z+\frac{\sigma\eps_n}{\sqrt{n}}\Big] ; \cG_1(M)\right) \\
& \stackrel{(b)}{\le} 2s_0\sup_{z\in\reals}\prob\left(Z\in [z-2\eps_n,z+2\eps_n] \right) \\
&\le C\, s_0\eps_n\, ,
\end{align}
where $(a)$ holds for $Z\sim\normal(0,1)$, and $X\in\cG_1(M)$ was used
in $(b)$.
\end{proof}

\begin{proof}[Proof of Theorem \ref{thm:RiskEst}]
First notice that
\begin{align}
 \Big|\Delta_1-\frac{2\sigma^2}{n}\|\hth\|_0\Big|&\le \Big|\Delta_1-\frac{2\sigma^2}{n}\|\hth^0\|_0\Big|+
\frac{2\sigma^2}{n}\Big|\|\hth\|_0-\|\hth^0\|_0\Big|\\
&\le \Big|\E_w\Delta_1-\frac{2\sigma^2}{n}\E_w\|\hth^0\|_0\Big|+
\Big|\Delta_1-\E_w\Delta_1\Big|\\
& \phantom{AAA}+
\frac{2\sigma^2}{n}\Big|\|\hth^0\|_0-\E_w\|\hth^0\|_0\Big|+
\frac{2\sigma^2}{n}\Big|\|\hth\|_0-\|\hth^0\|_0\Big|\nonumber\\
&\stackrel{(a)}{\le} 2Cs_0\sigma^2 \sqrt{\frac{\log p}{n^3}}+ \frac{2t\sigma^2}{\sqrt{n}} +
+2L_ns_0\sigma^2 \frac{(\log p)^{1/4}}{n^{5/4}}+
2L_n\sigma^2 \Big(\frac{s_0}{n}\Big)^{3/2} \log p\\
& \le \frac{2t\sigma^2}{\sqrt{n}} +\frac{6L_ns_0\sigma^2}{n}\left(\Big(\frac{\log p}{n}\Big)^{1/4}\vee 
\Big(\frac{s_0(\log p)^2}{n}\Big)^{1/2}\right)  \, ,
\end{align}
where the  inequality $(a)$ holds probability 
larger than $1-o_n(1)-2e^{-ct^2}$ by lemmas \ref{lemma:ExpDelta1}, \ref{lemma:ConcentrationDelta1},
\ref{lemma:Cheby}, \ref{lemma:HthHth0}  for any sequence $L_n\to \infty$ as $n\to\infty$.
We let
\begin{align}
\eps_n\equiv 6 L_n \left(\Big(\frac{\log p}{n}\Big)^{1/4}\vee 
\Big(\frac{s_0(\log p)^2}{n}\Big)^{1/2}\right) \, .
\end{align}

Using the decomposition (\ref{eq:Delta1Delta2}), we have
\begin{align}
\Big|\Risk(y,X,\theta^*) -\hRisk(y,X)  -\frac{2\sigma^2}{n}\|\hth\|_0\Big|& \le \Big|\Delta_1-
                                       \frac{2\sigma^2}{n}\|\hth\|_0\Big|+|\Delta_2|\\
&\le \frac{2t\sigma^2}{\sqrt{n}}+\frac{\eps_ns_0\sigma^2}{n} + \frac{Cs_0\sigma^2}{n}\, \sqrt{\frac{s_0(\log p)^2}{n}} \\
&\le \frac{2t\sigma^2}{\sqrt{n}}+\frac{2\eps_ns_0\sigma^2}{n} \, ,
\end{align}
where the last inequality holds for all $n$ large enough.

By choosing $L_n$ to be a sequence with slow enough growth rate, e.g. $L_n = (\frac{n}{s_0 (\log p)^2})^{1/4}$, we have $\eps_n \to 0$.
This completes the proof for Gaussian designs.
\end{proof}

\end{document}